\documentclass[a4paper]{amsart}
\pdfoutput=1
\usepackage{a4,exscale,amsfonts,amssymb,amsmath,amscd,nicefrac,tikz,tikz-3dplot,graphicx,color}
\usepackage[english]{babel}
\usepackage{url}
\usepackage{tabularx}
\usepackage{tabulary}
\usepackage{comment}
\allowdisplaybreaks
\newtheorem{lem}{Lemma}[section]
\newtheorem{theo}[lem]{Theorem}
\newtheorem{cor}[lem]{Corollary}
\newtheorem{rem}[lem]{Remark}
\newtheorem{ass}[lem]{Assumption}

\newcommand{\ol}{\overline}
\def\To{\longrightarrow}
\def\ot{\leftarrow}
\def\wto{\rightharpoonup}
\def\emb{\hookrightarrow}
\def\impl{\Rightarrow}
\def\qimpl{\quad\impl\quad}

\newcommand{\qtext}[1]{\quad\text{#1}\quad}

\def\reals{\mathbb{R}}
\def\om{\Omega}
\def\ga{\Gamma}
\def\gat{\ga_{\!\tau}}
\def\gan{\ga_{\!\nu}}
\def\calH{\mathcal{H}}
\def\calA{\mathcal{A}}
\def\sfL{\mathsf{L}}
\def\sfH{\mathsf{H}}
\def\sfC{\mathsf{C}}
\newcommand{\Harm}[2]{\calH^{#1}_{\mathsf{#2}}}
\renewcommand{\L}[1]{\sfL^{#1}}
\renewcommand{\H}[2]{\sfH^{#1}_{#2}}
\newcommand{\C}[2]{\sfC^{#1}_{#2}}
\newcommand{\eps}{\varepsilon}
\DeclareMathOperator{\p}{\partial}
\DeclareMathOperator{\A}{A}
\DeclareMathOperator{\cA}{\calA}
\DeclareMathOperator{\M}{M}
\DeclareMathOperator{\opgrad}{grad}
\DeclareMathOperator{\oprot}{curl}

\DeclareMathOperator{\opdiv}{div}
\DeclareMathOperator{\opd}{d}
\DeclareMathOperator{\opde}{\delta}
\newcommand{\grad}[1]{\opgrad_{#1}}
\newcommand{\rot}[1]{\oprot_{#1}}
\newcommand{\rotv}[1]{\vec\oprot_{#1}}
\renewcommand{\div}[1]{\opdiv_{#1}}
\newcommand{\ed}[1]{\opd_{#1}}
\newcommand{\cd}[1]{\opde_{#1}}
\newcommand{\norm}[1]{|#1|}
\newcommand{\bnorm}[1]{\big|#1\big|}
\newcommand{\scp}[2]{\langle#1,#2\rangle}
\newcommand{\bscp}[2]{\big\langle#1,#2\big\rangle}

\DeclareMathOperator{\stiff}{K}
\DeclareMathOperator{\mass}{M}
\def\pone{{\sf P1}}
\def\rt{{\sf RT}}
\def\ned{{\sf N}}

\newcommand{\bk}{k}

\title[Poincar\'e-Friedrichs Type Constants]
{Poincar\'e-Friedrichs Type Constants for Operators Involving\\ 
grad, curl, and div: Theory and Numerical Experiments}
\author{Dirk Pauly}
\address{Fakult\"at f\"ur Mathematik,
Universit\"at Duisburg-Essen, Campus Essen, Germany}
\email[Dirk Pauly]{dirk.pauly@uni-due.de}
\author{Jan Valdman}
\address{Institute of mathematics, Faculty of Science, University of South Bohemia, \v{C}esk\'e Bud\v{e}jovice 
\&  Department of Decision-Making Theory, Institute of Information Theory and Automation, Prague, Czech Republic}
\email[Jan Valdman]{jvaldman@prf.jcu.cz}
\date{\today; {\it Corresponding Author}: Dirk Pauly}
\thanks{}

\begin{document}


\def\titlerepude{\sf Poincar\'e-Friedrichs Type Constants for Operators Involving\\ 
grad, curl, and div: Theory and Numerical Experiments}
\def\authorrepude{Dirk Pauly and Jan Valdman}
\def\daterepdue{\today}
\def\reportudemathyesno{no}
\def\reportudemathnumber{SM-UDE-820}
\def\reportudemathyear{2019}
\def\reportudematheingang{\daterepdue}
\newcommand{\preprintudemath}[5]{
\thispagestyle{empty}
\begin{center}\normalsize SCHRIFTENREIHE DER FAKULT\"AT F\"UR MATHEMATIK\end{center}
\vspace*{5mm}
\begin{center}#1\end{center}
\vspace*{5mm}
\begin{center}by\end{center}
\vspace*{0mm}
\begin{center}#2\end{center}
\vspace*{5mm}
\normalsize 
\begin{center}#3\hspace{69mm}#4\end{center}
\newpage
\thispagestyle{empty}
\vspace*{210mm}
Received: #5
\newpage
\addtocounter{page}{-2}
\normalsize
}
\ifthenelse{\equal{\reportudemathyesno}{yes}}
{\preprintudemath{\titlerepude}{\authorrepude}{\reportudemathnumber}{\reportudemathyear}{\reportudematheingang}}
{}

\begin{abstract}
We give some theoretical as well as computational results
on Laplace and Maxwell constants, i.e.,
on the smallest constants $c_{n}>0$ arising in estimates of the form
$$\norm{u}_{\L{2}(\om)}\leq c_{0}\norm{\grad{}u}_{\L{2}(\om)},\qquad
\norm{E}_{\L{2}(\om)}\leq c_{1}\norm{\rot{}E}_{\L{2}(\om)},\qquad
\norm{H}_{\L{2}(\om)}\leq c_{2}\norm{\div{}H}_{\L{2}(\om)}.$$
Besides the classical de Rham complex we investigate 
the complex of elasticity and the complex related to the biharmonic equation 
and general relativity as well
using the general functional analytical concept of Hilbert complexes.
We consider mixed boundary conditions and bounded Lipschitz domains of arbitrary topology.
Our numerical aspects are presented by examples for the de Rham complex
in 2D and 3D which not only confirm our theoretical findings
but also indicate some interesting conjectures. 
\end{abstract}


\maketitle
\tableofcontents


\section{Introduction}

We present some theoretical results as well as some computations on Laplace and Maxwell constants 
for bounded Lipschitz domains $\om$ with mixed boundary conditions
defined on boundary parts $\gat$ and $\gan$ of the boundary $\ga$.
While a lot of our theoretical findings hold for domains $\om$ in arbitrary dimensions,
we restrict our numerical experiments to the 2D and 3D cases.
Moreover, we verify various theoretical results established in the last few years
in \cite{paulymaxconst0,paulymaxconst1,paulymaxconst2,paulymaxconst3}. 
There is a recent interest in these eigenvalues, see, e.g., 
\cite{boffigastaldi2019,boffigastaldirodriguezsebestova2019,costabeldauge2019a}
and related contributions 
\cite{boffikikuchischoeberl2006,buffahoustonperugia2007,costabelremmaxlip,costabelcoercbilinMax,
costabeldaugemaxwelllameeigenvaluespolyhedra,
filonovdirneulapeigen,payneweinbergerpoincareconvex},
but little results for mixed boundary conditions
are known in the literature, except for, e.g., \cite{paulymaxconst2,paulymaxconst3}.
In 3D these constants are the best possible real numbers
$c_{0,\gat}$, $c_{1,\gat}$, $c_{2,\gat}>0$ in the estimates
\begin{align*}
\forall\,u&\in D(\grad{\gat})\cap R(\div{\gan})
&
\norm{u}_{\L{2}(\om)}&\leq c_{0,\gat}\norm{\grad{}u}_{\L{2}(\om)},\\
\forall\,E&\in D(\rot{\gat})\cap R(\rot{\gan})
&
\norm{E}_{\L{2}(\om)}&\leq c_{1,\gat}\norm{\rot{}E}_{\L{2}(\om)},\\
\forall\,H&\in D(\div{\gat})\cap R(\grad{\gan})
&
\norm{H}_{\L{2}(\om)}&\leq c_{2,\gat}\norm{\div{}H}_{\L{2}(\om)},
\end{align*}
which are often called Poincar\'e-Friedrichs type constants,
cf. Section \ref{seclapmax3D} for notations.
More precisely, we have
\begin{align*}
\frac{1}{c_{0,\gat}}
&=\inf_{0\neq u\in D(\grad{\gat})\cap R(\div{\gan})}\frac{\norm{\grad{}u}_{\L{2}(\om)}}{\norm{u}_{\L{2}(\om)}}
&
&\text{(Friedrichs/Poincar\'e constants)},\\
\frac{1}{c_{1,\gat}}
&=\inf_{0\neq E\in D(\rot{\gat})\cap R(\rot{\gan})}\frac{\norm{\rot{}E}_{\L{2}(\om)}}{\norm{E}_{\L{2}(\om)}}
&
&\text{(rotation/Maxwell constants)},\\
\frac{1}{c_{2,\gat}}
&=\inf_{0\neq H\in D(\div{\gat})\cap R(\grad{\gan})}\frac{\norm{\div{}H}_{\L{2}(\om)}}{\norm{H}_{\L{2}(\om)}}
&
&\text{(divergence constants)}.
\end{align*}
We also point out the strong connection 
to the well known de Rham complex with mixed boundary conditions.

It turns out that in the \underline{3D} case, cf. Theorem \ref{stateoftheartconstest}, 
the estimates and equations
$$c_{0,\ga}\leq c_{0,\gat}=c_{2,\gan},\qquad
c_{1,\gat}=c_{1,\gan},\qquad
c_{0,\ga}\leq\min\big\{c_{0,\emptyset},\frac{\text{\rm diam}(\om)}{\pi}\big\}$$
always hold and that in convex domains we even have
$$c_{1,\ga}=c_{1,\emptyset}\leq c_{0,\emptyset}\leq\frac{\text{\rm diam}(\om)}{\pi}.$$
Here, 
$$c_{\mathsf{F}}:=c_{0,\ga},\qquad
c_{\mathsf{P}}:=c_{0,\emptyset}$$
are the classical Friedrichs and Poincar\'e constants,
respectively, and the constants 
$$c_{\mathsf{M,t}}:=c_{1,\ga}=c_{1,\emptyset}=:c_{\mathsf{M,n}}$$
are often called tangential (electric)
and normal (magnetic) Maxwell constants, respectively. 
All these constants relate to minimal positive eigenvalues
of certain Laplace and Maxwell operators. More precisely, 
$$\lambda_{0,\gat}=\frac{1}{c_{0,\gat}},\qquad 
\lambda_{1,\gat}=\frac{1}{c_{1,\gat}},\qquad 
\lambda_{2,\gat}=\frac{1}{c_{2,\gat}}$$ 
are the smallest positive eigenvalues of the first order matrix operators
$$\begin{bmatrix}0&-\div{\gan}\\\grad{\gat}&0\end{bmatrix},\qquad
\begin{bmatrix}0&\rot{\gan}\\\rot{\gat}&0\end{bmatrix},\qquad
\begin{bmatrix}0&-\grad{\gan}\\\div{\gat}&0\end{bmatrix},$$
respectively, and 
\begin{equation} 
\label{eigenvalues_to_compute}
\lambda_{0,\gat}^2,\qquad 
\lambda_{1,\gat}^2,\qquad 
\lambda_{2,\gat}^2
\end{equation}
are the smallest positive eigenvalues of the second order operators
\begin{equation}
\label{operators_to_discretize}
-\div{\gan}\grad{\gat},\qquad
\rot{\gan}\rot{\gat},\qquad
-\grad{\gan}\div{\gat},
\end{equation}
respectively. 
In \underline{2D}, cf. Corollary \ref{stateoftheartconstestcor2D}, we will see that 
$$c_{0,\ga}\leq c_{0,\gat}=c_{1,\gan}=c_{2,\gan},\qquad
c_{0,\ga}\leq\min\big\{c_{0,\emptyset},\frac{\text{\rm diam}(\om)}{\pi}\big\}.$$
Generally, in \underline{ND}, cf. Theorem \ref{stateoftheartconstesttheoND}, we have
$$c_{0,\ga}\leq c_{0,\gat}=c_{N-1,\gan},\qquad
c_{q,\gat}=c_{N-q-1,\gan},\qquad
c_{0,\ga}\leq\min\big\{c_{0,\emptyset},\frac{\text{\rm diam}(\om)}{\pi}\big\}$$
and in convex domains
$$c_{q,\ga}=c_{N-q-1,\emptyset}\leq c_{0,\emptyset}\leq\frac{\text{\rm diam}(\om)}{\pi}.$$
Here, $q=0,\dots,N-1$ and the differential operators $\grad{}$, $\rot{}$, and $\div{}$
are simply replaced by the exterior derivative $\ed{q}$
acting on the rank $q$ of the respective differential form.
So far, all findings are related to the ND de Rham complex.
We will present more examples and results for the 3D elasticity complex
as well as for the 3D biharmonic complex.

In a series of numerical tests we discretize the operators \eqref{operators_to_discretize} 
by the finite element method and compute upper bounds for the eigenvalues
\eqref{eigenvalues_to_compute} from generalized eigenvalue systems 
$$K u = \lambda^2 \, M u $$
with discretized stiffness and mass matrices $K$ and $M$, respectively. 
There are also recent interests in guaranteed lower bounds, cf.
\cite{zbMATH06342370,zbMATH06261586,zbMATH06296642}.
In a search for the smallest positive eigenvalue $\lambda^2$ we exploit a projection into the range of $K$ 
for smaller size problems or the nested iteration technique for large size problems. 
The latter theoretical results are confirmed by computations in 2D
for the unit square and the unit L-shape domain as well as in 3D
for the unit cube and the unit Fichera corner domain. 
Note that all three constants 
$c_{\ell,\gat}$, $\ell=1,2,3$,
grow proportionally to the ``radius'' of the domain. More precisely, e.g., it holds
$$c_{\ell}(r\cdot\om,r\cdot\gat)=r\cdot c_{\ell}(\om,\gat)=r\cdot c_{\ell,\gat},$$
where 
$$r\cdot\om:=\{r\cdot x\,:\,x\in\om\},\qquad
r\cdot\gat:=\{r\cdot x\,:\,x\in\gat\}$$
for some bounded domain $\om\subset\reals^{3}$
being star-shaped with respect to the origin.
Moreover, we performed some monotonicity tests (with respect to boundary conditions) 
which are just partially guaranteed by our theoretical findings. To our surprise 
we found (numerically) much stronger inequalities in \eqref{PVconjPFM}, 
see also the related Figure \ref{monoconstants3DucFc},
giving rise to some interesting conjectures.

\section{Theoretical Results}

We shall summarise some basic results from functional analysis 
and apply those to the classical operators of vector analysis.

\subsection{Functional Analysis ToolBox}

We start with collecting and citing some results from
\cite{paulyapostfirstordergen,paulydivcurl,paulyzulehnerbiharmonic,paulymaxconst2,paulymaxconst3}
about the so-called functional analysis toolbox (fa-toolbox).

\subsubsection{Preliminaries}

Let $\A:D(\A)\subset\H{}{0}\to\H{}{1}$ be a densely defined and closed linear operator
with domain of definition $D(\A)$ on two Hilbert spaces $\H{}{0}$ and $\H{}{1}$. 
Then the adjoint $\A^{*}:D(\A^{*})\subset\H{}{1}\to\H{}{0}$
is well defined and characterised by
$$\forall\,x\in D(\A)\quad
\forall\,y\in D(\A^{*})\qquad
\scp{\A x}{y}_{\H{}{1}}=\scp{x}{\A^{*} y}_{\H{}{0}}.$$
$\A$ and $\A^{*}$ are both densely defined and closed, but typically unbounded.
Often $(\A,\A^{*})$ is called a dual pair as $(\A^{*})^{*}=\ol{\A}=\A$.
The projection theorem shows
\begin{align}
\label{helm}
\H{}{0}=N(\A)\oplus_{\H{}{0}}\ol{R(\A^{*})},\quad
\H{}{1}=N(\A^{*})\oplus_{\H{}{1}}\ol{R(\A)},
\end{align}
often called Helmholtz/Hodge/Weyl decompositions,
where we introduce the notation $N$ for the kernel (or null space)
and $R$ for the range of a linear operator.
These orthogonal decompositions reduce the operators $\A$ and $\A^{*}$,
leading to the injective operators $\cA:=\A|_{\ol{R(\A^{*})}}$ and $\cA^{*}:=\A^{*}|_{\ol{R(\A)}}$, i.e.
\begin{align*}
\cA:D(\cA)\subset\ol{R(\A^{*})}&\to\ol{R(\A)},
&
D(\cA)&=D(\cA)\cap\ol{R(\A^{*})},\\
\cA^{*}:D(\cA^{*})\subset\ol{R(\A)}&\to\ol{R(\A^{*})},
&
D(\cA^{*})&=D(\cA^{*})\cap\ol{R(\A)}.
\end{align*}
Note that 
$$\ol{R(\A^{*})}=N(\A)^{\bot_{\H{}{0}}},\quad
\ol{R(\A)}=N(\A^{*})^{\bot_{\H{}{1}}}$$
and that $\cA$ and $\cA^{*}$ are indeed adjoint to each other, i.e.,
$(\cA,\cA^{*})$ is a dual pair as well. Then the inverse operators 
$$\cA^{-1}:R(\A)\to D(\cA),\quad
(\cA^{*})^{-1}:R(\A^{*})\to D(\cA^{*})$$
are well defined and bijective, but possibly unbounded.
Furthermore, by \eqref{helm} we have the refined Helmholtz type decompositions
\begin{align}
\label{helmdecoAcA}
D(\A)=N(\A)\oplus_{\H{}{0}}D(\cA),\quad
D(\A^{*})=N(\A^{*})\oplus_{\H{}{1}}D(\cA^{*})
\end{align}
and thus we obtain for the ranges
$$R(\A)=R(\cA),\quad
R(\A^{*})=R(\cA^{*}).$$

\subsubsection{Basic Results}

The following result is a well known and direct consequence of
the closed graph theorem and the closed range theorem.

\begin{lem}[fa-toolbox lemma 1]
\label{fatbl1}
The following assertions are equivalent:
\begin{itemize}
\item[\bf(i)] 
$\exists\,c_{\A}\,\,\in(0,\infty)$ \quad 
$\forall\,x\in D(\cA)$ \qquad\,
$\norm{x}_{\H{}{0}}\leq c_{\A}\norm{\A x}_{\H{}{1}}$
\item[\bf(i${}^{*}$)] 
$\exists\,c_{\A^{*}}\in(0,\infty)$ \quad 
$\forall\,y\in D(\cA^{*})$ \qquad
$\norm{y}_{\H{}{1}}\leq c_{\A^{*}}\norm{\A^{*} y}_{\H{}{0}}$
\item[\bf(ii)] 
$R(\A)=R(\cA)$ is closed in $\H{}{1}$.
\item[\bf(ii${}^{*}$)] 
$R(\A^{*})=R(\cA^{*})$ is closed in $\H{}{0}$.
\item[\bf(iii)] 
$\cA^{-1}:R(\A)\to D(\cA)$ is bounded by $c_{\A}$.
\item[\bf(iii${}^{*}$)] 
$(\cA^{*})^{-1}:R(\A^{*})\to D(\cA^{*})$ is bounded by $c_{\A^{*}}$.
\item[\bf(iv)] 
$\cA:D(\cA)\subset R(\A^{*})\to R(\A)$ is bijective with continuous inverse.
\item[\bf(iv${}^{*}$)] 
$\cA^{*}:D(\cA^{*})\subset R(\A)\to R(\A^{*})$ is bijective with continuous inverse.
\item[\bf(v)] 
$\cA:D(\cA)\to R(\A)$ is a topological isomorphism.
\item[\bf(v${}^{*}$)] 
$\cA^{*}:D(\cA^{*})\to R(\A^{*})$ is a topological isomorphism.
\end{itemize}
\end{lem}

The latter inequalities will be called Poincar\'e-Friedrichs type estimates.
Note that in (iv) and (iv${}^{*}$) we consider $\cA$ and $\cA^{*}$
as unbounded linear operators, whereas in (v) and (v${}^{*}$) 
we consider $\cA$ and $\cA^{*}$ as bounded linear operators.

\begin{lem}[fa-toolbox lemma 2]
\label{fatbl2}
The following assertions are equivalent:
\begin{itemize}
\item[\bf(i)]
$D(\cA)\emb\H{}{0}$ is compact.
\item[\bf(i${}^{*}$)]
$D(\cA^{*})\emb\H{}{1}$ is compact.
\item[\bf(ii)]
$\cA^{-1}:R(\A)\to R(\A^{*})$ is compact.
\item[\bf(ii${}^{*}$)]
$(\cA^{*})^{-1}:R(\A^{*})\to R(\A)$ is compact.
\end{itemize}
Moreover: Each of these assumptions imply
the assertions of Lemma \ref{fatbl1} (and of Lemma \ref{fatbl2}).
\end{lem}

\begin{rem}[sufficient assumptions for the fa-toolbox]
\label{fatbr}
\mbox{}
\begin{itemize}
\item[\bf(i)]
If $R(\A)$ is closed, then the assertions of Lemma \ref{fatbl1} hold.
\item[\bf(ii)]
If $D(\cA)\emb\H{}{0}$ is compact,
then the assertions of Lemma \ref{fatbl1} and Lemma \ref{fatbl2} hold.
In particular, the Poincar\'e-Friedrichs type estimates hold,
all ranges are closed and the inverse operators are compact.
\end{itemize}
\end{rem}

\subsubsection{Constants, Spectra, and Eigenvalues}
\label{seccstgen}

Let us introduce the ``best'' constants $c_{\A}$, $c_{\A^{*}}$
by utilising the Rayleigh quotients 
\begin{align}
\label{lambdaAdef}
\frac{1}{c_{\A}}
:=\inf_{0\neq x\in D(\cA)}\frac{\norm{\A x}_{\H{}{1}}}{\norm{x}_{\H{}{0}}},\quad
\frac{1}{c_{\A^{*}}}
:=\inf_{0\neq y\in D(\cA^{*})}\frac{\norm{\A^{*} y}_{\H{}{0}}}{\norm{y}_{\H{}{1}}}.
\end{align}
Then $0<c_{\A},c_{\A^{*}}\leq\infty$ and we refer to $c_{\A}$ and $c_{\A^{*}}$
as Poincar\'e-Friedrichs type constants.
From now on, we assume that we always deal with these best constants.

\begin{lem}[constant lemma]
\label{lemconstants}
The Poincar\'e-Friedrichs type constants coincide, i.e., $c_{\A}=c_{\A^{*}}$.
\end{lem}

In the case that $R(\A)$ is closed, we shall denote
$$\lambda_{\A}:=\frac{1}{c_{\A}}=\frac{1}{c_{\A^{*}}}>0.$$

Let us emphasise that
\begin{align}
\label{saops}
\A^{*}\A,\quad
\A\A^{*},\quad
|\A|,\quad
|\A^{*}|,\quad
\begin{bmatrix}\A^{*}\A&0\\0&\A\A^{*}\end{bmatrix},\quad
\begin{bmatrix}\A\A^{*}&0\\0&\A^{*}\A\end{bmatrix},\quad
\begin{bmatrix}0&\A^{*}\\\A&0\end{bmatrix},\quad
\begin{bmatrix}0&\A\\\A^{*}&0\end{bmatrix}
\end{align}
are self-adjoint, see Appendix \ref{appproofs}, and have essentially 
- except of $0$ and taking square roots -
the same spectra contained in $\reals$.
Moreover, the first four operators are non-negative.
The same holds true for the reduced operators $\cA$ and $\cA^{*}$.
We will give more details in the next lemma.

\begin{lem}[constant and eigenvalue lemma]
\label{lemconstev}
Let $D(\cA)\emb\H{}{0}$ be compact. 
Then the operators in \eqref{saops} have pure and discrete point spectra
with no accumulation point in $\reals$. Moreover:
\begin{itemize}
\item[\bf(i)]
$\lambda_{\A}$ is the smallest positive eigenvalue of $\begin{bmatrix}0&\A^{*}\\\A&0\end{bmatrix}$
and of $\begin{bmatrix}0&\A\\\A^{*}&0\end{bmatrix}$.
\item[\bf(ii)]
$\lambda_{\A}^2$ is the smallest positive eigenvalue of $\A^{*}\A$ and of $\A\A^{*}$.
\item[\bf(iii)]
$\lambda_{\A}^2$ is the smallest positive eigenvalue of $\begin{bmatrix}\A^{*}\A&0\\0&\A\A^{*}\end{bmatrix}$
and of $\begin{bmatrix}\A\A^{*}&0\\0&\A^{*}\A\end{bmatrix}$.
\item[\bf(iv)] 
$\displaystyle
\sigma(\A^{*}\A)\setminus\{0\}
=\sigma(\A\A^{*})\setminus\{0\}
=\sigma\Big(\begin{bmatrix}\A^{*}\A&0\\0&\A\A^{*}\end{bmatrix}\Big)\setminus\{0\}
=\sigma\Big(\begin{bmatrix}\A\A^{*}&0\\0&\A^{*}\A\end{bmatrix}\Big)\setminus\{0\}
>0$
\item[\bf(v)] 
$\displaystyle
\sigma\Big(\begin{bmatrix}0&\A^{*}\\\A&0\end{bmatrix}\Big)\setminus\{0\}
=\sigma\Big(\begin{bmatrix}0&\A\\\A^{*}&0\end{bmatrix}\Big)\setminus\{0\}
=\pm\sqrt{\sigma(\A^{*}\A)\setminus\{0\}}$
\item[\bf(vi)] 
$\sigma(\A^{*}\A)\setminus\{0\}=\sigma(\cA^{*}\cA)$ 
and corresponding results hold for all other spectra in {\bf(iv)}.
\item[\bf(vii)] 
$\displaystyle
\sigma\Big(\begin{bmatrix}0&\A^{*}\\\A&0\end{bmatrix}\Big)\setminus\{0\}
=\sigma\Big(\begin{bmatrix}0&\cA^{*}\\\cA&0\end{bmatrix}\Big)$ 
and corresponding results hold for all other spectra in {\bf(v)}.
\item[\bf(viii)]
$\bnorm{\cA^{-1}}_{R(\A),R(\A^{*})}
=\bnorm{(\cA^{*})^{-1}}_{R(\A^{*}),R(\A)}
=c_{\A}$
\item[\bf(viii${}^{*}$)]
$\bnorm{\cA^{-1}}_{R(\A),D(\A)}
=\bnorm{(\cA^{*})^{-1}}_{R(\A^{*}),D(\A^{*})}
=(c_{\A}^2+1)^{1/2}$
\item[\bf(ix)]
$\bnorm{(\cA^{*}\cA)^{-1}}_{R(\A^{*}),R(\A^{*})}
=\bnorm{(\cA\cA^{*})^{-1}}_{R(\A),R(\A)}
=c_{\A}^2$
\item[\bf(x)] 
$N(\A)=N(\A^{*}\A)=N(\A\A^{*}\A)=\dots$ and $N(\A^{*})=N(\A\A^{*})=N(\A^{*}\A\A^{*})=\dots$.
\item[\bf(xi)] 
$R(\A)=R(\A\A^{*})=R(\A\A^{*}\A)=\dots$ and $R(\A^{*})=R(\A^{*}\A)=R(\A^{*}\A\A^{*})=\dots$
and the same holds for the operators $\cA$ and $\cA^{*}$.
\end{itemize}
\end{lem}

For a proof see Appendix \ref{appproofs}.

\begin{rem}[variational formulations]
\label{remconstev}
By Lemma \ref{lemconstev} the infima in \eqref{lambdaAdef} are minima,
provided that $D(\cA)\emb\H{}{0}$ is compact. 
In particular, the respective minimisers $x_{\A}$ and $y_{\A}$
are the eigenvectors to the eigenvalue $\lambda_{\A}^2$, i.e, 
$$\inf_{0\neq x\in D(\cA)}\frac{\norm{\A x}_{\H{}{1}}}{\norm{x}_{\H{}{0}}}
=\norm{\A x_{\A}}_{\H{}{1}}
=\lambda_{\A}
=\norm{\A^{*}y_{\A}}_{\H{}{0}}
=\inf_{0\neq y\in D(\cA^{*})}\frac{\norm{\A^{*} y}_{\H{}{0}}}{\norm{y}_{\H{}{1}}},$$
where we assume without loss of generality $\norm{x_{\A}}_{\H{}{0}}=\norm{y_{\A}}_{\H{}{1}}=1$.
Moreover,
\begin{align*}
(\A^{*}\A-\lambda_{\A}^2)x_{\A}&=0,
&
x_{\A}&\in D(\A^{*}\A)\cap R(\A^{*})=D(\cA^{*}\cA)\subset D(\cA),\\
(\A\A^{*}-\lambda_{\A}^2)y_{\A}&=0,
&
y_{\A}&\in D(\A\A^{*})\cap R(\A)=D(\cA\cA^{*})\subset D(\cA^{*}),
\end{align*}
and the eigenvectors satisfy the variational formulations
\begin{align*}
\forall\,\phi&\in D(\A)
&
\scp{\A x_{\A}}{\A\phi}_{\H{}{1}}
&=\lambda_{\A}^2\scp{x_{\A}}{\phi}_{\H{}{0}},\\
\forall\,\psi&\in D(\A^{*})
&
\scp{\A^{*}y_{\A}}{\A^{*}\psi}_{\H{}{0}}
&=\lambda_{\A}^2\scp{y_{\A}}{\psi}_{\H{}{1}}.
\end{align*}
\end{rem}

\subsubsection{Complex Structure Results}

Now, let 
$$\A_{0}\!:\!D(\A_{0})\subset\H{}{0}\to\H{}{1},\quad
\A_{1}\!:\!D(\A_{1})\subset\H{}{1}\to\H{}{2}$$
be two densely defined and closed linear operators 
on three Hilbert spaces $\H{}{0}$, $\H{}{1}$, and $\H{}{2}$ with adjoints 
$$\A_{0}^{*}\!:\!D(\A_{0}^{*})\subset\H{}{1}\to\H{}{0},\quad
\A_{1}^{*}\!:\!D(\A_{1}^{*})\subset\H{}{2}\to\H{}{1}$$
as well as reduced operators $\cA_{0}$, $\cA_{0}^{*}$, and $\cA_{1}$, $\cA_{1}^{*}$.
Furthermore, we assume the complex property (also called sequence property)
of $\A_{0}$ and $\A_{1}$, that is $\A_{1}\A_{0}=0$, i.e.,
\begin{align}
\label{sequenceprop}
R(\A_{0})\subset N(\A_{1}),
\end{align}
which is equivalent to $\A_{0}^{*}\A_{1}^{*}=0$, i.e., $R(\A_{1}^{*})\subset N(\A_{0}^{*})$.
Recall that
$$R(\A_{0})=R(\cA_{0}),\quad
R(\A_{0}^{*})=R(\cA_{0}^{*}),\quad
R(\A_{1})=R(\cA_{1}),\quad
R(\A_{1}^{*})=R(\cA_{1}^{*}).$$
From the Helmholtz type decompositions \eqref{helm} for $\A=\A_{0}$ and $\A=\A_{1}$ we get in particular
\begin{align}
\label{helmappclone}
\H{}{1}=\ol{R(\A_{0})}\oplus_{\H{}{1}}N(\A_{0}^{*}),\quad
\H{}{1}=\ol{R(\A_{1}^{*})}\oplus_{\H{}{1}}N(\A_{1}).
\end{align}
Introducing the cohomology group 
$$N_{0,1}:=N(\A_{1})\cap N(\A_{0}^{*}),$$
we obtain the refined Helmholtz type decompositions 
\begin{align}
\label{helmrefinedone}
\begin{aligned}
N(\A_{1})&=\ol{R(\A_{0})}\oplus_{\H{}{1}}N_{0,1},
&
N(\A_{0}^{*})&=\ol{R(\A_{1}^{*})}\oplus_{\H{}{1}}N_{0,1},\\
D(\A_{1})&=\ol{R(\A_{0})}\oplus_{\H{}{1}}\big(D(\A_{1})\cap N(\A_{0}^{*})\big),\quad
&
D(\A_{0}^{*})&=\ol{R(\A_{1}^{*})}\oplus_{\H{}{1}}\big(D(\A_{0}^{*})\cap N(\A_{1})\big),
\end{aligned}
\end{align}
and therefore
\begin{align}
\label{helmrefinedtwo}
\H{}{1}&=\ol{R(\A_{0})}\oplus_{\H{}{1}}N_{0,1}\oplus_{\H{}{1}}\ol{R(\A_{1}^{*})}.
\end{align}
Let us remark that the first line of \eqref{helmrefinedone} can also be written as
$$\ol{R(\A_{0})}=N(\A_{1})\cap N_{0,1}^{\bot_{\H{}{1}}},\quad
\ol{R(\A_{1}^{*})}=N(\A_{0}^{*})\cap N_{0,1}^{\bot_{\H{}{1}}}.$$
Note that \eqref{helmrefinedtwo} can be further refined and specialised, e.g., to
\begin{align}
\begin{split}
\label{helmrefinedthree}
D(\A_{1})&=\ol{R(\A_{0})}\oplus_{\H{}{1}}N_{0,1}\oplus_{\H{}{1}}D(\cA_{1}),\\
D(\A_{0}^{*})&=D(\cA_{0}^{*})\oplus_{\H{}{1}}N_{0,1}\oplus_{\H{}{1}}\ol{R(\A_{1}^{*})},\\
D(\A_{1})\cap D(\A_{0}^{*})&=D(\cA_{0}^{*})\oplus_{\H{}{1}}N_{0,1}\oplus_{\H{}{1}}D(\cA_{1}).
\end{split}
\end{align}
We observe
\begin{align*}
D(\cA_{1})
&=D(\A_{1})\cap\ol{R(\A_{1}^{*})}
\subset D(\A_{1})\cap N(\A_{0}^{*})
\subset D(\A_{1})\cap D(\A_{0}^{*}),\\
D(\cA_{0}^{*})
&=D(\A_{0}^{*})\cap\ol{R(\A_{0})}
\subset D(\A_{0}^{*})\cap N(\A_{1})
\subset D(\A_{0}^{*})\cap D(\A_{1}),
\end{align*}
and using the refined Helmholtz type decompositions \eqref{helmrefinedtwo} and \eqref{helmrefinedthree}
as well as the results of Lemma \ref{fatbl2} we immediately see:

\begin{lem}[fa-toolbox lemma 3]
\label{compemblem}
The following assertions are equivalent: 
\begin{itemize}
\item[\bf(i)]
$D(\cA_{0})\emb\H{}{0}$, $D(\cA_{1})\emb\H{}{1}$,
and $N_{0,1}\emb\H{}{1}$ are compact.
\item[\bf(ii)]
$D(\A_{1})\cap D(\A_{0}^{*})\emb\H{}{1}$ is compact.
\end{itemize}
In this case, the cohomology group $N_{0,1}$ has finite dimension.
\end{lem}

We summarise:

\begin{theo}[fa-toolbox theorem]
\label{compembtheo}
Let the ranges $R(\A_{0})$ and $R(\A_{1})$ be closed.
Then all ranges $R(\A_{0})$, $R(\A_{0}^{*})$, and $R(\A_{1})$, $R(\A_{1}^{*})$ are closed, and
the corresponding Poincar\'e-Friedrichs type estimates hold, i.e.
there exists positive constants $c_{\A_{0}},c_{\A_{1}}$ such that
\begin{align*}
\forall\,z&\in D(\cA_{0})=D(\A_{0})\cap R(\A_{0}^{*})
&
\norm{z}_{\H{}{0}}&\leq c_{\A_{0}}\norm{\A_{0} z}_{\H{}{1}},\\
\forall\,x&\in D(\cA_{0}^{*})=D(\A_{0}^{*})\cap R(\A_{0})=D(\A_{0}^{*})\cap N(\A_{1})\cap N_{0,1}^{\bot_{\H{}{1}}}
&
\norm{x}_{\H{}{1}}&\leq c_{\A_{0}}\norm{\A_{0}^{*} x}_{\H{}{0}},\\
\forall\,x&\in D(\cA_{1})=D(\A_{1})\cap R(\A_{1}^{*})=D(\A_{1})\cap N(\A_{0}^{*})\cap N_{0,1}^{\bot_{\H{}{1}}}
&
\norm{x}_{\H{}{1}}&\leq c_{\A_{1}}\norm{\A_{1} x}_{\H{}{2}},\\
\forall\,y&\in D(\cA_{1}^{*})=D(\A_{1}^{*})\cap R(\A_{1})
&
\norm{y}_{\H{}{2}}&\leq c_{\A_{1}}\norm{\A_{1}^{*} y}_{\H{}{1}},
\end{align*}
and
$$\forall\,x\in D(\A_{1})\cap D(\A_{0}^{*})\cap N_{0,1}^{\bot_{\H{}{1}}}\qquad
\norm{x}_{\H{}{1}}^2
\leq c_{\A_{1}}^2\norm{\A_{1} x}_{\H{}{2}}^2
+c_{\A_{0}}^2\norm{\A_{0}^{*} x}_{\H{}{0}}^2.$$
Moreover, all refined Helmholtz type decompositions \eqref{helmrefinedone}-\eqref{helmrefinedthree}
hold with closed ranges, in particular, e.g.,
$$\H{}{1}=R(\A_{0})\oplus_{\H{}{1}}N_{0,1}\oplus_{\H{}{1}}R(\A_{1}^{*}).$$
\end{theo}

\begin{rem}
\label{compembremone}
Let us define $c_{\A_{0},\A_{1}}>0$ by 
$$\frac{1}{c_{\A_{0},\A_{1}}^2}
:=\inf
\frac{\norm{\A_{1} x}_{\H{}{2}}^2+\norm{\A_{0}^{*} x}_{\H{}{0}}^2}
{\norm{x}_{\H{}{0}}^2},$$
where the infimum is taken over all 
$0\neq x\in D(\A_{1})\cap D(\A_{0}^{*})\cap N_{0,1}^{\bot_{\H{}{1}}}$.
Assuming - as mentioned above - that we always take the best constants, 
we obtain by Theorem \ref{compembtheo}
$$c_{\A_{0},\A_{1}}=\max\{c_{\A_{0}},c_{\A_{1}}\}.$$
This can be seen as follows: Theorem \ref{compembtheo} shows
$c_{\A_{0},\A_{1}}\leq\max\{c_{\A_{0}},c_{\A_{1}}\}$.
On the other hand, for
$$x\in D(\cA_{1})
=D(\A_{1})\cap R(\A_{1}^{*})
=D(\A_{1})\cap N(\A_{0}^{*})\cap N_{0,1}^{\bot_{\H{}{1}}}$$
we have $\norm{x}_{\H{}{1}}\leq c_{\A_{0},\A_{1}}\norm{\A_{1} x}_{\H{}{2}}$
and hence $c_{\A_{1}}\leq c_{\A_{0},\A_{1}}$. Analogously we get $c_{\A_{0}}\leq c_{\A_{0},\A_{1}}$.
\end{rem}

\begin{rem}
\label{compembremtwo}
If $D(\A_{1})\cap D(\A_{0}^{*})\emb\H{}{1}$ is compact, then
$D(\cA_{0})\emb\H{}{0}$, $D(\cA_{1})\emb\H{}{1}$,
and $D(\cA_{0}^{*})\emb\H{}{1}$, $D(\cA_{1}^{*})\emb\H{}{2}$ are compact,
as well as $\dim N_{0,1}<\infty$.
Hence all ranges are closed, see Remark \ref{fatbr},
and all assertions of Theorem \ref{compembtheo} hold.
\end{rem}

In other words, the primal and dual complex, i.e.,
\begin{align}
\begin{split}
\label{genpdcomplex}
\begin{CD}
D(\A_{0}) @> \A_{0} >>
D(\A_{1}) @> \A_{1} >>
\H{}{2},
\end{CD}\\
\begin{CD}
\H{}{0} @< \A_{0}^{*} <<
D(\A_{0}^{*}) @< \A_{1}^{*} <<
D(\A_{1}^{*}),
\end{CD}
\end{split}
\end{align}
is a Hilbert complex of closed and densely defined linear operators.
The additional assumption that the ranges $R(\A_{0})$ and $R(\A_{1})$ are closed
(and so also the ranges $R(\A_{0}^{*})$ and $R(\A_{1}^{*})$)
is equivalent to the closedness of the Hilbert complex.
Moreover, the complex is exact if and only if $N_{0,1}=\{0\}$.
The complex is called compact, if 
\begin{align}
\label{cptembcrucial}
D(\A_{1})\cap D(\A_{0}^{*})\emb\H{}{1}
\end{align} 
is compact. Remark \ref{compembremtwo} shows that \eqref{cptembcrucial}
is the crucial assumption for the complex \eqref{genpdcomplex}.

Finally, we present some results for the (unbounded linear) operator
$$\A_{0}\A_{0}^{*}+\A_{1}^{*}\A_{1}:
D(\A_{0}\A_{0}^{*}+\A_{1}^{*}\A_{1})\subset\H{}{1}\to\H{}{1}$$
with
$D(\A_{0}\A_{0}^{*}+\A_{1}^{*}\A_{1})
:=\big\{x\in D(\A_{1})\cap D(\A_{0}^{*})\,:\,\A_{1}x\in D(\A_{1}^{*})\,\wedge\,\A_{0}^{*}x\in D(\A_{0})\big\}$.

\begin{lem}[constant and eigenvalue lemma]
\label{lemconstevcomplex}
Let $D(\A_{1})\cap D(\A_{0}^{*})\emb\H{}{1}$ be compact. Then:
\begin{itemize}
\item[\bf(i)]
$\A_{0}^{*}\A_{0}$, $\A_{0}\A_{0}^{*}$, $\A_{1}^{*}\A_{1}$, $\A_{1}\A_{1}^{*}$, 
and $\A_{0}\A_{0}^{*}+\A_{1}^{*}\A_{1}$
are self-adjoint and have pure and discrete point spectra
with no accumulation point in $\reals$.
\item[\bf(ii)]
The results of Lemma \ref{lemconstev} hold for $\A_{0}$ and $\A_{1}$, in particular
$\sigma(\A_{0}^{*}\A_{0})\setminus\{0\}=\sigma(\A_{0}\A_{0}^{*})\setminus\{0\}$
and $\sigma(\A_{1}^{*}\A_{1})\setminus\{0\}=\sigma(\A_{1}\A_{1}^{*})\setminus\{0\}$
as well as $N(\A_{0}\A_{0}^{*})=N(\A_{0}^{*})$ and $N(\A_{1}^{*}\A_{1})=N(\A_{1})$.
\item[\bf(iii)] 
$N(\A_{0}\A_{0}^{*}+\A_{1}^{*}\A_{1})=N_{0,1}$
and $R(\A_{0}\A_{0}^{*}+\A_{1}^{*}\A_{1})=N_{0,1}^{\bot_{\H{}{1}}}$,
in particular the range is closed.
\item[\bf(iv)] 
$\A_{0}\A_{0}^{*}+\A_{1}^{*}\A_{1}:
D(\A_{0}\A_{0}^{*}+\A_{1}^{*}\A_{1})\cap N_{0,1}^{\bot_{\H{}{1}}}\subset N_{0,1}^{\bot_{\H{}{1}}}
\to N_{0,1}^{\bot_{\H{}{1}}}$ is bijective with compact inverse.
\item[\bf(iv${}^{*}$)] 
$\A_{0}\A_{0}^{*}+\A_{1}^{*}\A_{1}:
D(\A_{0}\A_{0}^{*}+\A_{1}^{*}\A_{1})\cap N_{0,1}^{\bot_{\H{}{1}}}\to N_{0,1}^{\bot_{\H{}{1}}}$
is a topological isomorphism.
\end{itemize}
Moreover, the spectrum of $\A_{0}\A_{0}^{*}+\A_{1}^{*}\A_{1}$
is given by the spectra of $\A_{0}\A_{0}^{*}$ and $\A_{1}^{*}\A_{1}$, i.e.,
\begin{itemize}
\item[\bf(v)]
$\sigma(\A_{0}\A_{0}^{*}+\A_{1}^{*}\A_{1})\setminus\{0\}
=\big(\sigma(\A_{0}^{*}\A_{0})\setminus\{0\}\big)
\cup\big(\sigma(\A_{1}^{*}\A_{1})\setminus\{0\}\big)$.
\item[\bf(v${}^{*}$)] 
In particular, the smallest positive eigenvalue of $\A_{0}\A_{0}^{*}+\A_{1}^{*}\A_{1}$ 
is given by $\min\{\lambda_{\A_{0}}^2,\lambda_{\A_{1}}^2\}$.
\end{itemize}
\end{lem}

For a proof see Appendix \ref{appproofs}.

\begin{rem}[Helmholtz decomposition]
\label{lemconstevcomplexremhelm}
$\A_{0}\A_{0}^{*}+\A_{1}^{*}\A_{1}$ provides the Helmholtz decomposition from Theorem \ref{compembtheo}.
To see this, let us denote the orthonormal projector onto the cohomology group $N_{0,1}$ by
$\pi_{N_{0,1}}:\H{}{1}\to N_{0,1}$. Then, for $x\in\H{}{1}$ we have
$(1-\pi_{N_{0,1}})x\in N_{0,1}^{\bot_{\H{}{1}}}$ and
\begin{align*}
x&=\pi_{N_{0,1}}x+(1-\pi_{N_{0,1}})x\\
&=\pi_{N_{0,1}}x+(\A_{0}\A_{0}^{*}+\A_{1}^{*}\A_{1})(\A_{0}\A_{0}^{*}+\A_{1}^{*}\A_{1})^{-1}(1-\pi_{N_{0,1}})x
\in N_{0,1}\oplus_{\H{}{1}}R(\A_{0})\oplus_{\H{}{1}}R(\A_{1}^{*}).
\end{align*}
\end{rem}

\subsection{Laplace and Maxwell Constants in 3D}
\label{seclapmax3D}

Now, we specialise to linear acoustics and electromagnetics in 3D,
i.e., to the classical operators of the 3D-de Rham complex, cf. \eqref{genpdcomplex},
\begin{align}
\begin{split}
\label{derhamcomplex}
\begin{CD}
\H{}{\gat}(\grad{},\om) @> \A_{0}=\grad{\gat} >>
\H{}{\gat}(\rot{},\om) @> \A_{1}=\rot{\gat} >>
\H{}{\gat}(\div{},\om) @> \A_{2}=\div{\gat} >>
\L{2}(\om),
\end{CD}\\
\begin{CD}
\L{2}(\om) @< \A_{0}^{*}=-\div{\gan} <<
\H{}{\gan}(\div{},\om) @< \A_{1}^{*}=\rot{\gan} <<
\H{}{\gan}(\rot{},\om) @< \A_{2}^{*}=-\grad{\gan} <<
\H{}{\gan}(\grad{},\om),
\end{CD}
\end{split}
\end{align}
and apply the fa-toolbox to these operators.

More precisely, 
let $\om\subset\reals^{3}$ be a bounded weak Lipschitz domain, 
see \cite[Definition 2.3]{bauerpaulyschomburgmaxcompweaklip} for details, with boundary $\ga:=\p\om$, 
which is divided into two relatively open weak Lipschitz subsets $\gat$ 
and $\gan:=\ga\setminus\ol{\gat}$ (its complement),
see \cite[Definition 2.5]{bauerpaulyschomburgmaxcompweaklip} for details.
We shall call $(\om,\gat)$ a bounded weak Lipschitz pair.
Moreover, if $(\om,\gat)$ is a bounded weak Lipschitz pair, so is $(\om,\gan)$.
Note that strong Lipschitz (graph of Lipschitz functions) implies weak Lipschitz (Lipschitz manifolds)
for the boundary as well as for the interface. 
We introduce the usual Lebesgue and Sobolev spaces by $\L{2}(\om)$ 
and $\H{k}{}(\om)$, $k\in\mathbb{N}_{0}$. For $k=1$ we also write
$$\H{1}{}(\om)
=\H{}{}(\grad{},\om)
:=\big\{u\in\L{2}(\om):\grad{}u\in\L{2}(\om)\big\}.$$
Homogeneous weak boundary conditions (in the strong sense) 
are defined by closure of respective test functions, i.e.,
\begin{align}
\label{Hgatdef}
\H{1}{\gat}(\om)
=\H{}{\gat}(\grad{},\om)
:=\ol{\C{\infty}{\gat}(\om)}^{\H{}{}(\grad{},\om)},
\end{align}
where
$$\C{\infty}{\gat}(\om)
:=\big\{u|_{\om}:u\in\C{\infty}{}(\reals^{3}),\,
\text{\rm supp}\,u\;\text{\rm  compact in }\reals^{3},\,
\text{\rm dist}(\text{\rm supp}\,u,\gat)>0\big\}.$$
Analogously we define (using test vector fields)
$$\H{}{}(\rot{},\om),\quad
\H{}{\gat}(\rot{},\om),\quad
\H{}{}(\div{},\om),\quad
\H{}{\gat}(\div{},\om).$$
All latter definitions extend to $\om\subset\reals^{N}$, $N\geq1$, 
in an obvious way, see \cite{bauerpaulyschomburgmaxcompweakliprNarxiv,bauerpaulyschomburgmaxcompweakliprN} for details.
Throughout this paper and until otherwise stated,
we shall assume the latter minimal regularity on $\om$ and $\gat$.

\begin{ass}
$(\om,\gat)$ is a bounded weak Lipschitz pair.
\end{ass}

As closures of the respective classical operators of vector analysis
defined on test functions/vector fields from $\C{\infty}{\gat}(\om)$,
we consider the densely defined and closed linear operators
\begin{align*}
\A_{0}:=\grad{\gat}:D(\grad{\gat})\subset\L{2}(\om)&\To\L{2}(\om);
&
u&\mapsto\grad{}u,\\
\A_{1}:=\rot{\gat}:D(\rot{\gat})\subset\L{2}(\om)&\To\L{2}(\om);
&
E&\mapsto\rot{}E,\\
\A_{2}:=\div{\gat}:D(\div{\gat})\subset\L{2}(\om)&\To\L{2}(\om);
&
H&\mapsto\div{}H,
\intertext{together with their adjoints, 
see \cite[Theorem 4.5, Section 5.2]{bauerpaulyschomburgmaxcompweaklip} and
\cite[Theorem 4.7, Section 5.2]{bauerpaulyschomburgmaxcompweakliprNarxiv,bauerpaulyschomburgmaxcompweakliprN},}
\A^{*}_{0}=\grad{\gat}^{*}=-\div{\gan}:D(\div{\gan})\subset\L{2}(\om)&\To\L{2}(\om);
&
H&\mapsto-\div{}H,\\
\A^{*}_{1}=\rot{\gat}^{*}=\rot{\gan}:D(\rot{\gan})\subset\L{2}(\om)&\To\L{2}(\om);
&
E&\mapsto\rot{}E,\\
\A^{*}_{2}=\div{\gat}^{*}=-\grad{\gan}:D(\grad{\gan})\subset\L{2}(\om)&\To\L{2}(\om);
&
u&\mapsto-\grad{}u.
\end{align*}
Note that
$$D(\grad{\gat})=\H{}{\gat}(\grad{},\om),\quad
D(\rot{\gat})=\H{}{\gat}(\rot{},\om),\quad
D(\div{\gat})=\H{}{\gat}(\div{},\om)$$
and that \eqref{derhamcomplex} is indeed a Hilbert complex.

Recently, in \cite{bauerpaulyschomburgmaxcompweaklip,bauerpaulyschomburgmaxcompweakliprNarxiv,bauerpaulyschomburgmaxcompweakliprN}, 
Weck's selection theorem, also known as the Maxwell compactness property, has been shown
to hold for such bounded weak Lipschitz domains and mixed boundary conditions.

\begin{theo}[Weck's selection theorem]
\label{weckst}
The embedding
$$\H{}{\gat}(\rot{},\om)\cap\H{}{\gan}(\div{},\om)\emb\L{2}(\om)$$
is compact.
\end{theo}

For a proof see \cite{bauerpaulyschomburgmaxcompweaklip,bauerpaulyschomburgmaxcompweakliprNarxiv,bauerpaulyschomburgmaxcompweakliprN}.
A short historical overview of Weck's selection theorem
is given in the introduction of \cite{bauerpaulyschomburgmaxcompweaklip},
see also the original paper \cite{weckmax} and 
\cite{picardcomimb,webercompmax,costabelremmaxlip,witschremmax,jochmanncompembmaxmixbc,leisbook,picardweckwitschxmas}.

Now, Theorem \ref{weckst} implies that the crucial assumption \eqref{cptembcrucial} holds
for the operators $\A_{n}$ of the de Rham complex \eqref{derhamcomplex}, 
cf. the general complex \eqref{genpdcomplex}.
More precisely, by Theorem \ref{weckst}
\begin{align*}
\boldsymbol{\mathrm{(a)}}&
&
D(\A_{1})\cap D(\A_{0}^{*})
&=\H{}{\gat}(\rot{},\om)\cap\H{}{\gan}(\div{},\om)
\emb\L{2}(\om)=\H{}{1},\\
\boldsymbol{\mathrm{(b)}}&
&
D(\A_{2})\cap D(\A_{1}^{*})
&=\H{}{\gat}(\div{},\om)\cap\H{}{\gan}(\rot{},\om)
\emb\L{2}(\om)=\H{}{2}
\end{align*}
are compact and, hence, \eqref{derhamcomplex} is a compact Hilbert complex.
Thus, by Theorem \ref{compembtheo} and Remark \ref{compembremtwo},
all ranges are closed, all corresponding Poincar\'e-Friedrichs type estimates hold, 
and all refined Helmholtz type decompositions \eqref{helmrefinedone}-\eqref{helmrefinedthree}
hold with closed ranges. In particular, denoting the corresponding constants by
\begin{align}
\nonumber
\frac{1}{\lambda_{0,\gat}}:=c_{0,\gat}
:=c_{\grad{\gat}}:=c_{\A_{0}}&=c_{\A_{0}^{*}}=c_{\div{\gan}}
=c_{2,\gan}=\frac{1}{\lambda_{2,\gan}},\\
\label{constdef3D}
\frac{1}{\lambda_{1,\gat}}:=c_{1,\gat}
:=c_{\rot{\gat}}:=c_{\A_{1}}&=c_{\A_{1}^{*}}=c_{\rot{\gan}}
=c_{1,\gan}=\frac{1}{\lambda_{1,\gan}},\\
\nonumber
\frac{1}{\lambda_{2,\gat}}:=c_{2,\gat}
:=c_{\div{\gat}}:=c_{\A_{2}}&=c_{\A_{2}^{*}}=c_{\grad{\gan}}
=c_{0,\gan}=\frac{1}{\lambda_{0,\gan}},
\end{align}
and introducing the (finite-dimensional) cohomology groups
\begin{align*}
\Harm{}{1}&:=N_{0,1}:=N(\A_{1})\cap N(\A_{0}^{*})=N(\rot{\gat})\cap N(\div{\gan}),\\
\Harm{}{2}&:=N_{1,2}:=N(\A_{2})\cap N(\A_{1}^{*})=N(\div{\gat})\cap N(\rot{\gan}),
\end{align*}
the so-called Dirichlet/Neumann fields,
we have by Theorem \ref{compembtheo} and Remark \ref{compembremtwo} the following inequalities:

\begin{theo}[Poincar\'e-Friedrichs type estimates]
\label{fpest}
It holds
\begin{align*}
\forall\,u&\in D(\cA_{0})=D(\grad{\gat})\cap R(\div{\gan})
&
\norm{u}_{\L{2}(\om)}&\leq c_{0,\gat}\norm{\grad{}u}_{\L{2}(\om)},\\
\forall\,E&\in D(\cA_{0}^{*})=D(\div{\gan})\cap R(\grad{\gat})
&
\norm{E}_{\L{2}(\om)}&\leq c_{0,\gat}\norm{\div{}E}_{\L{2}(\om)},\\
\forall\,E&\in D(\cA_{1})=D(\rot{\gat})\cap R(\rot{\gan})
&
\norm{E}_{\L{2}(\om)}&\leq c_{1,\gat}\norm{\rot{}E}_{\L{2}(\om)},\\
\forall\,H&\in D(\cA_{1}^{*})=D(\rot{\gan})\cap R(\rot{\gat})
&
\norm{H}_{\L{2}(\om)}&\leq c_{1,\gat}\norm{\rot{}H}_{\L{2}(\om)},
\end{align*}
and for all $E\in D(\A_{1})\cap D(\A_{0}^{*})\cap N_{0,1}^{\bot_{\H{}{1}}}
=D(\rot{\gat})\cap D(\div{\gan})\cap\Harm{\bot_{\L{2}(\om)}}{1}$ 
$$\norm{E}_{\L{2}(\om)}^2
\leq c_{1,\gat}^2\norm{\rot{}E}_{\L{2}(\om)}^2
+c_{0,\gat}^2\norm{\div{}E}_{\L{2}(\om)}^2,$$
where
\begin{align*}
R(\grad{\gat})&=N(\rot{\gat})\cap\Harm{\bot_{\L{2}(\om)}}{1},
&
R(\rot{\gan})&=N(\div{\gan})\cap\Harm{\bot_{\L{2}(\om)}}{1},\\
R(\grad{\gan})&=N(\rot{\gan})\cap\Harm{\bot_{\L{2}(\om)}}{2},
&
R(\rot{\gat})&=N(\div{\gat})\cap\Harm{\bot_{\L{2}(\om)}}{2}.
\end{align*}
\end{theo}

Let $c_{0,1,\gat}:=c_{\grad{\gat},\rot{\gat}}>0$ be defined by 
$$\frac{1}{c_{0,1,\gat}^2}
:=\inf
\frac{\norm{\rot{}E}_{\L{2}(\om)}^2+\norm{\div{}E}_{\L{2}(\om)}^2}
{\norm{E}_{\L{2}(\om)}^2},$$
where the infimum taken over all 
$0\neq E\in D(\rot{\gat})\cap D(\div{\gan})\cap\Harm{\bot_{\L{2}(\om)}}{1}$.

\begin{rem}
\label{fpestrem}
By Remark \ref{compembremone} it holds
$c_{0,1,\gat}=\max\{c_{0,\gat},c_{1,\gat}\}$.
\end{rem}

Note that by the symmetry of the de Rham complex the corresponding two estimates 
for $\cA_{2}$ and $\cA_{2}^{*}$, i.e., 
\begin{align*}
\forall\,H&\in D(\cA_{2})=D(\div{\gat})\cap R(\grad{\gan})
&
\norm{H}_{\L{2}(\om)}&\leq c_{2,\gat}\norm{\div{}H}_{\L{2}(\om)},\\
\forall\,u&\in D(\cA_{2}^{*})=D(\grad{\gan})\cap R(\div{\gat})
&
\norm{u}_{\L{2}(\om)}&\leq c_{2,\gat}\norm{\grad{}u}_{\L{2}(\om)},
\end{align*}
are redundant, as these are already included in the two estimates 
for $\cA_{0}$ and $\cA_{0}^{*}$ just by interchanging the boundary conditions on $\gat$ and $\gan$.
In other words, $c_{2,\gat}=c_{0,\gan}$.
Furthermore, 
\begin{align*}
N(\grad{\gat})
&\;=\begin{cases}\{0\}&\text{if }\gat\neq\emptyset,\\
\reals&\text{if }\gat=\emptyset,\end{cases}\\
R(\div{\gan})=N(\grad{\gat})^{\bot_{\L{2}(\om)}}=\L{2}_{\gan}(\om)
&:=\begin{cases}\L{2}(\om)&\text{if }\gan\neq\ga,\\
\L{2}(\om)\cap\reals^{\bot_{\L{2}(\om)}}&\text{if }\gan=\ga,\end{cases}
\end{align*}
where
$$\L{2}(\om)\cap\reals^{\bot_{\L{2}(\om)}}
=\big\{u\in\L{2}(\om):\scp{u}{1}_{\L{2}(\om)}=0\big\}
=\big\{u\in\L{2}(\om):\int_{\om}u=0\big\}.$$

Combinations of the latter operators give the well known operators 
from acoustics, Maxwell equations, Laplace equations, and the double rotation equations, i.e.,
\begin{align*}
\M_{0}&:=\begin{bmatrix}0&\A_{0}^{*}\\\A_{0}&0\end{bmatrix}
=\begin{bmatrix}0&\grad{\gat}^{*}\\\grad{\gat}&0\end{bmatrix}
=\begin{bmatrix}0&-\div{\gan}\\\grad{\gat}&0\end{bmatrix},\\
\M_{1}&:=\begin{bmatrix}0&\A_{1}^{*}\\\A_{1}&0\end{bmatrix}
=\begin{bmatrix}0&\rot{\gat}^{*}\\\rot{\gat}&0\end{bmatrix}
=\begin{bmatrix}0&\rot{\gan}\\\rot{\gat}&0\end{bmatrix},\\
\M_{2}&:=\begin{bmatrix}0&\A_{2}^{*}\\\A_{2}&0\end{bmatrix}
=\begin{bmatrix}0&\div{\gat}^{*}\\\div{\gat}&0\end{bmatrix}
=\begin{bmatrix}0&-\grad{\gan}\\\div{\gat}&0\end{bmatrix},
\end{align*}
and
\begin{align*}
\A_{0}^{*}\A_{0}&=\grad{\gat}^{*}\grad{\gat}
=-\div{\gan}\grad{\gat}
=:-\Delta_{\gat},
&
\A_{0}\A_{0}^{*}&=\grad{\gat}\grad{\gat}^{*}
=-\grad{\gat}\div{\gan}
=-\Diamond_{\gan},\\
\A_{1}^{*}\A_{1}&=\rot{\gat}^{*}\rot{\gat}
=\rot{\gan}\rot{\gat}
=:\Box_{\gat},
&
\A_{1}\A_{1}^{*}&=\rot{\gat}\rot{\gat}^{*}
=\rot{\gat}\rot{\gan}
=\square_{\gan},\\
\A_{2}^{*}\A_{2}&=\div{\gat}^{*}\div{\gat}
=-\grad{\gan}\div{\gat}
=:-\Diamond_{\gat},
&
\A_{2}\A_{2}^{*}&=\div{\gat}\div{\gat}^{*}
=-\div{\gat}\grad{\gan}
=-\Delta_{\gan}.
\end{align*}
Again, $\M_{2}$ and the operators involving $\A_{2}$, $\A_{2}^{*}$ are redundant
by interchanging the boundary conditions in $\M_{0}$ and $\A_{0}$, $\A_{0}^{*}$.
Hence, we may focus on $c_{0,\gat}$ and $c_{1,\gat}$. Section \ref{seccstgen} shows the following:

\begin{theo}[Poincar\'e-Friedrichs type constants]
\label{fpmconst}
The Poincar\'e-Friedrichs type constants can be computed 
by the four Rayleigh quotations
\begin{align*}
\frac{1}{c_{0,\gat}}=\lambda_{0,\gat}
&=\inf_{0\neq u\in D(\grad{\gat})\cap\L{2}_{\gan}(\om)}\frac{\norm{\grad{}u}_{\L{2}(\om)}}{\norm{u}_{\L{2}(\om)}}
=\inf_{0\neq E\in D(\div{\gan})\cap R(\grad{\gat})}\frac{\norm{\div{}E}_{\L{2}(\om)}}{\norm{E}_{\L{2}(\om)}},\\
\frac{1}{c_{1,\gat}}=\lambda_{1,\gat}
&=\inf_{0\neq E\in D(\rot{\gat})\cap R(\rot{\gan})}\frac{\norm{\rot{}E}_{\L{2}(\om)}}{\norm{E}_{\L{2}(\om)}}
=\inf_{0\neq H\in D(\rot{\gan})\cap R(\rot{\gat})}\frac{\norm{\rot{}H}_{\L{2}(\om)}}{\norm{H}_{\L{2}(\om)}}.
\end{align*}
Moreover, $\lambda_{0,\gat}$ is the smallest positive eigenvalue of
$$\begin{bmatrix}0&\A_{0}^{*}\\\A_{0}&0\end{bmatrix}
=\begin{bmatrix}0&-\div{\gan}\\\grad{\gat}&0\end{bmatrix}$$
and $\lambda_{0,\gat}^2$ is the smallest (positive) eigenvalue of
$$\A_{0}^{*}\A_{0}=-\div{\gan}\grad{\gat}=-\Delta_{\gat}
\quad\text{and}\quad
\A_{0}\A_{0}^{*}=-\grad{\gat}\div{\gan}=-\Diamond_{\gan}.$$
$\lambda_{1,\gat}$ is the smallest positive eigenvalue of
$$\begin{bmatrix}0&\A_{1}^{*}\\\A_{1}&0\end{bmatrix}
=\begin{bmatrix}0&\rot{\gan}\\\rot{\gat}&0\end{bmatrix}$$
and $\lambda_{1,\gat}^2$ is the smallest (positive) eigenvalue of
$$\A_{1}^{*}\A_{1}=\rot{\gan}\rot{\gat}=\square_{\gat}
\quad\text{and}\quad
\A_{1}\A_{1}^{*}=\rot{\gat}\rot{\gan}=\square_{\gan}.$$
\end{theo}

\begin{rem}[variational formulations]
\label{fpmconstrem}
All infima in Theorem \ref{fpmconst} are minima
and the respective minimisers $u_{0,\gat}$, $E_{0,\gan}$, and $E_{1,\gat}$, $H_{1,\gan}$
are the eigenfunctions to the eigenvalues $\lambda_{0,\gat}^2$ and $\lambda_{1,\gat}^2$, i.e, 
$$\lambda_{0,\gat}
=\frac{\norm{\grad{}u_{0,\gat}}_{\L{2}(\om)}}{\norm{u_{0,\gat}}_{\L{2}(\om)}}
=\frac{\norm{\div{}E_{0,\gan}}_{\L{2}(\om)}}{\norm{E_{0,\gan}}_{\L{2}(\om)}},\qquad
\lambda_{1,\gat}
=\frac{\norm{\rot{}E_{1,\gat}}_{\L{2}(\om)}}{\norm{E_{1,\gat}}_{\L{2}(\om)}}
=\frac{\norm{\rot{}H_{1,\gan}}_{\L{2}(\om)}}{\norm{H_{1,\gan}}_{\L{2}(\om)}},$$
\begin{align*}
(-\Delta_{\gat}-\lambda_{0,\gat}^2)u_{0,\gat}&=0,
&
u_{0,\gat}&\in D(\Delta_{\gat})\cap\L{2}_{\gan}(\om)\subset D(\grad{\gat})\cap\L{2}_{\gan}(\om),\\
(-\Diamond_{\gan}-\lambda_{0,\gat}^2)E_{0,\gan}&=0,
&
E_{0,\gan}&\in D(\Diamond_{\gan})\cap R(\grad{\gat})\subset D(\div{\gan})\cap R(\grad{\gat}),\\
(\square_{\gat}-\lambda_{1,\gat}^2)E_{1,\gat}&=0,
&
E_{1,\gat}&\in D(\square_{\gat})\cap R(\rot{\gan})\subset D(\rot{\gat})\cap R(\rot{\gan}),\\
(\square_{\gan}-\lambda_{1,\gat}^2)H_{1,\gan}&=0,
&
H_{1,\gan}&\in D(\square_{\gan})\cap R(\rot{\gat})\subset D(\rot{\gan})\cap R(\rot{\gat}).
\end{align*}
Moreover, the eigenvectors satisfy the variational formulations
\begin{align*}
\forall\,\psi&\in D(\grad{\gat})
&
\scp{\grad{}u_{0,\gat}}{\grad{}\psi}_{\L{2}(\om)}
&=\lambda_{0,\gat}^2\scp{u_{0,\gat}}{\psi}_{\L{2}(\om)},\\
\forall\,\Psi&\in D(\div{\gan})
&
\scp{\div{}E_{0,\gan}}{\div{}\Psi}_{\L{2}(\om)}
&=\lambda_{0,\gat}^2\scp{E_{0,\gan}}{\Psi}_{\L{2}(\om)},\\
\forall\,\Phi&\in D(\rot{\gat})
&
\scp{\rot{}E_{1,\gat}}{\rot{}\Phi}_{\L{2}(\om)}
&=\lambda_{1,\gat}^2\scp{E_{1,\gat}}{\Phi}_{\L{2}(\om)},\\
\forall\,\Theta&\in D(\rot{\gan})
&
\scp{\rot{}H_{1,\gan}}{\rot{}\Theta}_{\L{2}(\om)}
&=\lambda_{1,\gat}^2\scp{H_{1,\gan}}{\Theta}_{\L{2}(\om)}.
\end{align*}
\end{rem}

\begin{rem}
We emphasise that Lemma \ref{lemconstevcomplex} provides results for the vector Laplacian
$$\A_{0}\A_{0}^{*}+\A_{1}^{*}\A_{1}
=-\Diamond_{\gan}+\square_{\gat}
=-\grad{\gat}\div{\gan}+\rot{\gan}\rot{\gat},$$
which has been recently discussed in, e.g., \cite{zbMATH07008886}.
\end{rem}

\subsubsection{Known Results for the Constants in 3D}

Let us summarise and cite some recent results from \cite{paulymaxconst0,paulymaxconst1,paulymaxconst2,paulymaxconst3}
about the Poincar\'e-Friedrichs type constants, i.e.,
about the Poincar\'e-Friedrichs constants $\lambda_{0,\gat}$
and the Maxwell constants $\lambda_{1,\gat}$.

\begin{theo}[Poincar\'e-Friedrichs/Maxwell constants in 3D]
\label{stateoftheartconstest}
For $c_{\ell,\gat}=1/\lambda_{\ell,\gat}$ the following holds:
\begin{itemize}
\item[\bf(i)]
The Poincar\'e-Friedrichs constants depend monotonically on the boundary conditions, i.e., 
$$\emptyset\neq\widetilde\ga_{\tau}\subset\gat
\qimpl c_{0,\gat}\leq c_{0,\widetilde\ga_{\tau}}.$$
\item[\bf(ii)]
The Friedrichs constant is always smaller than the Poincar\'e constant, i.e., 
$$c_{0,\ga}\leq c_{0,\emptyset},$$
where $c_{0,\ga}$ is the classical Friedrichs constant
and $c_{0,\emptyset}$ is the classical Poincar\'e constant.
Moreover, $\lambda_{0,\ga}$ is usually called the first Dirichlet-Laplace eigenvalue 
and $\lambda_{0,\emptyset}$ is usually called the second Neumann-Laplace eigenvalue.
\item[\bf(iii)]
$c_{0,\ga}\leq \text{\rm diam}(\om)/\pi$
\item[\bf(iv)]
$c_{0,\gat}=c_{2,\gan}$
\item[\bf(v)]
$c_{1,\gat}=c_{1,\gan}$
\item[\bf(vi)]
$c_{0,\ga}\leq c_{0,\gat}\leq c_{0,1,\gat}=\max\{c_{0,\gat},c_{1,\gat}\}$
\item[\bf(vii)]
If $\om$ is convex, then $c_{0,\ga}\leq c_{0,\emptyset}\leq \text{\rm diam}(\om)/\pi$.
\item[\bf(viii)]
If $\om$ is convex, then $c_{1,\ga}=c_{1,\emptyset}\leq c_{0,\emptyset}\leq \text{\rm diam}(\om)/\pi$.
\item[\bf(ix)]
If $\om$ is convex, then 
$c_{0,\ga}\leq c_{0,1,\ga}=\max\{c_{0,\ga},c_{1,\ga}\}\leq c_{0,\emptyset}\leq \text{\rm diam}(\om)/\pi$.
\item[\bf(ix')]
If $\om$ is convex, then 
$c_{0,\ga}\leq c_{0,1,\emptyset}=\max\{c_{0,\emptyset},c_{1,\emptyset}\}=c_{0,\emptyset}\leq \text{\rm diam}(\om)/\pi$.
\end{itemize}
\end{theo}

\begin{rem}
\label{stateoftheartconstestrem}
To the best of our knowledge, it is an open question whether or not
$$c_{0,\gat}\leq c_{1,\gat}
\qquad\text{or at least}\qquad
c_{0,\ga}\leq c_{1,\ga}$$
holds in general.
\end{rem}

\subsection{Other Complexes and Constants}
\label{ocplxcst}

So far, we have discussed the de Rham complex \eqref{derhamcomplex} in 3D.
While in higher dimensions $N\geq4$ the situation is
very similar to the 3D case (but the adjoint of $\rot{\gat}$
is no longer a rotation itself), the situations in
1D and 2D are much simpler.
Moreover, similar to the 3D-de Rham complex \eqref{derhamcomplex},
other important complexes of shape \eqref{genpdcomplex} fit nicely into our general fa-toolbox
and, therefore, can be handled with our theory,
see also \cite{paulyapostfirstordergen,paulydivcurl} for details.

\subsubsection{1D-de Rham Complex, Laplace and Maxwell Constants in 1D}
\label{ocplxcst1DdR}

In 1D the domain $\om$ is an interval and 
we have just one operator $\A_{0}=\grad{\gat}=(\,\cdot\,)_{\gat}'$
with adjoint $\A_{0}^{*}=-\div{\gan}=-(\,\cdot\,)_{\gan}'$, i.e., 
the complex \eqref{genpdcomplex}, compare to \eqref{derhamcomplex}, reads
\begin{align*}
\begin{CD}
\H{1}{\gat}(\om)=\H{}{\gat}(\grad{},\om) @> \A_{0}=\grad{\gat}=(\,\cdot\,)_{\gat}' >>
\L{2}(\om),
\end{CD}\\
\begin{CD}
\L{2}(\om) @< \A_{0}^{*}=-\div{\gan}=-(\,\cdot\,)_{\gan}' <<
\H{}{\gan}(\div{},\om)=\H{1}{\gan}(\om).
\end{CD}
\end{align*}
Hence, just the Laplacians $\Delta_{\gat}=\div{\gan}\grad{\gat}=(\,\cdot\,)_{\gat}''$ 
and $\Diamond_{\gan}=\grad{\gat}\div{\gan}=(\,\cdot\,)_{\gan}''$ 
exist and there are no Maxwell operators.
The crucial compact embedding \eqref{cptembcrucial}
is simply Rellich's selection theorem, compare to Theorem \ref{weckst}.
Moreover, here in the 1D case we have
\begin{align*}
\lambda_{0,\gat}
&=\inf_{0\neq u\in\H{1}{\gat}(\om)\cap\L{2}_{\gan}(\om)}\frac{\norm{\grad{}u}_{\L{2}(\om)}}{\norm{u}_{\L{2}(\om)}}
=\inf_{0\neq u\in\H{1}{\gat}(\om)\cap\L{2}_{\gan}(\om)}\frac{\norm{u'}_{\L{2}(\om)}}{\norm{u}_{\L{2}(\om)}}\\
&=\underbrace{\inf_{0\neq E\in \H{}{\gan}(\div{},\om)\cap R(\grad{\gat})}
\frac{\norm{\div{}E}_{\L{2}(\om)}}{\norm{E}_{\L{2}(\om)}}}_{=\lambda_{2,\gan}}
=\inf_{0\neq E\in\H{1}{\gan}(\om)\cap\L{2}_{\gat}(\om)}\frac{\norm{E'}_{\L{2}(\om)}}{\norm{E}_{\L{2}(\om)}}
=\lambda_{0,\gan},
\end{align*}
i.e., it is sufficient to compute the eigenvalues $\lambda_{0,\gat}$,
and we can also give a meaning to $\lambda_{2,\gan}$. Thus
$$\lambda_{0,\gat}
=\lambda_{0,\gan}
=\lambda_{2,\gan}
=\frac{1}{c_{2,\gan}}
=\frac{1}{c_{0,\gan}}
=\frac{1}{c_{0,\gat}}.$$
Note that 
\begin{align*}
\lambda_{0,\gat}
=\frac{\norm{\grad{}u_{0,\gat}}_{\L{2}(\om)}}{\norm{u_{0,\gat}}_{\L{2}(\om)}}
=\frac{\norm{u_{0,\gat}'}_{\L{2}(\om)}}{\norm{u_{0,\gat}}_{\L{2}(\om)}}
=\frac{\norm{\div{}E_{0,\gan}}_{\L{2}(\om)}}{\norm{E_{0,\gan}}_{\L{2}(\om)}}
=\frac{\norm{E_{0,\gan}'}_{\L{2}(\om)}}{\norm{E_{0,\gan}}_{\L{2}(\om)}}
=\lambda_{0,\gan}.
\end{align*}

Theorem \ref{stateoftheartconstest} turns to:

\begin{cor}[Poincar\'e-Friedrichs/Maxwell constants in 1D]
\label{stateoftheartconstestcor1D}
For $c_{\ell,\gat}=1/\lambda_{\ell,\gat}$ the following holds:
\begin{itemize}
\item[\bf(i)]
$\emptyset\neq\widetilde\ga_{\tau}\subset\gat
\;\impl\;c_{0,\gat}\leq c_{0,\widetilde\ga_{\tau}}$
\item[\bf(ii)]
$c_{0,\ga}=c_{0,\emptyset}\leq \text{\rm diam}(\om)/\pi$
\item[\bf(iii)]
$c_{0,\ga}\leq c_{0,\gat}=c_{0,\gan}$
\item[\bf(iv)]
There is no $c_{1,\gat}$, but $c_{2,\gan}=c_{0,\gan}=c_{0,\gat}$.
\end{itemize}
\end{cor}

\subsubsection{2D-de Rham Complex, Laplace and Maxwell Constants in 2D}
\label{ocplxcst2DdR}

In 2D there are just the two operators 
$\A_{0}=\grad{\gat}$ and $\A_{1}=\rot{\gat}=\div{\gat}R$
with adjoints $\A_{0}^{*}=-\div{\gan}$ and $\A_{1}^{*}=\rotv{\gan}=R\grad{\gan}$, where 
$$\rot{}E=\div{}RE=\p_{1}E_{2}-\p_{2}E_{1},\qquad
\rotv{}u=R\grad{}u=\begin{bmatrix}\p_{2}u\\-\p_{1}u\end{bmatrix},\qquad
R=\begin{bmatrix}0&1\\-1&0\end{bmatrix},$$
and the complex \eqref{genpdcomplex}, compare to \eqref{derhamcomplex}, reads
\begin{align*}
\begin{CD}
\H{1}{\gat}(\om)=\H{}{\gat}(\grad{},\om) @> \A_{0}=\grad{\gat} >>
\H{}{\gat}(\rot{},\om)=\H{}{\gat}(\div{},\om)R @> \A_{1}=\rot{\gat} >>
\L{2}(\om),
\end{CD}\\
\begin{CD}
\L{2}(\om) @< \A_{0}^{*}=-\div{\gan} <<
\H{}{\gan}(\div{},\om) @< \A_{1}^{*}=\rotv{\gan} <<
\H{}{\gan}(\rotv{},\om)=\H{1}{\gan}(\om).
\end{CD}
\end{align*}
Hence, we have the Laplacian $\Delta_{\gat}=\div{\gan}\grad{\gat}$ and
$\Diamond_{\gan}=\grad{\gat}\div{\gan}$, as well as the second order Maxwell operators
(related to the 3D notations)
\begin{align*}
\square_{\gat}&=\rotv{\gan}\rot{\gat}=R\grad{\gan}\div{\gat}R=R\Diamond_{\gat}R,\\
\square_{\gan}&=\rot{\gat}\rotv{\gan}=\div{\gat}RR\grad{\gan}=-\div{\gat}\grad{\gan}=-\Delta_{\gan}.
\end{align*}
By Lemma \ref{compemblem}
the crucial compact embedding \eqref{cptembcrucial}
is just Rellich's selection theorem, compare to Theorem \ref{weckst}.
Moreover, here in the 2D case we have
\begin{align*}
\lambda_{0,\gat}
&=\inf_{0\neq u\in\H{1}{\gat}(\om)\cap\L{2}_{\gan}(\om)}\frac{\norm{\grad{}u}_{\L{2}(\om)}}{\norm{u}_{\L{2}(\om)}}
=\inf_{0\neq u\in\H{1}{\gat}(\om)\cap\L{2}_{\gan}(\om)}\frac{\norm{\rotv{}u}_{\L{2}(\om)}}{\norm{u}_{\L{2}(\om)}}\\
&=\underbrace{\inf_{0\neq E\in \H{}{\gan}(\div{},\om)\cap R(\grad{\gat})}
\frac{\norm{\div{}E}_{\L{2}(\om)}}{\norm{E}_{\L{2}(\om)}}}_{=\lambda_{2,\gan}}
\overset{E\leadsto RE}{=}
\inf_{0\neq E\in \H{}{\gan}(\rot{},\om)\cap R(\rotv{\gat})}\frac{\norm{\rot{}E}_{\L{2}(\om)}}{\norm{E}_{\L{2}(\om)}}
=\lambda_{1,\gan},
\end{align*}
i.e., it is sufficient to compute the eigenvalues $\lambda_{0,\gat}$,
and we can also give a meaning to $\lambda_{2,\gan}$. Thus
$$\lambda_{0,\gat}
=\lambda_{1,\gan}
=\lambda_{2,\gan}
=\frac{1}{c_{2,\gan}}
=\frac{1}{c_{1,\gan}}
=\frac{1}{c_{0,\gat}}.$$
Note that 
\begin{align*}
\lambda_{0,\gat}
=\frac{\norm{\grad{}u_{0,\gat}}_{\L{2}(\om)}}{\norm{u_{0,\gat}}_{\L{2}(\om)}}
=\frac{\norm{\rotv{}u_{0,\gat}}_{\L{2}(\om)}}{\norm{u_{0,\gat}}_{\L{2}(\om)}}
=\frac{\norm{\div{}E_{0,\gan}}_{\L{2}(\om)}}{\norm{E_{0,\gan}}_{\L{2}(\om)}}
=\frac{\norm{\rot{}E_{1,\gan}}_{\L{2}(\om)}}{\norm{E_{1,\gan}}_{\L{2}(\om)}}
=\lambda_{1,\gan},
\end{align*}
i.e., in our 3D-notation $H_{1,\gat}=u_{0,\gat}$ and $E_{0,\gan}=RE_{1,\gan}$.
Theorem \ref{stateoftheartconstest} turns to:

\begin{cor}[Poincar\'e-Friedrichs/Maxwell constants in 2D]
\label{stateoftheartconstestcor2D}
For $c_{\ell,\gat}=1/\lambda_{\ell,\gat}$ the following holds:
\begin{itemize}
\item[\bf(i)]
The Poincar\'e-Friedrichs constants depend monotonically on the boundary conditions, i.e., 
$$\emptyset\neq\widetilde\ga_{\tau}\subset\gat
\qimpl c_{0,\gat}\leq c_{0,\widetilde\ga_{\tau}}.$$
\item[\bf(ii)]
The Friedrichs constant is always smaller than the Poincar\'e constant, i.e., 
$c_{0,\ga}\leq c_{0,\emptyset}$.
\item[\bf(iii)]
$c_{0,\ga}\leq \text{\rm diam}(\om)/\pi$
\item[\bf(iv)]
$c_{0,\ga}\leq c_{0,\gat}=c_{1,\gan}=c_{2,\gan}\leq c_{0,1,\gat}
=\max\{c_{0,\gat},c_{1,\gat}\}=\max\{c_{0,\gat},c_{0,\gan}\}$
\item[\bf(v)]
If $\om$ is convex, then $c_{0,\ga}\leq c_{0,\emptyset}\leq \text{\rm diam}(\om)/\pi$.
\end{itemize}
\end{cor}

\subsubsection{ND-de Rham Complex, Laplace and Maxwell Constants in ND}
\label{ocplxcstNDdR}

In ND differential forms generalize suitably functions and vector fields used for $N=1,2,3$.
The de Rham complex of \eqref{genpdcomplex} in ND, compare to \eqref{derhamcomplex},
consists of $N$ differential operators $\A_{q}:=\A_{q,\gat}:=\ed{q,\gat}$, $q=0,\dots,N-1$,
with adjoints $\A_{q}^{*}=\A_{q,\gat}^{*}=-\cd{q+1,\gan}$ acting on alternating $q$ resp. $(q+1)$-forms, i.e.,
\begin{align*}
\begin{CD}
\cdots @> \cdots >>
\H{}{\gat}(\ed{q},\om) @> \A_{q}=\ed{q,\gat} >>
\H{}{\gat}(\ed{q+1},\om) @> \A_{q+1}=\ed{q+1,\gat} >>
\cdots,
\end{CD}\\
\begin{CD}
\cdots @< \A_{q-1}^{*}=-\cd{q,\gan} <<
\H{}{\gan}(\cd{q},\om) @< \A_{q}^{*}=-\cd{q+1,\gan} <<
\H{}{\gan}(\cd{q+1},\om) @< \cdots <<
\cdots,
\end{CD}
\end{align*}
see, e.g.,
\cite{zbMATH01329309,hiptmairfemmax,arnoldfalkwintherfemec,zbMATH05130995,zbMATH05696861,zbMATH07045589,zbMATH06825239,zbMATH06894075}
for details about the complex and numerical applications.
Hence, the second order ``Laplace'' and ``Maxwell'' operators are simply
$$\A_{q}^{*}\A_{q}=-\cd{q+1,\gan}\ed{q,\gat},\qquad
\A_{q}\A_{q}^{*}=-\ed{q,\gat}\cd{q+1,\gan},$$
and for the constants and eigenvalues $c_{q,\gat}=1/\lambda_{q,\gat}$ we have
\begin{align*}
\forall\,E&\in D(\cA_{q})=D(\ed{q,\gat})\cap R(\cd{q+1,\gan})
&
\norm{E}_{\L{2,q}(\om)}&\leq c_{q,\gat}\norm{\ed{q}E}_{\L{2,q+1}(\om)},\\
\forall\,H&\in D(\cA_{q}^{*})=D(\cd{q+1,\gan})\cap R(\ed{q,\gat})
&
\norm{H}_{\L{2,q+1}(\om)}&\leq c_{q,\gat}\norm{\cd{q+1}H}_{\L{2,q}(\om)}.
\end{align*}
The crucial compact embeddings \eqref{cptembcrucial} are given by the following theorem
from \cite[Theorem 4.9]{bauerpaulyschomburgmaxcompweakliprN}
or \cite[Theorem 4.8]{bauerpaulyschomburgmaxcompweakliprNarxiv}.

\begin{theo}[Weck's selection theorem]
\label{weckstdf}
The embeddings
$$D(\A_{q})\cap D(\A_{q-1}^{*})
=\H{}{\gat}(\ed{q},\om)\cap\H{}{\gan}(\cd{q},\om)
\emb\L{2,q}(\om)$$
are compact.
\end{theo}

The general theory, the definition $\cd{q+1,\gan}=\pm\star_{N-q}\ed{N-q-1,\gan}\star_{q+1}$,
where $\star_{q}$ is the Hodge star-operator, 
and the substitution $E=\star_{q+1}H$ show again a symmetry for the eigenvalues, i.e.,
\begin{align*}
\lambda_{q,\gat}
&=\inf_{0\neq E\in D(\cA_{q,\gat})=D(\ed{q,\gat})\cap R(\cd{q+1,\gan})}
\frac{\norm{\ed{q}E}_{\L{2,q+1}(\om)}}{\norm{E}_{\L{2,q}(\om)}}\\
&=\inf_{0\neq H\in D(\cA_{q,\gat}^{*})=D(\cd{q+1,\gan})\cap R(\ed{q,\gat})}
\frac{\norm{\cd{q+1}H}_{\L{2,q}(\om)}}{\norm{H}_{\L{2,q+1}(\om)}}\\
&=\inf_{0\neq E\in D(\cA_{N-q-1,\gan})=D(\ed{N-q-1,\gan})\cap R(\cd{N-q,\gat})}
\frac{\norm{\ed{N-q-1}E}_{\L{2,N-q}(\om)}}{\norm{E}_{\L{2,N-q-1}(\om)}}
=\lambda_{N-q-1,\gan}.
\end{align*}
Therefore, we obtain the relations
$$\frac{1}{c_{q,\gat}}
=\lambda_{q,\gat}
=\lambda_{N-q-1,\gan}
=\frac{1}{c_{N-q-1,\gan}},$$
which also confirm (for $N=1,2$) the results of Sections \ref{ocplxcst1DdR} and \ref{ocplxcst2DdR}.
Using the notations from the 3D case we define
$$\frac{1}{c_{q-1,q,\gat}^2}
:=\lambda_{q-1,q,\gat}^2
:=\inf\frac{\norm{\ed{q}E}_{\L{2,q+1}(\om)}^2+\norm{\cd{q}E}_{\L{2,q-1}(\om)}^2}{\norm{E}_{\L{2,q}(\om)}^2},$$
where the infimum is taken over all
$0\neq E\in D(\A_{q,\gat})\cap D(\A_{q-1,\gat}^{*})=D(\ed{q,\gat})\cap D(\cd{q,\gan})$
being perpendicular to the respective generalized Dirichlet-Neumann forms
$N(\A_{q,\gat})\cap N(\A_{q-1,\gat}^{*})$.
Theorem \ref{stateoftheartconstest} turns to:

\begin{theo}[Poincar\'e-Friedrichs/Maxwell constants in ND]
\label{stateoftheartconstesttheoND}
For $c_{q,\gat}=1/\lambda_{q,\gat}$ the following holds:
\begin{itemize}
\item[\bf(i)]
The Poincar\'e-Friedrichs constants depend monotonically on the boundary conditions, i.e., 
$$\emptyset\neq\widetilde\ga_{\tau}\subset\gat
\qimpl c_{0,\gat}\leq c_{0,\widetilde\ga_{\tau}}.$$
\item[\bf(ii)]
The Friedrichs constant is always smaller than the Poincar\'e constant, i.e., 
$c_{0,\ga}\leq c_{0,\emptyset}$.
\item[\bf(iii)]
$c_{0,\ga}\leq \text{\rm diam}(\om)/\pi$
\item[\bf(iv)]
$c_{q,\gat}=c_{N-q-1,\gan}$
\item[\bf(v)]
$c_{q-1,q,\gat}=\max\{c_{q-1,\gat},c_{q,\gat}\}$
\item[\bf(vi)]
If $\om$ is topologically trivial, then $c_{0,\ga}\leq c_{q-1,q,\gat}$.
\item[\bf(vii)]
If $\om$ is convex, then $c_{0,\ga}\leq c_{0,\emptyset}\leq \text{\rm diam}(\om)/\pi$.
\item[\bf(viii)]
If $\om$ is convex, then 
$c_{q,\ga}\,,\,c_{q,\emptyset}\leq c_{0,\emptyset}\leq \text{\rm diam}(\om)/\pi$.
\item[\bf(ix)]
If $\om$ is convex, then 
$c_{0,\ga}\leq c_{q-1,q,\ga}=\max\{c_{q-1,\ga},c_{q,\ga}\}\leq c_{0,\emptyset}\leq \text{\rm diam}(\om)/\pi$.
\item[\bf(ix')]
If $\om$ is convex, then 
$c_{0,\ga}\leq c_{q-1,q,\emptyset}=\max\{c_{q-1,\emptyset},c_{q,\emptyset}\}\leq c_{0,\emptyset}\leq \text{\rm diam}(\om)/\pi$.
\end{itemize}
\end{theo}

For proofs and details see \cite{paulymaxconst3}. To show (vi),
for which an argument is missing in \cite{paulymaxconst3}, 
let $I$ be a multi-index of order $q$ and
let $u\in\H{1}{\ga}(\om)=\H{}{\ga}(\grad{},\om)$. Then
$$E:=u\ed{}x^{I}\in\H{1,q}{\ga}(\om)
\subset\H{}{\gat}(\ed{q},\om)\cap\H{}{\gan}(\cd{q},\om)$$
and we have by approximation and the triviality of Dirichlet-Neumann forms
\begin{align*}
\norm{u}_{\L{2}(\om)}
=\norm{E}_{\L{2,q}(\om)}
&\leq c_{q-1,q,\gat}\big(\norm{\ed{q}E}_{\L{2,q+1}(\om)}^2+\norm{\cd{q}E}_{\L{2,q-1}(\om)}^2\big)^{1/2}\\
&=c_{q-1,q,\gat}\big(\sum_{\ell=1}^{N}\norm{\p_{\ell}E}_{\L{2,q}(\om)}^2\big)^{1/2}
=c_{q-1,q,\gat}\norm{\grad{}u}_{\L{2}(\om)},
\end{align*}
showing $c_{0,\ga}\leq c_{q-1,q,\gat}$.

\subsubsection{3D-Elasticity Complex}
\label{ocplxcst3Dela}

The complex (involving vector as well as symmetric tensor fields)
\begin{align*}
\begin{CD}
\H{}{\gat}(\text{\rm Grad},\om) @> \A_{0}=\text{\rm sym\,Grad}_{\gat} >>
\H{}{\gat}(\text{\rm Curl\,Curl}^{\top},\mathbb{S},\om) @> \A_{1}=\text{\rm Curl\,Curl}^{\top}_{\gat} >>
\H{}{\gat}(\text{\rm Div},\mathbb{S},\om) @> \A_{2}=\text{\rm Div}_{\gat} >>
\L{2}(\om),
\end{CD}\\
\begin{CD}
\L{2}(\om) @< \A_{0}^{*}=-\text{\rm Div}_{\gan} <<
\H{}{\gan}(\text{\rm Div},\mathbb{S},\om) @< \A_{1}^{*}=\text{\rm Curl\,Curl}^{\top}_{\gan} <<
\H{}{\gan}(\text{\rm Curl\,Curl}^{\top},\mathbb{S},\om) @< \A_{2}^{*}=\text{\rm sym\,Grad}_{\gan} <<
\H{}{\gan}(\text{\rm Grad},\om),
\end{CD}
\end{align*}
is related to elasticity,
see, e.g., 
\cite{zbMATH05696861,zbMATH01837365,zbMATH02084083,zbMATH05190885,zbMATH05763606,zbMATH05761760,zbMATH06295983,zbMATH05492686,zbMATH06051824}
for details about the complex and numerical applications.
Note that, indeed, by Korn's inequality the regularity
$$D(\A_{0})=D(\text{\rm sym\,Grad}_{\gat})=\H{}{\gat}(\text{\rm sym\,Grad},\om)=\H{}{\gat}(\text{\rm Grad},\om)=\H{1}{\gat}(\om)$$
holds. The ``second order Laplace and Maxwell'' operators are given by
\begin{align*}
\A_{0}^{*}\A_{0}&=-\text{\rm Div}_{\gan}\text{\rm sym\,Grad}_{\gat},
&
\A_{0}\A_{0}^{*}&=-\text{\rm sym\,Grad}_{\gat}\text{\rm Div}_{\gan},\\
\A_{1}^{*}\A_{1}&=\text{\rm Curl\,Curl}^{\top}_{\gan}\text{\rm Curl\,Curl}^{\top}_{\gat},
&
\A_{1}\A_{1}^{*}&=\text{\rm Curl\,Curl}^{\top}_{\gat}\text{\rm Curl\,Curl}^{\top}_{\gan},
\end{align*}
and for the constants and eigenvalues $c_{\ell,\gat}^{\text{\sf ela}}=1/\lambda_{\ell,\gat}^{\text{\sf ela}}$ we have
\begin{align*}
\forall\,v&\in D(\cA_{0})=D(\text{\rm sym\,Grad}_{\gat})\cap R(\text{\rm Div}_{\gan})
&
\norm{v}_{\L{2}(\om)}&\leq c_{0,\gat}^{\text{\sf ela}}\norm{\text{\rm sym\,Grad\,}v}_{\L{2}(\mathbb{S},\om)},\\
\forall\,S&\in D(\cA_{0}^{*})=D(\text{\rm Div}_{\gan})\cap R(\text{\rm sym\,Grad}_{\gat})
&
\norm{S}_{\L{2}(\mathbb{S},\om)}&\leq c_{0,\gat}^{\text{\sf ela}}\norm{\text{\rm Div\,}S}_{\L{2}(\om)},\\
\forall\,S&\in D(\cA_{1})=D(\text{\rm Curl\,Curl}^{\top}_{\gat})\cap R(\text{\rm Curl\,Curl}^{\top}_{\gan})
&
\norm{S}_{\L{2}(\mathbb{S},\om)}&\leq c_{1,\gat}^{\text{\sf ela}}\norm{\text{\rm Curl\,Curl}^{\top}S}_{\L{2}(\mathbb{S},\om)},\\
\forall\,T&\in D(\cA_{1}^{*})=D(\text{\rm Curl\,Curl}^{\top}_{\gan})\cap R(\text{\rm Curl\,Curl}^{\top}_{\gat})
&
\norm{T}_{\L{2}(\mathbb{S},\om)}&\leq c_{1,\gat}^{\text{\sf ela}}\norm{\text{\rm Curl\,Curl}^{\top}T}_{\L{2}(\mathbb{S},\om)},\\
\forall\,T&\in D(\cA_{2})=D(\text{\rm Div}_{\gat})\cap R(\text{\rm sym\,Grad}_{\gan})
&
\norm{T}_{\L{2}(\mathbb{S},\om)}&\leq c_{2,\gat}^{\text{\sf ela}}\norm{\text{\rm Div\,}T}_{\L{2}(\om)},\\
\forall\,v&\in D(\cA_{2}^{*})=D(\text{\rm sym\,Grad}_{\gan})\cap R(\text{\rm Div}_{\gat})
&
\norm{v}_{\L{2}(\om)}&\leq c_{2,\gat}^{\text{\sf ela}}\norm{\text{\rm sym\,Grad\,}v}_{\L{2}(\mathbb{S},\om)}.
\end{align*}
As in the 3D Maxwell case the last two inequalities are already given by the first two.
Note that
\begin{align*}
N(\text{\rm sym\,Grad}_{\gat})
&=\begin{cases}\{0\}&\text{if }\gat\neq\emptyset,\\
\text{\sf RM}&\text{if }\gat=\emptyset,\end{cases}\\
R(\text{\rm Div}_{\gan})=N(\text{\rm sym\,Grad}_{\gat})^{\bot_{\L{2}(\om)}}
&=\begin{cases}\L{2}(\om)&\text{if }\gan\neq\ga,\\
\L{2}(\om)\cap\text{\sf RM}^{\bot_{\L{2}(\om)}}&\text{if }\gan=\ga,\end{cases}
\end{align*}
where $\text{\sf RM}$ denotes the space of global rigid motions.
The crucial compact embeddings \eqref{cptembcrucial} 
have recently been proved in \cite{paulyzulehnerela}.

\begin{theo}[selection theorems for elasticity]
\label{pzstela}
The embedding
$$D(\A_{1})\cap D(\A_{0}^{*})
=\H{}{\gat}(\text{\rm Curl\,Curl}^{\top},\mathbb{S},\om)\cap\H{}{\gan}(\text{\rm Div},\mathbb{S},\om)
\emb\L{2}(\mathbb{S},\om)$$
is compact.
\end{theo}

Note that by the latter theorem the embedding 
$$D(\A_{2})\cap D(\A_{1}^{*})
=\H{}{\gat}(\text{\rm Div},\mathbb{S},\om)\cap\H{}{\gan}(\text{\rm Curl\,Curl}^{\top},\mathbb{S},\om)
\emb\L{2}(\mathbb{S},\om)$$
is compact as well by interchanging $\gat$ and $\gan$.

Similar to the 3D Maxwell case we get the following theorem, cf. Theorem \ref{stateoftheartconstest}.

\begin{theo}[Poincar\'e-Friedrichs type constants for elasticity]
\label{stateoftheartconstestela}
For $c_{\ell,\gat}^{\text{\sf ela}}=1/\lambda_{\ell,\gat}^{\text{\sf ela}}$ the following holds:
\begin{itemize}
\item[\bf(i)]
The Poincar\'e-Friedrichs type constants depend monotonically on the boundary conditions, i.e., 
$$\emptyset\neq\widetilde\ga_{\tau}\subset\gat
\qimpl c_{0,\gat}^{\text{\sf ela}}\leq c_{0,\widetilde\ga_{\tau}}^{\text{\sf ela}}.$$
\item[\bf(ii)]
$c_{0,\gat}^{\text{\sf ela}}=c_{2,\gan}^{\text{\sf ela}}$ and $c_{1,\gat}^{\text{\sf ela}}=c_{1,\gan}^{\text{\sf ela}}$.
\end{itemize}
\end{theo}

\begin{rem}[Poincar\'e-Friedrichs type constants for elasticity]
\label{stateoftheartconstestelarem}
The Poincar\'e-Friedrichs type constants of the elasticity complex $c_{0,\gat}^{\text{\sf ela}}=c_{2,\gan}^{\text{\sf ela}}$
are related to the classical Poincar\'e-Friedrichs constants $c_{0,\gat}=c_{2,\gan}$ by Korn's inequality, i.e.,
\begin{align*}
\forall\,v&\in D(\cA_{0})=\underbrace{D(\text{\rm sym\,Grad}_{\gat})}_{=\H{1}{\gat}(\om)}\cap R(\text{\rm Div}_{\gan})
&
\norm{\text{\rm Grad\,}v}_{\L{2}(\om)}&\leq c_{\text{\sf k},\gat}\norm{\text{\rm sym\,Grad\,}v}_{\L{2}(\mathbb{S},\om)}.
\end{align*}
More precisely, 
$$c_{2,\gan}^{\text{\sf ela}}=c_{0,\gat}^{\text{\sf ela}}
\leq c_{\text{\sf k},\gat}c_{0,\gat}=c_{\text{\sf k},\gat}c_{2,\gan}$$
holds, as for all $v\in D(\cA_{0})$
$$\norm{v}_{\L{2}(\om)}
\leq c_{0,\gat}\norm{\text{\rm Grad\,}v}_{\L{2}(\om)}
\leq c_{\text{\sf k},\gat}c_{0,\gat}\norm{\text{\rm sym\,Grad\,}v}_{\L{2}(\mathbb{S},\om)}.$$
In particular, for $\gat=\ga$ we know $c_{\text{\sf k},\ga}\leq\sqrt{2}$, 
see \cite{bauerpaulytangnormkorn,zbMATH06663664}, which shows by Theorem \ref{stateoftheartconstest}
$$c_{2,\emptyset}^{\text{\sf ela}}
=c_{0,\ga}^{\text{\sf ela}}
\leq\sqrt{2}c_{0,\ga}\leq\sqrt{2}\min\{c_{0,\gat},c_{0,\emptyset},\frac{\text{\rm diam}(\om)}{\pi}\}
\leq\frac{\sqrt{2}}{\pi}\text{\rm diam}(\om).$$
\end{rem}

\subsubsection{3D-Biharmonic Complex (div\,Div-complex)}
\label{ocplxcst3Dbih}

The complex  (involving scalar as well as symmetric and deviatoric tensor fields)
\begin{align*}
\begin{CD}
\H{2}{\gat}(\om)=\H{}{\gat}(\text{\rm Grad\,grad},\om) @> \A_{0}=\text{\rm Grad\,grad}_{\gat} >>
\H{}{\gat}(\text{\rm Curl},\mathbb{S},\om) @> \A_{1}=\text{\rm Curl}_{\gat} >>
\H{}{\gat}(\text{\rm Div},\mathbb{T},\om) @> \A_{2}=\text{\rm Div}_{\gat} >>
\L{2}(\om),
\end{CD}\\
\begin{CD}
\L{2}(\om) @< \A_{0}^{*}=\text{\rm div\,Div}_{\gan} <<
\H{}{\gan}(\text{\rm div\,Div},\mathbb{S},\om) @< \A_{1}^{*}=\text{\rm sym\,Curl}_{\gan} <<
\H{}{\gan}(\text{\rm sym\,Curl},\mathbb{T},\om) @< \A_{2}^{*}=-\text{\rm dev\,Grad}_{\gan} <<
\H{}{\gan}(\text{\rm Grad},\om),
\end{CD}
\end{align*}
arises in general relativity and for the biharmonic equation,
see, e.g., \cite{paulyzulehnerbiharmonic} for details and, e.g.,
\cite{zbMATH06440388,zbMATH06614110,zbMATH06898486,zbMATH06804027,zbMATH06869171} 
for numerical applications.
Note that, indeed, similar to using Korn's inequality in the latter section, the regularity
$$D(\A_{2}^{*})=D(\text{\rm dev\,Grad}_{\gan})=\H{}{\gan}(\text{\rm dev\,Grad},\om)=\H{}{\gan}(\text{\rm Grad},\om)=\H{1}{\gan}(\om)$$
holds, cf. \cite[Lemma 3.2]{paulyzulehnerbiharmonic}. The ``second order Laplace and Maxwell'' operators are given by
\begin{align*}
\A_{0}^{*}\A_{0}&=\text{\rm div\,Div}_{\gan}\text{\rm Grad\,grad}_{\gat},
&
\A_{0}\A_{0}^{*}&=\text{\rm Grad\,grad}_{\gat}\text{\rm div\,Div}_{\gan},\\
\A_{1}^{*}\A_{1}&=\text{\rm sym\,Curl}_{\gan}\text{\rm Curl}_{\gat},
&
\A_{1}\A_{1}^{*}&=\text{\rm Curl}_{\gat}\text{\rm sym\,Curl}_{\gan},\\
\A_{2}^{*}\A_{2}&=-\text{\rm dev\,Grad}_{\gan}\text{\rm Div}_{\gat},
&
\A_{2}\A_{2}^{*}&=-\text{\rm Div}_{\gat}\text{\rm dev\,Grad}_{\gan},
\end{align*}
and for the constants and eigenvalues $c_{\ell,\gat}^{\text{\sf bih}}=1/\lambda_{\ell,\gat}^{\text{\sf bih}}$ we have
\begin{align*}
\forall\,u&\in D(\cA_{0})=D(\text{\rm Grad\,grad}_{\gat})\cap R(\text{\rm div\,Div}_{\gan})
&
\norm{u}_{\L{2}(\om)}&\leq c_{0,\gat}^{\text{\sf bih}}\norm{\text{\rm Grad\,grad\,}u}_{\L{2}(\mathbb{S},\om)},\\
\forall\,S&\in D(\cA_{0}^{*})=D(\text{\rm div\,Div}_{\gan})\cap R(\text{\rm Grad\,grad}_{\gat})
&
\norm{S}_{\L{2}(\mathbb{S},\om)}&\leq c_{0,\gat}^{\text{\sf bih}}\norm{\text{\rm div\,Div\,}S}_{\L{2}(\om)},\\
\forall\,S&\in D(\cA_{1})=D(\text{\rm Curl}_{\gat})\cap R(\text{\rm sym\,Curl}_{\gan})
&
\norm{S}_{\L{2}(\mathbb{S},\om)}&\leq c_{1,\gat}^{\text{\sf bih}}\norm{\text{\rm Curl\,}S}_{\L{2}(\mathbb{T},\om)},\\
\forall\,T&\in D(\cA_{1}^{*})=D(\text{\rm sym\,Curl}_{\gan})\cap R(\text{\rm Curl}_{\gat})
&
\norm{T}_{\L{2}(\mathbb{T},\om)}&\leq c_{1,\gat}^{\text{\sf bih}}\norm{\text{\rm sym\,Curl\,}T}_{\L{2}(\mathbb{S},\om)},\\
\forall\,T&\in D(\cA_{2})=D(\text{\rm Div}_{\gat})\cap R(\text{\rm dev\,Grad}_{\gan})
&
\norm{T}_{\L{2}(\mathbb{T},\om)}&\leq c_{2,\gat}^{\text{\sf bih}}\norm{\text{\rm Div\,}T}_{\L{2}(\om)},\\
\forall\,v&\in D(\cA_{2}^{*})=D(\text{\rm dev\,Grad}_{\gan})\cap R(\text{\rm Div}_{\gat})
&
\norm{v}_{\L{2}(\om)}&\leq c_{2,\gat}^{\text{\sf bih}}\norm{\text{\rm dev\,Grad\,}v}_{\L{2}(\mathbb{T},\om)}.
\end{align*}
We emphasise that this complex is the first non-symmetric one
and we get additional results for the operators involving $\A_{2}$.
Note that
\begin{align*}
N(\text{\rm Grad\,grad}_{\gat})
&=\begin{cases}\{0\}&\text{if }\gat\neq\emptyset,\\
\text{\sf P}^{1}&\text{if }\gat=\emptyset,\end{cases}\\
R(\text{\rm div\,Div}_{\gan})=N(\text{\rm Grad\,grad}_{\gat})^{\bot_{\L{2}(\om)}}
&=\begin{cases}\L{2}(\om)&\text{if }\gan\neq\ga,\\
\L{2}(\om)\cap(\text{\sf P}^{1})^{\bot_{\L{2}(\om)}}&\text{if }\gan=\ga,\end{cases}\\
N(\text{\rm dev\,Grad}_{\gan})
&=\begin{cases}\{0\}&\text{if }\gan\neq\emptyset,\\
\text{\sf RT}&\text{if }\gan=\emptyset,\end{cases}\\
R(\text{\rm Div}_{\gat})=N(\text{\rm dev\,Grad}_{\gan})^{\bot_{\L{2}(\om)}}
&=\begin{cases}\L{2}(\om)&\text{if }\gat\neq\ga,\\
\L{2}(\om)\cap\text{\sf RT}^{\bot_{\L{2}(\om)}}&\text{if }\gat=\ga,\end{cases}
\end{align*}
where $\text{\sf P}^{1}$ denotes the polynomials of order less then $1$
and $\text{\sf RT}$ the space of global Raviart-Thomas vector fields.
The crucial compact embeddings \eqref{cptembcrucial} 
have recently been proved in \cite[Lemma 3.22]{paulyzulehnerbiharmonic}.

\begin{theo}[selection theorems for the biharmonic complex]
\label{pzstbih}
The embeddings
\begin{align*}
D(\A_{1})\cap D(\A_{0}^{*})
=\H{}{\gat}(\text{\rm Curl},\mathbb{S},\om)\cap\H{}{\gan}(\text{\rm div\,Div},\mathbb{S},\om)
&\emb\L{2}(\mathbb{S},\om),\\
D(\A_{2})\cap D(\A_{1}^{*})
=\H{}{\gat}(\text{\rm Div},\mathbb{T},\om)\cap\H{}{\gan}(\text{\rm sym\,Curl},\mathbb{T},\om)
&\emb\L{2}(\mathbb{T},\om)
\end{align*}
are compact.
\end{theo}

Similar to the 3D Maxwell case and the 3D elasticity case  
we get the following result, cf. Theorem \ref{stateoftheartconstest},
Theorem \ref{stateoftheartconstestela}, and Remark \ref{stateoftheartconstestelarem}.

\begin{rem}[Poincar\'e-Friedrichs type constants for the biharmonic complex]
\label{stateoftheartconstestbihrem}
For $c_{\ell,\gat}^{\text{\sf bit}}=1/\lambda_{\ell,\gat}^{\text{\sf bit}}$ the following holds:
\begin{itemize}
\item[\bf(i)]
The Poincar\'e-Friedrichs type constants depend monotonically on the boundary conditions, i.e., 
$$\emptyset\neq\widetilde\ga_{\tau}\subset\gat
\qimpl c_{0,\gat}^{\text{\sf bih}}\leq c_{0,\widetilde\ga_{\tau}}^{\text{\sf bih}}.$$
\item[\bf(ii)]
Due to the lack a symmetry in the biharmonic complex there are no further formulas 
relating $c_{0,\gat}^{\text{\sf bih}}$ to $c_{2,\gan}^{\text{\sf bih}}$ 
or $c_{1,\gat}^{\text{\sf bih}}$ to $c_{1,\gan}^{\text{\sf bih}}$.
\item[\bf(iii)]
As pointed out in Remark \ref{stateoftheartconstestelarem} for the elasticity complex,
there is a similar relation between 
the Poincar\'e-Friedrichs type constants of the biharmonic complex 
$c_{0,\gat}^{\text{\sf bih}}$ and $c_{2,\gat}^{\text{\sf bih}}$
and the classical Poincar\'e-Friedrichs constants $c_{0,\gat}=c_{2,\gan}$ 
by the classical Poincar\'e-Friedrichs estimate and a Korn like inequality, i.e.,
\begin{align*}
\forall\,v&\in D(\cA_{2}^{*})=\underbrace{D(\text{\rm dev\,Grad}_{\gan})}_{=\H{1}{\gan}(\om)}\cap R(\text{\rm Div}_{\gat})
&
\norm{\text{\rm Grad\,}v}_{\L{2}(\om)}&\leq c_{\text{\sf dev},\gan}\norm{\text{\rm dev\,Grad\,}v}_{\L{2}(\mathbb{T},\om)},
\end{align*}
cf. \cite[Lemma 3.2]{paulyzulehnerbiharmonic}. More precisely, 
$c_{0,\gat}^{\text{\sf bih}}\leq c_{0,\gat}^2$ and $c_{2,\gat}^{\text{\sf bih}}\leq c_{\text{\sf dev},\gan}c_{0,\gan}$
hold, as
$$\forall\,v\in D(\cA_{2}^{*})\qquad
\norm{v}_{\L{2}(\om)}
\leq c_{0,\gan}\norm{\text{\rm Grad\,}v}_{\L{2}(\om)}
\leq c_{\text{\sf dev},\gan}c_{0,\gan}\norm{\text{\rm dev\,Grad\,}v}_{\L{2}(\mathbb{T},\om)}.$$
\end{itemize}
\end{rem}

\section{Analytical Examples}
\label{anaex}

In the sequel we will compute all 
Poincar\'e-Friedrichs and Maxwell eigenvalues for the unit cube
in 1D, 2D, and 3D with mixed boundary conditions on
canonical boundary parts. We emphasise that the completeness 
of the respective eigensystems can be shown as in \cite{costabeldauge2019a}.

\subsection{1D}
\label{anaex1D}

Let $\om:=I:=(0,1)$, $\ga=\{0,1\}$,
and $\gat\in P\big(\{0,1\}\big)=\big\{\emptyset,\{0\},\{1\},\ga\big\}$,
and recall Section \ref{ocplxcst1DdR}.
From Appendix \ref{appanaex1D} we see
\begin{align}
\label{constval1D}
c_{0,\ga}=c_{0,\emptyset}=\frac{1}{\pi},\qquad
c_{0,\{0\}}=c_{0,\{1\}}=\frac{2}{\pi}.
\end{align}
Note that from $c_{0,\gat}=c_{0,\gan}$, see Corollary \ref{stateoftheartconstestcor1D}, we already know
$c_{0,\ga}=c_{0,\emptyset}$ and $c_{0,\{0\}}=c_{0,\{1\}}$.

\begin{rem}
\label{stateoftheartconstestrem1D}
Corollary \ref{stateoftheartconstestcor1D} may be verified by this example.
\begin{itemize}
\item[\bf(i)]
$\displaystyle\emptyset\neq\{0\},\{1\}\subset\ga
\qimpl c_{0,\ga}=\frac{1}{\pi}\leq\frac{2}{\pi}=c_{0,\{0\}}=c_{0,\{1\}}$
\item[\bf(ii)]
$\displaystyle c_{0,\ga}=c_{0,\emptyset}=\frac{1}{\pi}=\frac{\text{\rm diam}(\om)}{\pi}$
\end{itemize}
\end{rem}

\subsection{2D}
\label{anaex2D}

Let $\om:=I^2$, $I:=(0,1)$, $\ga=\ol{\ga_{b}\cup\ga_{t}\cup\ga_{l}\cup\ga_{r}}$,
where $\ga_{b}$, $\ga_{t}$, $\ga_{l}$, $\ga_{r}$ are the open 
bottom, top, left, and right boundary parts of $\ga$, respectively,
and $\gat\in P\big(\{\ga_{b},\ga_{t},\ga_{l},\ga_{r}\}\big)$,
and recall Section \ref{ocplxcst2DdR}.
We shall use canonical index notations such as
$$\ga_{b,l}:=\text{int}(\ol{\ga_{b}\cup\ga_{l}}),\quad
\ga_{b,l,t}:=\text{int}(\ol{\ga_{b}\cup\ga_{l}\cup\ga_{t}}).$$
From Appendix \ref{appanaex2D} we see
\begin{align}
\nonumber
c_{0,\emptyset}
&=\frac{1}{\pi},
&
c_{0,\ga_{b,l}}
=c_{0,\ga_{b,r}}
=c_{0,\ga_{t,l}}
=c_{0,\ga_{t,r}}
&=\frac{\sqrt{2}}{\pi},\\
\label{constval2D}
c_{0,\ga_{b}}
=c_{0,\ga_{t}}
=c_{0,\ga_{l}}
=c_{0,\ga_{r}}
&=\frac{2}{\pi},
&
c_{0,\ga_{b,l,r}}
=c_{0,\ga_{t,l,r}}
=c_{0,\ga_{b,t,l}}
=c_{0,\ga_{b,t,r}}
&=\frac{2}{\sqrt{5}\pi},\\
\nonumber
c_{0,\ga_{b,t}}
=c_{0,\ga_{l,r}}
&=\frac{1}{\pi},
&
c_{0,\ga}
&=\frac{1}{\sqrt{2}\pi}.
\end{align}

\begin{rem}
\label{stateoftheartconstestrem2D}
Corollary \ref{stateoftheartconstestcor2D} may be verified by this example.
\begin{itemize}
\item[\bf(i)]
$\displaystyle\emptyset\neq\ga_{b}
\subset\ga_{b,l}
\subset\ga_{b,l,r}
\subset\ga
\qimpl
c_{0,\ga}=\frac{1}{\sqrt{2}\pi}
\leq c_{0,\ga_{b,l,r}}=\frac{2}{\sqrt{5}\pi}
\leq c_{0,\ga_{b,l}}=\frac{2}{\sqrt{2}\pi}
\leq c_{0,\ga_{b}}=\frac{2}{\pi}$
\item[\bf(i')]
$\displaystyle\emptyset\neq\ga_{l}
\subset\ga_{l,r}
\subset\ga_{b,l,r}
\subset\ga
\qimpl
c_{0,\ga}=\frac{1}{\sqrt{2}\pi}
\leq c_{0,\ga_{b,l,r}}=\frac{2}{\sqrt{5}\pi}
\leq c_{0,\ga_{l,r}}=\frac{1}{\pi}
\leq c_{0,\ga_{l}}=\frac{2}{\pi}$
\item[\bf(ii)]
$\displaystyle c_{0,\ga}=\frac{1}{\sqrt{2}\pi}\leq\frac{1}{\pi}=c_{0,\emptyset}$
\item[\bf(iii)]
$\displaystyle c_{0,\ga}=\frac{1}{\sqrt{2}\pi}\leq\frac{\sqrt{2}}{\pi}
=\frac{\text{\rm diam}(\om)}{\pi}$
\item[\bf(iv)]
$\om$ is convex and 
$\displaystyle c_{0,\ga}=\frac{1}{\sqrt{2}\pi}\leq\frac{1}{\pi}
=c_{0,\emptyset}\leq\frac{\sqrt{2}}{\pi}=\frac{\text{\rm diam}(\om)}{\pi}$.
\end{itemize}
\end{rem}

\subsection{3D}
\label{anaex3D}

Let $\om:=\widehat\om\times I=I^3$, $\widehat\om:=I^2$, $I:=(0,1)$, 
$\ga=\ol{\ga_{b}\cup\ga_{t}\cup\ga_{l}\cup\ga_{r}\cup\ga_{f}\cup\ga_{\bk}}$,
where $\ga_{b}$, $\ga_{t}$, $\ga_{l}$, $\ga_{r}$, $\ga_{f}$, $\ga_{\bk}$ are the open 
bottom, top, left, right, front, and back boundary parts of $\ga$, respectively,
and $\gat\in P\big(\{\ga_{b},\ga_{t},\ga_{l},\ga_{r},\ga_{f},\ga_{\bk}\}\big)$,
and recall Section \ref{seclapmax3D}
as well as Theorem \ref{stateoftheartconstest}.
Again, we use canonical index notations such as
$$\ga_{b,r}:=\text{int}(\ol{\ga_{b}\cup\ga_{r}}),\quad
\ga_{b,r,\bk,f}:=\text{int}(\ol{\ga_{b}\cup\ga_{r}\cup\ga_{\bk}\cup\ga_{f}}).$$
From Appendix \ref{appanaex3D} we see for $c_{0,\gat}$
\begin{align}
\nonumber
c_{0,\emptyset}
&=\frac{1}{\pi},\\
\nonumber
c_{0,\ga_{b}}
=c_{0,\ga_{t}}
=c_{0,\ga_{l}}
=c_{0,\ga_{r}}
=c_{0,\ga_{f}}
=c_{0,\ga_{\bk}}
&=\frac{2}{\pi},\\
\nonumber
c_{0,\ga_{b,t}}
=c_{0,\ga_{l,r}}
=c_{0,\ga_{f,\bk}}
&=\frac{1}{\pi},\\
\nonumber
c_{0,\ga_{b,l}}
=c_{0,\ga_{b,r}}
=c_{0,\ga_{b,f}}
=c_{0,\ga_{b,\bk}}
\hspace*{57mm}&\\
\nonumber
=c_{0,\ga_{t,l}}
=c_{0,\ga_{t,r}}
=c_{0,\ga_{t,f}}
=c_{0,\ga_{t,\bk}}
=c_{0,\ga_{f,l}}
=c_{0,\ga_{f,r}}
=c_{0,\ga_{\bk,l}}
=c_{0,\ga_{\bk,r}}
&=\frac{\sqrt{2}}{\pi},\\
\nonumber
c_{0,\ga_{b,t,l}}
=c_{0,\ga_{b,t,r}}
=c_{0,\ga_{b,t,f}}
=c_{0,\ga_{b,t,\bk}}
=c_{0,\ga_{l,r,b}}
\hspace*{37mm}&\\
\label{constval3Da}
=c_{0,\ga_{l,r,t}}
=c_{0,\ga_{l,r,f}}
=c_{0,\ga_{l,r,\bk}}
=c_{0,\ga_{f,\bk,l}}
=c_{0,\ga_{f,\bk,r}}
=c_{0,\ga_{f,\bk,b}}
=c_{0,\ga_{f,\bk,t}}
&=\frac{2}{\sqrt{5}\pi},\\
\nonumber
c_{0,\ga_{b,\bk,l}}
=c_{0,\ga_{b,l,f}}
=c_{0,\ga_{b,f,r}}
=c_{0,\ga_{b,r,\bk}}
=c_{0,\ga_{t,\bk,l}}
=c_{0,\ga_{t,l,f}}
=c_{0,\ga_{t,f,r}}
=c_{0,\ga_{t,r,\bk}}
&=\frac{2}{\sqrt{3}\pi},\\
\nonumber
c_{0,\ga_{b,t,l,r}}
=c_{0,\ga_{b,t,f,\bk}}
=c_{0,\ga_{l,r,f,\bk}}
&=\frac{1}{\sqrt{2}\pi},\\
\nonumber
c_{0,\ga_{b,t,l,\bk}}
=c_{0,\ga_{b,t,f,l}}
=c_{0,\ga_{b,t,r,f}}
=c_{0,\ga_{b,t,r,\bk}}
=c_{0,\ga_{l,r,f,t}}
=c_{0,\ga_{l,r,f,b}}
\hspace*{6mm}&\\
\nonumber
=c_{0,\ga_{l,r,t,\bk}}
=c_{0,\ga_{l,r,b,\bk}}
=c_{0,\ga_{f,\bk,b,l}}
=c_{0,\ga_{f,\bk,t,l}}
=c_{0,\ga_{f,\bk,b,r}}
=c_{0,\ga_{f,\bk,r,r}}
&=\frac{2}{\sqrt{6}\pi},\\
\nonumber
c_{0,\ga_{b,t,l,r,\bk}}
=c_{0,\ga_{b,t,l,r,f}}
=c_{0,\ga_{b,t,l,f,\bk}}
=c_{0,\ga_{b,t,r,f,\bk}}
=c_{0,\ga_{b,l,r,f,\bk}}
=c_{0,\ga_{t,l,r,f,\bk}}
&=\frac{2}{3\pi},\\
\nonumber
c_{0,\ga}
&=\frac{1}{\sqrt{3}\pi},
\intertext{and for $c_{1,\gat}$}
\nonumber
c_{1,\emptyset}
=c_{1,\ga}
&=\frac{1}{\sqrt{2}\pi},\\
\nonumber
c_{1,\ga_{b}}
=c_{1,\ga_{t}}
=c_{1,\ga_{l}}
=c_{1,\ga_{r}}
=c_{1,\ga_{f}}
=c_{1,\ga_{\bk}}
&=\frac{2}{\sqrt{5}\pi},\\
\nonumber
c_{1,\ga_{l,r}}=c_{1,\ga_{b,t}}=c_{1,\ga_{f,\bk}}
&=\frac{1}{\pi},\\
\nonumber
c_{1,\ga_{b,l}}
=c_{1,\ga_{b,r}}
=c_{1,\ga_{b,f}}=c_{1,\ga_{b,\bk}}
\hspace*{59mm}&\\
\label{constval3Db}
=c_{1,\ga_{t,l}}=c_{1,\ga_{t,r}}
=c_{1,\ga_{t,f}}=c_{1,\ga_{t,\bk}}
=c_{1,\ga_{f,l}}=c_{1,\ga_{l,\bk}}
=c_{1,\ga_{\bk,r}}=c_{1,\ga_{f,r}}
&=\frac{\sqrt{2}}{\pi},\\
\nonumber
c_{1,\ga_{b,l,t}}=c_{1,\ga_{b,r,t}}
=c_{1,\ga_{b,f,t}}=c_{1,\ga_{b,\bk,t}}
=c_{1,\ga_{r,l,t}}
\hspace*{38mm}&\\
\nonumber
=c_{1,\ga_{r,l,b}}
=c_{1,\ga_{r,l,f}}=c_{1,\ga_{r,l,\bk}}
=c_{1,\ga_{f,\bk,l}}=c_{1,\ga_{f,\bk,r}}
=c_{1,\ga_{f,\bk,t}}=c_{1,\ga_{f,\bk,b}}
&=\frac{2}{\pi},\\
\nonumber
c_{1,\ga_{b,l,\bk}}=c_{1,\ga_{b,r,\bk}}
=c_{1,\ga_{b,l,f}}=c_{1,\ga_{b,r,f}}
=c_{1,\ga_{t,l,\bk}}=c_{1,\ga_{t,r,\bk}}
=c_{1,\ga_{t,l,f}}=c_{1,\ga_{t,r,f}}
&=\frac{2}{\sqrt{3\pi}},
\end{align}
and all the other remaining cases follow by $c_{1,\gan}=c_{1,\gat}$ as well as symmetry.

\begin{rem}
\label{stateoftheartconstestrem3D}
Theorem \ref{stateoftheartconstest} may be verified by these examples. E.g.:
\begin{itemize}
\item[\bf(i)]
$\emptyset\neq\ga_{b}
\subset\ga_{b,l}
\subset\ga_{b,l,r}
\subset\ga_{b,t,l,r}
\subset\ga_{b,t,l,r,f}
\subset\ga
\qimpl$
\begin{align*}
c_{0,\ga}=\frac{1}{\sqrt{3}\pi}
\leq c_{0,\ga_{b,t,l,r,f}}=\frac{2}{3\pi}
\leq c_{0,\ga_{b,t,l,r}}=\frac{1}{\sqrt{2}\pi}
\leq c_{0,\ga_{b,l,r}}=\frac{2}{\sqrt{5}\pi}
\leq c_{0,\ga_{b,l}}=\frac{2}{\sqrt{2}\pi}
\leq c_{0,\ga_{b}}=\frac{2}{\pi}
\end{align*}
\item[\bf(i')]
$\emptyset\neq\ga_{b}
\subset\ga_{b,l}
\subset\ga_{b,l,r}
\subset\ga_{b,f,l,r}
\subset\ga_{b,f,l,r,t}
\subset\ga
\qimpl$
\begin{align*}
c_{0,\ga}=\frac{1}{\sqrt{3}\pi}
\leq c_{0,\ga_{b,f,l,r,t}}=\frac{2}{3\pi}
\leq c_{0,\ga_{b,f,l,r}}=\frac{2}{\sqrt{6}\pi}
\leq c_{0,\ga_{b,l,r}}=\frac{2}{\sqrt{5}\pi}
\leq c_{0,\ga_{b,l}}=\frac{2}{\sqrt{2}\pi}
\leq c_{0,\ga_{b}}=\frac{2}{\pi}
\end{align*}
\item[\bf(i'')]
$\emptyset\neq\ga_{l}
\subset\ga_{l,r}
\subset\ga_{b,l,r}
\subset\ga_{b,f,l,r}
\subset\ga_{b,f,l,r,t}
\subset\ga
\qimpl$
\begin{align*}
c_{0,\ga}=\frac{1}{\sqrt{3}\pi}
\leq c_{0,\ga_{b,f,l,r,t}}=\frac{2}{3\pi}
\leq c_{0,\ga_{b,f,l,r}}=\frac{2}{\sqrt{6}\pi}
\leq c_{0,\ga_{b,l,r}}=\frac{2}{\sqrt{5}\pi}
\leq c_{0,\ga_{l,r}}=\frac{1}{\pi}
\leq c_{0,\ga_{l}}=\frac{2}{\pi}
\end{align*}
\item[\bf(i''')]
$\emptyset\neq\ga_{l}
\subset\ga_{b,l}
\subset\ga_{b,f,l}
\subset\ga_{b,f,l,r}
\subset\ga_{b,f,l,r,t}
\subset\ga
\qimpl$
\begin{align*}
c_{0,\ga}=\frac{1}{\sqrt{3}\pi}
\leq c_{0,\ga_{b,f,l,r,t}}=\frac{2}{3\pi}
\leq c_{0,\ga_{b,f,l,r}}=\frac{2}{\sqrt{6}\pi}
\leq c_{0,\ga_{b,f,l}}=\frac{2}{\sqrt{3}\pi}
\leq c_{0,\ga_{b,l}}=\frac{2}{\sqrt{2}\pi}
\leq c_{0,\ga_{l}}=\frac{2}{\pi}
\end{align*}
\item[\bf(ii)]
$\displaystyle c_{0,\ga}=\frac{1}{\sqrt{3}\pi}\leq\frac{1}{\pi}=c_{0,\emptyset}$
\item[\bf(iii)]
$\displaystyle c_{0,\ga}=\frac{1}{\sqrt{3}\pi}\leq\frac{\sqrt{3}}{\pi}
=\frac{\text{\rm diam}(\om)}{\pi}$
\item[\bf(iv)]
$\om$ is convex and 
$\displaystyle c_{0,\ga}=\frac{1}{\sqrt{3}\pi}\leq\frac{1}{\pi}
=c_{0,\emptyset}\leq\frac{\sqrt{3}}{\pi}=\frac{\text{\rm diam}(\om)}{\pi}$.
\item[\bf(v)]
$\om$ is convex and 
$\displaystyle c_{1,\ga}=c_{1,\emptyset}=\frac{1}{\sqrt{2}\pi}
\leq\frac{1}{\pi}=c_{0,\emptyset}\leq\frac{\sqrt{3}}{\pi}=\frac{\text{\rm diam}(\om)}{\pi}$.
\item[\bf(vi)]
$\om$ is convex and
\begin{align*}
c_{0,\ga}
&=\frac{1}{\sqrt{3}\pi}
\leq\frac{1}{\sqrt{2}\pi}
=c_{0,1,\ga}
=\max\{c_{0,\ga},c_{1,\ga}\}
=\max\{\frac{1}{\sqrt{3}\pi},\frac{1}{\sqrt{2}\pi}\}
\leq\frac{1}{\pi}
=c_{0,\emptyset}
\leq\frac{\sqrt{3}}{\pi}
=\frac{\text{\rm diam}(\om)}{\pi},\\
c_{0,\ga}
&=\frac{1}{\sqrt{3}\pi}
\leq\frac{1}{\pi}
=c_{0,1,\emptyset}
=\max\{c_{0,\emptyset},c_{1,\emptyset}\}
=\max\{\frac{1}{\pi},\frac{1}{\sqrt{2}\pi}\}
=c_{0,\emptyset}
\leq\frac{\sqrt{3}}{\pi}
=\frac{\text{\rm diam}(\om)}{\pi}.
\end{align*}
\end{itemize}
\end{rem}

\begin{rem}
In general, the Maxwell constants do not have monotonicity properties,
which can also be verified by the latter examples.
In fact, in our examples, the Maxwell constants are monotone increasing up
to a certain situation in the `middle', where the tangential and the normal boundary condition
are equally strong, and from there on the Maxwell constants are monotone decreasing. E.g.:

$\bullet\quad
\emptyset\neq\ga_{b}
\subset\ga_{b,l}
\subset\ga_{b,l,r}
\subset\ga_{b,t,l,r}
\subset\ga_{b,t,l,r,f}
\subset\ga$, but
\begin{align*}
c_{1,\ga}=\frac{1}{\sqrt{2}\pi}
&\overset{\text{\sf ok}}{\leq} c_{1,\ga_{b,t,l,r,f}}
=c_{1,\ga_{\bk}}=\frac{2}{\sqrt{5}\pi}
\overset{\text{\sf ok}}{\leq} c_{1,\ga_{b,t,l,r}}
=c_{1,\ga_{f,\bk}}=\frac{1}{\pi}\\
&\overset{\text{\sf ok}}{\leq} c_{1,\ga_{b,l,r}}=\frac{2}{\pi}
\underset{>}{\overset{\text{\sf not ok}}{\not\leq}} c_{1,\ga_{b,l}}=\frac{2}{\sqrt{2}\pi}
\underset{>}{\overset{\text{\sf not ok}}{\not\leq}} c_{1,\ga_{b}}=\frac{2}{\sqrt{5}\pi}.
\end{align*}

$\bullet\quad
\emptyset\neq\ga_{b}
\subset\ga_{b,l}
\subset\ga_{b,l,r}
\subset\ga_{b,f,l,r}
\subset\ga_{b,f,l,r,t}
\subset\ga$, but
\begin{align*}
c_{1,\ga}=\frac{1}{\sqrt{2}\pi}
&\overset{\text{\sf ok}}{\leq} c_{1,\ga_{b,t,l,r,f}}
=c_{1,\ga_{\bk}}=\frac{2}{\sqrt{5}\pi}
\overset{\text{\sf ok}}{\leq} c_{1,\ga_{b,f,l,r}}
=c_{1,\ga_{t,\bk}}=\frac{2}{\sqrt{2}\pi}\\
&\overset{\text{\sf ok}}{\leq} c_{1,\ga_{b,l,r}}=\frac{2}{\pi}
\underset{>}{\overset{\text{\sf not ok}}{\not\leq}} c_{1,\ga_{b,l}}=\frac{2}{\sqrt{2}\pi}
\underset{>}{\overset{\text{\sf not ok}}{\not\leq}} c_{1,\ga_{b}}=\frac{2}{\sqrt{5}\pi}.
\end{align*}

$\bullet\quad
\emptyset\neq\ga_{l}
\subset\ga_{l,r}
\subset\ga_{b,l,r}
\subset\ga_{b,f,l,r}
\subset\ga_{b,f,l,r,t}
\subset\ga$, but
\begin{align*}
c_{1,\ga}=\frac{1}{\sqrt{2}\pi}
&\overset{\text{\sf ok}}{\leq} c_{1,\ga_{b,t,l,r,f}}
=c_{1,\ga_{\bk}}=\frac{2}{\sqrt{5}\pi}
\overset{\text{\sf ok}}{\leq} c_{1,\ga_{b,f,l,r}}
=c_{1,\ga_{t,\bk}}=\frac{2}{\sqrt{2}\pi}\\
&\overset{\text{\sf ok}}{\leq} c_{1,\ga_{b,l,r}}=\frac{2}{\pi}
\underset{>}{\overset{\text{\sf not ok}}{\not\leq}} c_{1,\ga_{l,r}}=\frac{1}{\pi}
\underset{>}{\overset{\text{\sf not ok}}{\not\leq}} c_{1,\ga_{l}}=\frac{2}{\sqrt{5}\pi}.
\end{align*}

$\bullet\quad
\emptyset\neq\ga_{l}
\subset\ga_{b,l}
\subset\ga_{b,f,l}
\subset\ga_{b,f,l,r}
\subset\ga_{b,f,l,r,t}
\subset\ga$, but
\begin{align*}
c_{1,\ga}=\frac{1}{\sqrt{2}\pi}
&\overset{\text{\sf ok}}{\leq} c_{1,\ga_{b,t,l,r,f}}
=c_{1,\ga_{\bk}}=\frac{2}{\sqrt{5}\pi}
\overset{\text{\sf ok}}{\leq} c_{1,\ga_{b,f,l,r}}
=c_{1,\ga_{t,\bk}}=\frac{2}{\sqrt{2}\pi}\\
&\underset{>}{\overset{\text{\sf not ok}}{\not\leq}} c_{1,\ga_{b,l,f}}=\frac{2}{\sqrt{3}\pi}
\overset{\text{\sf ok}}{\leq} c_{1,\ga_{l,b}}=\frac{2}{\sqrt{2}\pi}
\underset{>}{\overset{\text{\sf not ok}}{\not\leq}} c_{1,\ga_{l}}=\frac{2}{\sqrt{5}\pi}.
\end{align*}
\end{rem}

\section{Numerical Examples}

The finite element method (FEM) is applied for evaluation of the Rayleigh quotients on finite dimensional subspaces. 
Constants are therefore approximated and convergence to their exact values is expected for higher dimensions. 
Assuming that $\om$ is discretised by a triangular (2D) or a tetrahedral (3D) mesh $\mathcal T$, 
we use only the lowest order finite elements available:
\begin{itemize}
\item[-] 
Linear Lagrange (\pone) nodal elements $\theta^\pone_i$
for approximations of $\H{}{\gat}(\grad{},\om) $ spaces,
\item[-] 
Linear N\'ed\'elec (\ned) edge elements $\Theta^\ned_i$
for approximations of $\H{}{\gat}(\rot{},\om)$ spaces,
\item[-] 
Linear Raviart-Thomas (\rt) face elements $\Theta^\rt_i$
for approximations of $\H{}{\gat}(\div{},\om)$ spaces.
\end{itemize}

\begin{figure}[h]
\begin{minipage}[c]{0.9\textwidth}
\center
\includegraphics[width=0.35\textwidth]{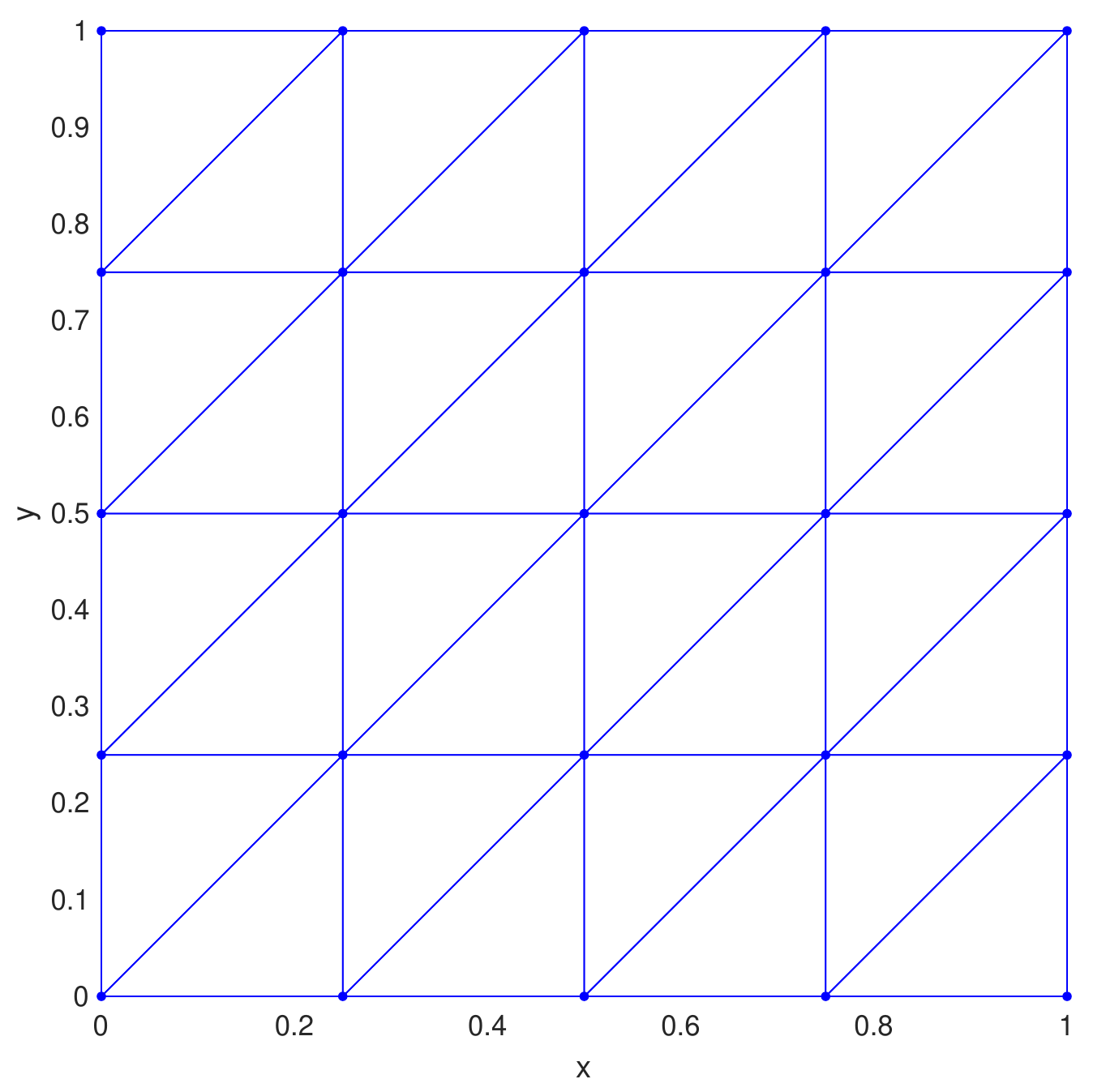}
\hspace{1cm}
\includegraphics[width=0.35\textwidth]{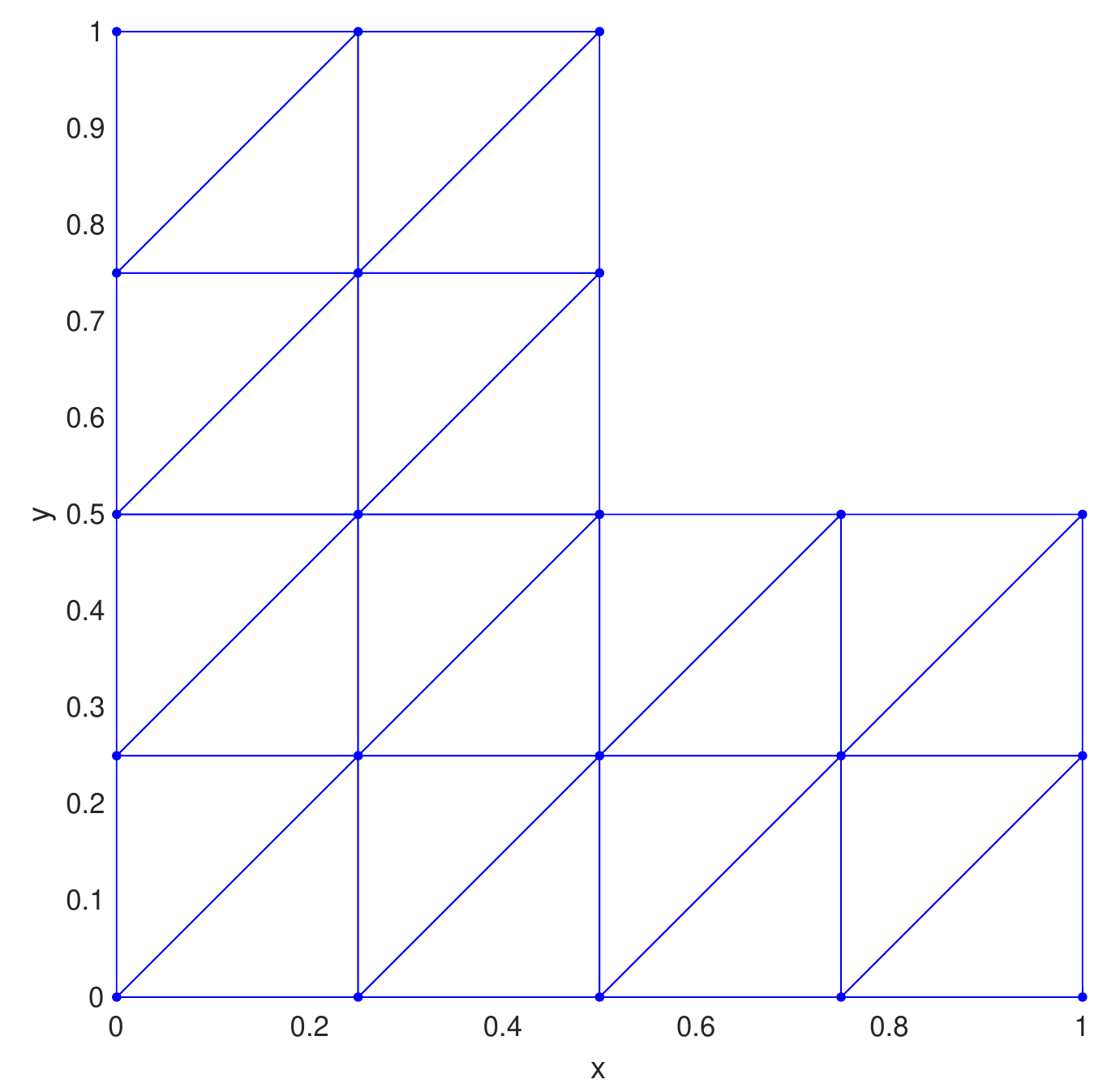}
\caption{Coarse (level 1) triangular meshes for the unit square 
and the L-shape domains.} \label{meshes2D}
\includegraphics[width=0.355\textwidth]{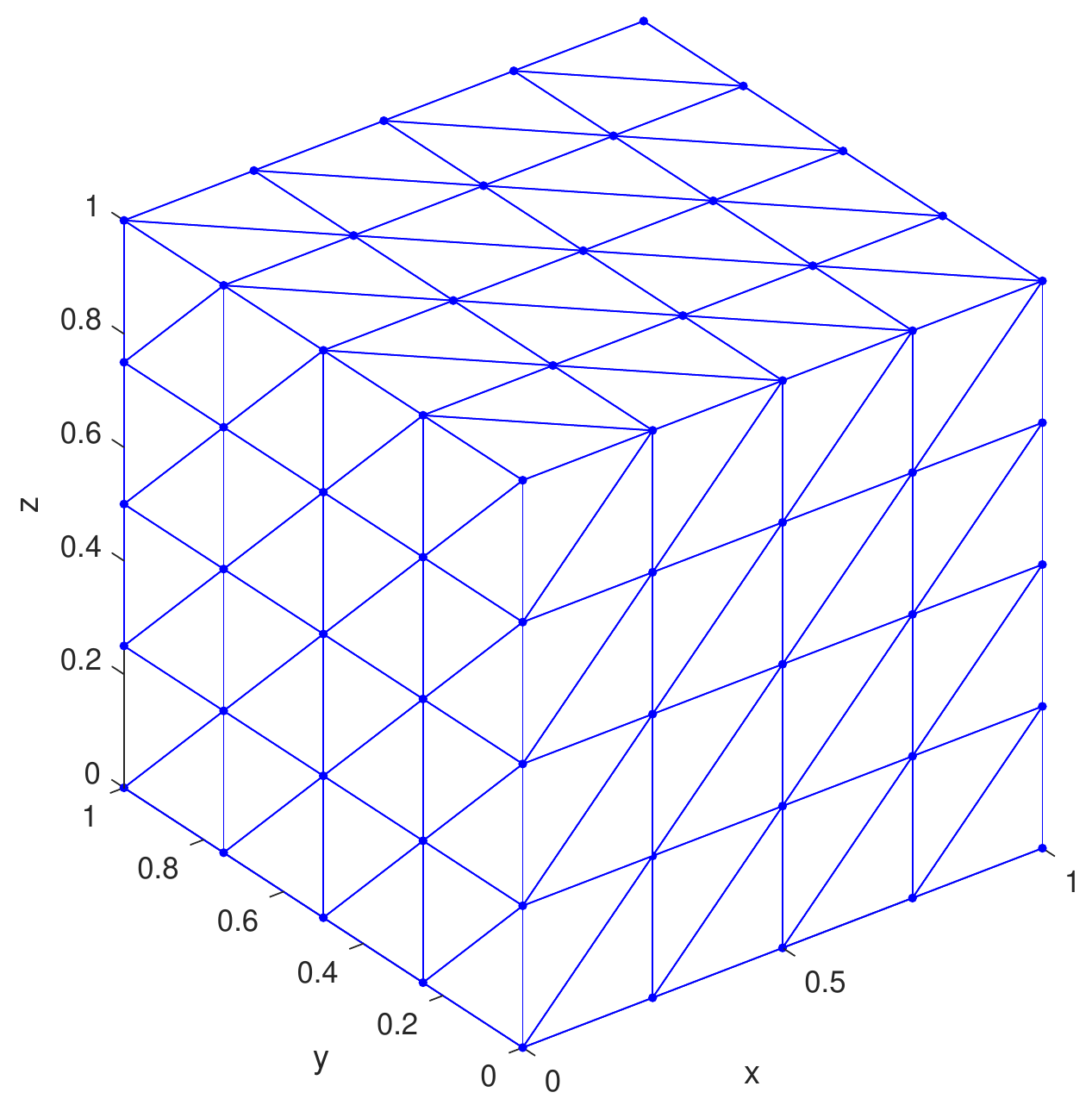}
\hspace{1cm}
\includegraphics[width=0.35\textwidth]{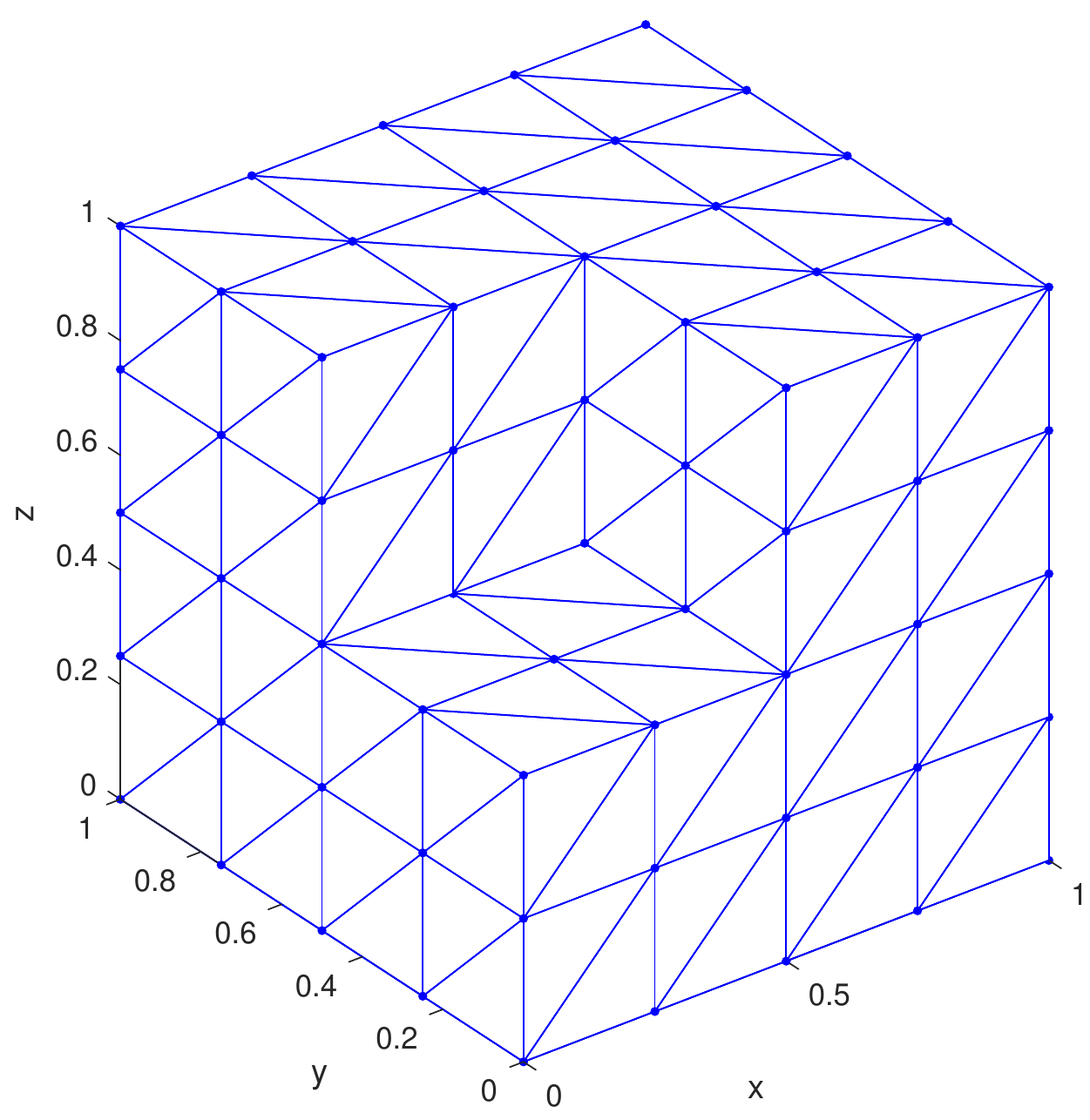}
\caption{Coarse (level 1) tetrahedral meshes for the unit cube and the Fichera corner domains.}
\label{meshes3D}
\end{minipage}
\end{figure}

We assemble the mass matrices $\mass^\pone$, $\mass^\ned$, $\mass^\rt$ and the stiffness matrices 
$\stiff^\pone, \stiff^\ned, \stiff^\rt$ defined by
\begin{align*}
\mass_{ij}^\pone & = \scp{\theta^\pone_i}{\theta^\pone_j}_{\L{2}(\om)}, 
&
\stiff_{ij}^\pone & = \scp{\grad{}\theta^\pone_i}{\grad{}\theta^\pone_j}_{\L{2}(\om)}, \\
\mass_{ij}^\ned & = \scp{\Theta^\ned_i}{\Theta^\ned_j}_{\L{2}(\om)}, 
&
\stiff_{ij}^\ned & = \scp{\rot{}\Theta^\ned_i}{\rot{}\Theta^\ned_j}_{\L{2}(\om)},\\
\mass_{ij}^\rt & = \scp{\Theta^\rt_i}{\Theta^\rt_j}_{\L{2}(\om)}, 
&
\stiff_{ij}^\rt & = \scp{\div{}\Theta^\rt_i}{\div{}\Theta^\rt_j}_{\L{2}(\om)}, 
\end{align*}
where the indices $i$, $j$ are the global numbers of the corresponding degrees of freedom, i.e., 
they are related to mesh nodes (for \pone\,elements) or mesh edges or faces 
(for \ned\,and \rt\,elements). By using the affine mappings (for \pone\,elements) 
or Piola mappings (for \ned\,and \rt\,elements) 
from reference elements we can assemble the local matrices. 
Detailed implementation of finite element assemblies is explained in 
\cite{anjamvaldmanfastmatlab, rahmanvaldmanfastmatlab}. 
Squares of terms from Theorem \ref{fpmconst} are easy to evaluate as quadratic forms with mass and stiffness matrices:
\begin{align*}
\norm{u}_{\L{2}(\om)}^2 = & \mass^\pone u^\pone \cdot u^\pone, 
& 
\norm{\grad{} u}_{\L{2}(\om)}^2 & = \stiff^\pone u^\pone \cdot u^\pone, \\
\norm{E}_{\L{2}(\om)}^2 = & \mass^\ned E^\ned \cdot E^\ned, 
& 
\norm{\rot{} E}_{\L{2}(\om)}^2 & = \stiff^\ned E^\ned \cdot E^\ned,\\
\norm{H}_{\L{2}(\om)}^2 = & \mass^\rt H^\rt \cdot H^\rt, 
& 
\norm{\div{} H}_{\L{2}(\om)}^2 & = \stiff^\rt H^\rt \cdot H^\rt, 
\end{align*}
where $u^\pone$, $E^\ned$, and $H^\rt$ represent (column) vectors of coefficients with respect 
to their finite element bases of \pone, \ned,\ and \rt, respectively.

\subsection{Poincar\'e-Friedrichs and Divergence Constants} 

The classical Friedrichs constant $c_{0,\ga}$ is approximated as
$$\frac{1}{c_{0,\ga,\pone}^2} 
= \lambda_{0,\ga,\pone}^2 
= \min_{\substack{0 \neq u^\pone,\\ u^\pone_{\ga}=0}} 
\frac{\stiff^\pone u^\pone \cdot u^\pone}{\mass^\pone u^\pone \cdot u^\pone},$$
where $u^\pone_{\ga}$ denotes a subvector of $u^\pone$ in indices corresponding to boundary nodes.
$\lambda_{0,\ga,\pone}^2$ is the minimal (positive) eigenvalue of the generalized eigenvalue problem
$$\stiff^\pone \, u^\pone 
= \lambda^2 \, \mass^\pone \, u^\pone,\qquad
u^\pone_{\ga}=0,$$
and may also be found by computing the minimal (positive) eigenvalue of
$$\stiff^\pone_{\sf int} \, u^\pone_{\sf int} 
= \lambda^2 \, \mass^\pone_{\sf int} \, u^\pone_{\sf int},$$
where $\stiff^\pone_{\sf int}$, $\mass^\pone_{\sf int}$,
and $u^\pone_{\sf int}$ are restrictions 
of the matrices $\stiff^\pone$, $\mass^\pone$,
and the vector $u^\pone$, respectively, 
to indices corresponding to internal mesh nodes only.
Note that $\stiff^\pone_{\sf int}$ is regular.

The classical Poincar\'e constant $c_{0,\emptyset}$ is approximated as
$$\frac{1}{c_{0,\emptyset,\pone}^2} 
= \lambda_{0,\emptyset,\pone}^2 
= \min_{\substack{0 \neq u^\pone,\\ u^\pone \cdot 1^\pone=0}} 
\frac{\stiff^\pone u^\pone\cdot u^\pone}{\mass^\pone u^\pone\cdot u^\pone},$$
where the constraint $u^\pone \cdot 1^\pone=0$ means that the vector $u^\pone$ 
has to be perpendicular to the constant vector of ones. 
$\lambda_{0,\emptyset,\pone}^2$ is the minimal positive eigenvalue of the generalized eigenvalue problem
$$\stiff^\pone \, u^\pone 
= \lambda^2 \, \mass^\pone \, u^\pone.$$
The minimal eigenvalue here is $\lambda^2=0$ and the corresponding eigenvector is the constant vector of ones. 
Analogously, the Poincar\'e-Friedrichs (Laplace) constants 
for mixed boundary conditions $c_{0,\gat}$ is approximated 
by using the same techniques and finite elements \pone.
More precisely, for $\gat\neq\emptyset$ we have 
\begin{align}
\label{FPconstapprox}
\frac{1}{c_{0,\gat,\pone}^2} 
= \lambda_{0,\gat,\pone}^2 
= \min_{\substack{0 \neq u^\pone,\\ u^\pone_{\gat}=0}} 
\frac{\stiff^\pone u^\pone \cdot u^\pone}{\mass^\pone u^\pone \cdot u^\pone},
\end{align}
where $u^\pone_{\gat}$ denotes a subvector of $u^\pone$ in indices corresponding to boundary nodes of $\gat$.
$\lambda_{0,\gat,\pone}^2$ is the minimal (positive) eigenvalue of the generalized eigenvalue problem
$$\stiff^\pone \, u^\pone 
= \lambda^2 \, \mass^\pone \, u^\pone,\qquad
u^\pone_{\gat}=0,$$
and may be computed again by solving a restricted problem (to internal nodes and some boundary nodes)
with a regular stiffness matrix $\stiff^\pone_{{\sf int},\gat}$, i.e.,
$$\stiff^\pone_{{\sf int},\gat}  \, u^\pone_{{\sf int},\gat} 
= \lambda^2 \, \mass^\pone_{{\sf int},\gat} \, u^\pone_{{\sf int},\gat}.$$

As in any dimension the Poincar\'e-Friedrichs constants can be computed 
either as a gradient or as a divergence constant, see Theorem \ref{fpmconst}, we can approximate 
$$c_{0,\gat}=c_{2,\gan}$$ 
either by \eqref{FPconstapprox} or by
$$\frac{1}{c_{2,\gan,\rt}^2} 
= \lambda_{2,\gan,\rt}^2 
= \min_{\substack{0 \neq H^\rt,\\ H^\rt_{\gan}=0,\\H^\rt\bot N(\stiff^\rt)}}
\frac{\stiff^\rt H^\rt \cdot H^\rt}{\mass^\rt H^\rt \cdot H^\rt},$$
where $H^\rt_{\gan}$ denotes a subvector of $H^\rt$ in indices corresponding to boundary faces of $\gan$
(boundary edges in 2D).
$\lambda_{2,\gan,\rt}^2$ is the minimal positive eigenvalue of the generalized eigenvalue problem
$$\stiff^\rt \, H^\rt
= \lambda ^2\, \mass^\rt \, H^\rt,\qquad
H^\rt_{\gan}=0,$$
respectively,
$$\stiff^\rt_{{\sf int},\gan} \, H^\rt_{{\sf int},\gan}
= \lambda ^2\, \mass^\rt_{{\sf int},\gan} \, H^\rt_{{\sf int},\gan},$$
where $\stiff^\rt_{{\sf int},\gan}$, $\mass^\rt_{{\sf int},\gan}$,
and $H^\rt_{{\sf int},\gan}$ are restrictions of the matrices 
$\stiff^\rt$, $\mass^\rt$, and the vector $H^\rt$ 
to indices corresponding to `free' mesh faces (edges in 2D) only. 
Note that there are a lot of first zero eigenvalues $\lambda^2=0$
as neither $\stiff^\rt$ nor $\stiff^\rt_{{\sf int},\gan}$ are regular
due to the existence of large kernels $N(\stiff^\rt)$ and $N(\stiff^\rt_{{\sf int},\gan})$
since all rotations belong to the kernel of the divergence.


\subsection{Maxwell Constants} 

While the computation of the Poincar\'e-Friedrichs constants $c_{0,\gat}=c_{2,\gan}$
is more or less independent of the dimension, 
the computation of the Maxwell constants is different in 2D and 3D, or generally, in ND.
By Remark \ref{stateoftheartconstestrem2D} (v) we have in 2D
$$c_{1,\gan}=c_{0,\gat},$$
and thus the Maxwell constants can simply be computed by the corresponding
Poincar\'e-Friedrichs (Laplace) constants.
In particular, for the tangential (electric) and normal (magnetic) Maxwell constants it holds
$$c_{1,\ga}=c_{0,\emptyset},\qquad
c_{1,\emptyset}=c_{0,\ga}.$$
By Remark \ref{stateoftheartconstestrem3D} (vii) we have in 3D
$$c_{1,\gat}=c_{1,\gan},$$
and thus this Maxwell constant has to be calculated separately, 
since the simple link to the Poincar\'e-Friedrichs (Laplace) constants is lost in higher dimensions.
In particular, for the tangential (electric) and normal (magnetic) Maxwell constants it holds
$$c_{1,\ga}=c_{1,\emptyset}.$$
The Maxwell constants are approximated as
$$\frac{1}{c_{1,\gat,\ned}^2} 
= \lambda_{1,\gat,\ned}^2 
= \min_{\substack{0\neq E^\ned,\\E^\ned_{\gat}=0,\\E^\ned\bot N(\stiff^\ned)}}
\frac{\stiff^\ned E^\ned\cdot E^\ned}{\mass^\ned E^\ned\cdot E^\ned},$$
where $E^\ned_{\gat}$ denotes a subvector of $E^\ned$ in indices corresponding to boundary edges of $\gat$.
$\lambda_{1,\gat,\ned}^2$ is the minimal positive eigenvalue of the generalized eigenvalue problem
$$\stiff^\ned \, E^\ned 
= \lambda ^2\, \mass^\ned \, E^\ned,\qquad
E^\ned_{\gat}=0,$$
respectively,
$$\stiff^\ned_{{\sf int},\gat} \, E^\ned_{{\sf int},\gat} 
= \lambda ^2\, \mass^\ned_{{\sf int},\gat} \, E^\ned_{{\sf int},\gat},$$
where $\stiff^\ned_{{\sf int},\gat}$, $\mass^\ned_{{\sf int},\gat}$, and $E^\ned_{{\sf int},\gat}$ 
are restrictions of the matrices 
$\stiff^\ned$, $\mass^\ned$, and the vector $E^\ned$ to indices 
corresponding to `free' mesh edges only. 
Note that similar to the computation of the divergence constants 
there are a lot of first zero eigenvalues $\lambda^2=0$
as neither $\stiff^\ned$ nor $\stiff^\ned_{{\sf int},\gat}$ are regular
due to the existence of large kernels $N(\stiff^\ned)$ and $N(\stiff^\ned_{{\sf int},\gat})$
since now all gradients belong to the kernel of the rotation.


We emphasise that the Maxwell constants are also approximated by
$$\frac{1}{c_{1,\gan,\ned}^2} 
= \lambda_{1,\gan,\ned}^2 
= \min_{\substack{0\neq E^\ned,\\E^\ned_{\gan}=0,\\E^\ned\bot N(\stiff^\ned)}}
\frac{\stiff^\ned E^\ned\cdot E^\ned}{\mass^\ned E^\ned\cdot E^\ned}.$$

\begin{figure}
\includegraphics[width=0.489\textwidth]{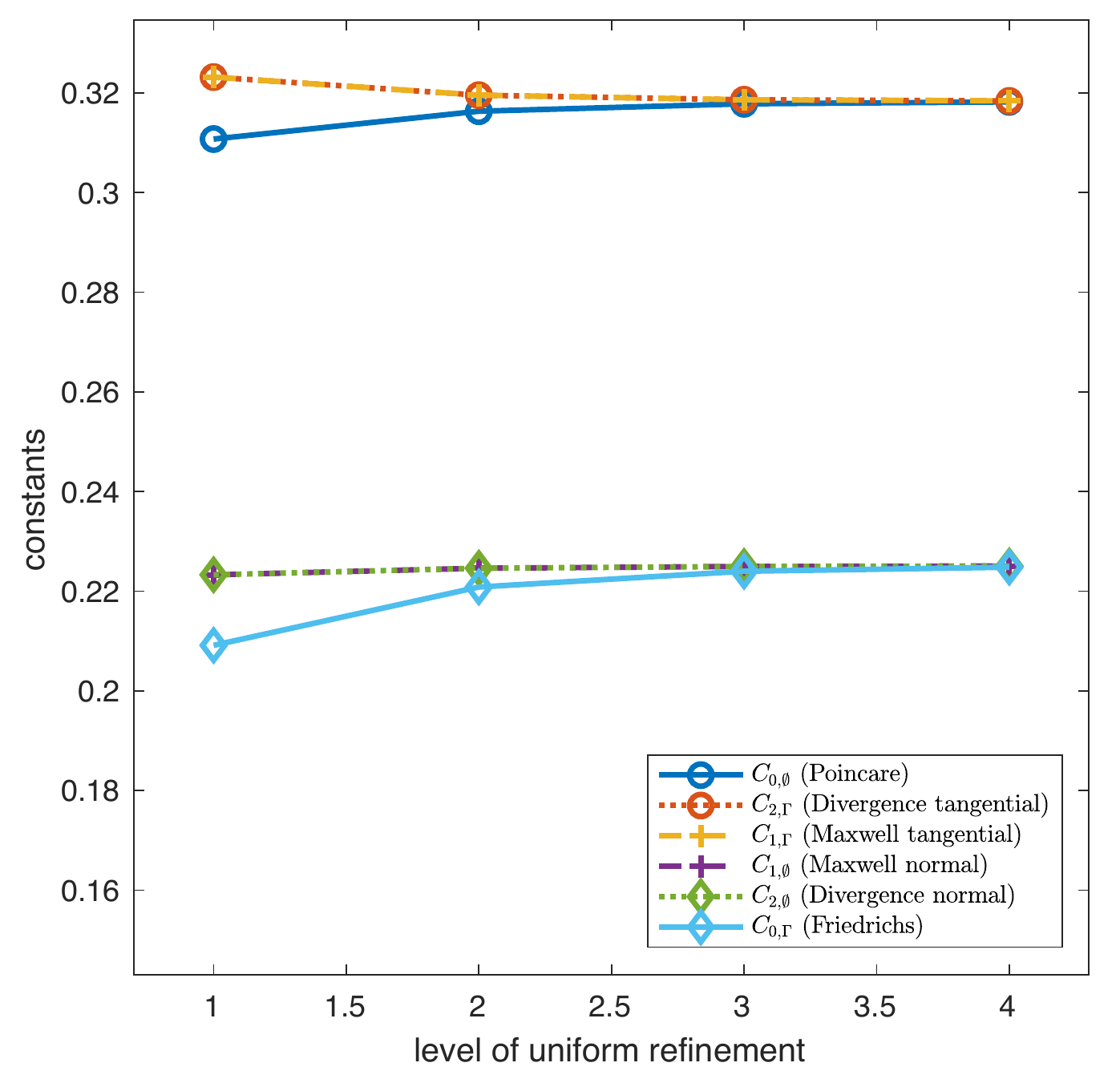}
\includegraphics[width=0.481\textwidth]{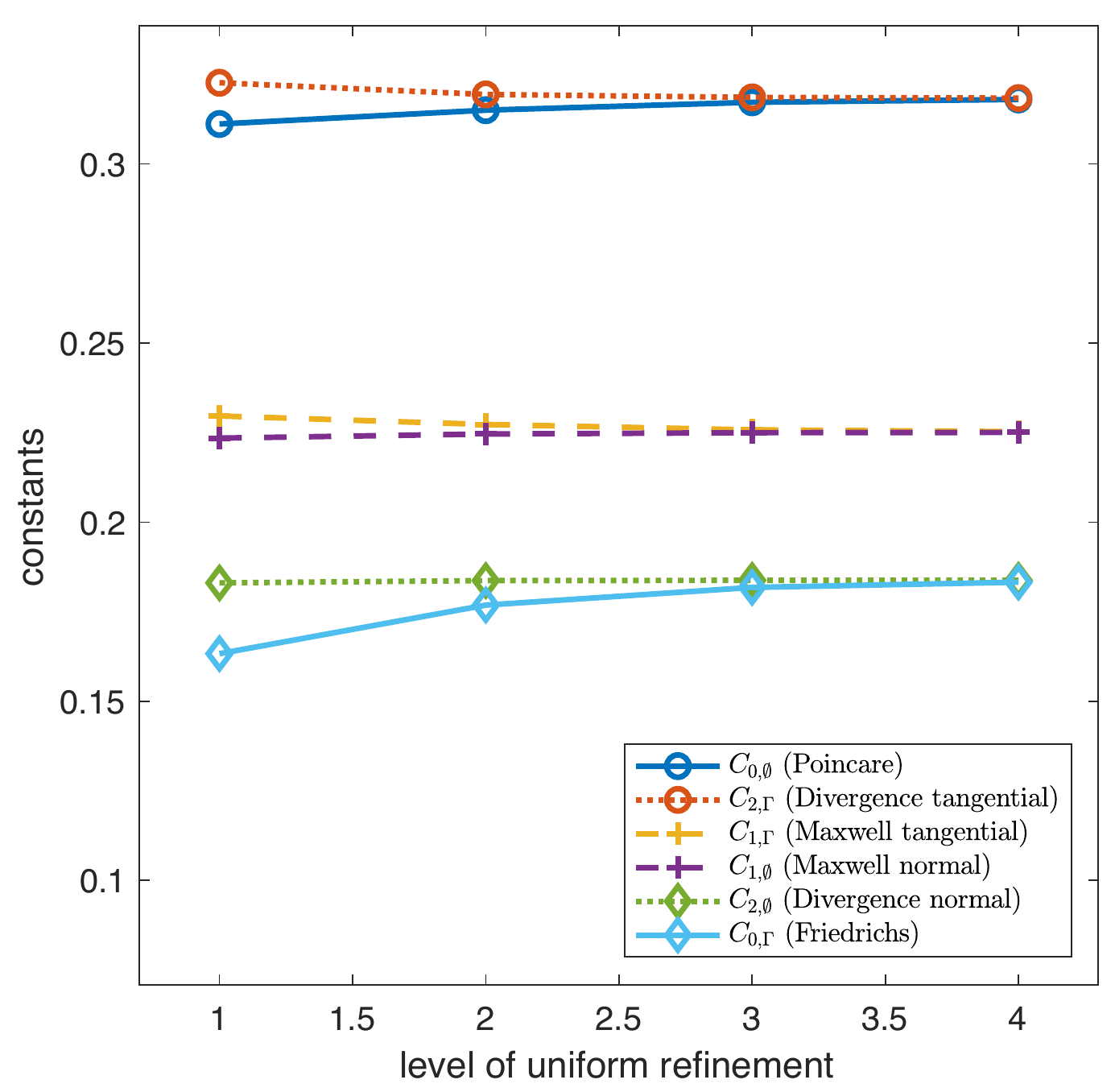}
\caption{Constants computed for the unit square domain (left) and the unit cube domain (right) with full boundary conditions.
Theoretically, it holds in 2D (left)
$$\hspace{-23mm}
c_{0,\ga}=c_{2,\emptyset}=c_{1,\emptyset}=\frac{1}{\sqrt{2}\pi}\approx0.225
<c_{0,\emptyset}=c_{2,\ga}=c_{1,\ga}=\frac{1}{\pi}\approx0.318$$ 
and in 3D (right)
$$\hspace{-23mm}
c_{0,\ga}=c_{2,\emptyset}=\frac{1}{\sqrt{3}\pi}\approx0.184
<c_{1,\ga}=c_{1,\emptyset}=\frac{1}{\sqrt{2}\pi}\approx0.225
<c_{0,\emptyset}=c_{2,\ga}=\frac{1}{\pi}\approx0.318.$$}
\label{constcompusuc}
\end{figure}

\subsection{2D Computations} 
\label{section:2D}

We demonstrate two benchmarks with the unit square and the L-shape domain, 
the first with known and the second with unknown values of the constants. 
Their coarse (level 1) meshes are displayed in Figure \ref{meshes2D}.
For the unit square we have by Remark \ref{stateoftheartconstestrem2D} exact values
$$c_{0,\ga}=c_{1,\emptyset}=\frac{1}{\sqrt{2}\pi}\approx0.22507907,\qquad 
c_{0,\emptyset}=c_{1,\ga}=\frac{1}{\pi}\approx0.31830988,$$
and our approximative values converge to them, see Table \ref{ta:unitsquare}. 
This extends results of \cite[Table 1]{Valdman}. 
For one case of mixed boundary conditions (with missing boundary part $\ga_{b}$) we have
by Remark \ref{stateoftheartconstestrem2D} and \eqref{constval2D} exact values
$$c_{0,\ga_{t,l,r}}=c_{1,\ga_{b}}=\frac{2}{\sqrt{5}\pi}\approx0.28470501,\qquad 
c_{0,\ga_{b}}=c_{1,\ga_{t,l,r}}=\frac{2}{\pi}\approx0.63661977,$$
and our approximative values converge again to them, see Table \ref{ta:unitsquare_mixed}. 
Approximative values for the L-shape domain are provided in Tables \ref{ta:Lshape} and \ref{ta:Lshape_mixed}. 
We notice a quadratic convergence of all constants with respect to the mesh size $h$. 
It means that an absolute error of any considered constant approximation is reduced 
by a factor of 4 after each uniform mesh refinement. Geometrical parameters of triangular meshes used for the unit square domain are given in Table \ref{ta:unitsquare_meshes}.  

\begin{table}[h]
\begin{tabularx}{\textwidth}{X | X X X X X X }
mesh level 
& $c_{0,\emptyset,\pone}$
& $c_{2,\ga,\rt}$
& $c_{1,\ga,\ned}$
& $c_{1,\emptyset,\ned}$
& $c_{2,\emptyset,\rt}$
& $c_{0,\ga,\pone}$\\
\hline
\hline
1 & 0.31072999 & 0.32316745 & 0.32316745 & 0.22328039 & 0.22328039 & 0.20912552 \\
2 & 0.31631302 & 0.31953907 & 0.31953907 & 0.22460517 & 0.22460517 & 0.22083319 \\
3 & 0.31780225 & 0.31861815 & 0.31861815 & 0.22495907 & 0.22495907 & 0.22400032 \\
4 & 0.31818232 & 0.31838701 & 0.31838701 & 0.22504898 & 0.22504898 & 0.22480828 \\
5 & 0.31827795 & 0.31832917 & 0.31832917 & 0.22507155 & 0.22507155 & 0.22501131 \\
6 & 0.31830190 & 0.31831471 & 0.31831471 & 0.22507720 & 0.22507720 & 0.22506213 \\
7 & 0.31830789 & 0.31831109 & 0.31831109 & 0.22507861 & 0.22507861 & 0.22507484 \\
\hline
$\infty$ & 0.31830988 & 0.31830988 & 0.31830988 & 0.22507907 & 0.22507907 & 0.22507907 \\
\end{tabularx}
\vspace{1mm}
\caption{Constants computed for the unit square domain and full boundary conditions.}
\label{ta:unitsquare}
\begin{tabularx}{\textwidth}{X | X X X X X X}
mesh level 
& $c_{0,\ga_{b},\pone}$ 
& $c_{2,\ga_{t,l,r},\rt}$ 
& $c_{1,\ga_{t,l,r},\ned}$  
& $c_{1,\ga_{b},\ned}$ 
& $c_{2,\ga_{b},\rt}$ 
& $c_{0,\ga_{t,l,r},\pone}$ \\
\hline
\hline
1 & 0.63267458 & 0.63798842 & 0.63798842 & 0.28486798 & 0.28486798 & 0.27318834 \\
2 & 0.63560893 & 0.63696095 & 0.63696095 & 0.28473767 & 0.28473767 & 0.28172459 \\
3 & 0.63636500 & 0.63670501 & 0.63670501 & 0.28471277 & 0.28471277 & 0.28395286 \\
4 & 0.63655592 & 0.63664108 & 0.63664108 & 0.28470693 & 0.28470693 & 0.28451652 \\
5 & 0.63660380 & 0.63662510 & 0.63662510 & 0.28470549 & 0.28470549 & 0.28465786 \\
6 & 0.63661578 & 0.63662110 & 0.63662110 & 0.28470514 & 0.28470514 & 0.28469323 \\
7 & 0.63661877 & 0.63662011 & 0.63662011 & 0.28470505 & 0.28470505 & 0.28470207 \\
\hline
$\infty$ & 0.63661977 & 0.63661977 & 0.63661977 & 0.28470501 & 0.28470501 & 0.28470501 \\
\end{tabularx}
\vspace{1mm}
\caption{Constants computed for the unit square domain and mixed boundary conditions.}
\label{ta:unitsquare_mixed}
\begin{tabularx}{\textwidth}{X | X X X X X X}
mesh level 
& $c_{0,\emptyset,\pone}$
& $c_{2,\ga,\rt}$
& $c_{1,\ga,\ned}$
& $c_{1,\emptyset,\ned}$
& $c_{2,\emptyset,\rt}$
& $c_{0,\ga,\pone}$\\
\hline
\hline
1 & 0.39156654 & 0.43611331 & 0.43611331 & 0.16795692 & 0.16795692 & 0.13325394 \\
2 & 0.40370423 & 0.42045050 & 0.42045050 & 0.16377267 & 0.16377267 & 0.15232573 \\
3 & 0.40850306 & 0.41492017 & 0.41492017 & 0.16214127 & 0.16214127 & 0.15838355 \\
4 & 0.41038725 & 0.41287500 & 0.41287500 & 0.16148392 & 0.16148392 & 0.16020361 \\
5 & 0.41112643 & 0.41209870 & 0.41209870 & 0.16121879 & 0.16121879 & 0.16076463 \\
6 & 0.41141712 & 0.41179918 & 0.41179918 & 0.16111230 & 0.16111230 & 0.16094566 \\
7 & 0.41153175 & 0.41168242 & 0.41168242 & 0.16106970 & 0.16106970 & 0.16100698 \\
\end{tabularx}
\vspace{1mm}
\caption{Constants computed for the L-shape domain and full boundary conditions.}
\label{ta:Lshape}
\begin{tabularx}{\textwidth}{X | X X X X X X}
mesh level 
& $c_{0,\ga_{b},\pone}$ 
& $c_{2,\ga_{t,l,r},\rt}$ 
& $c_{1,\ga_{t,l,r},\ned}$  
& $c_{1,\ga_{b},\ned}$ 
& $c_{2,\ga_{b},\rt}$ 
& $c_{0,\ga_{t,l,r},\pone}$ \\
\hline
\hline
1 & 0.55287499 & 0.58356116 & 0.58356116 & 0.24038804 & 0.24038804 & 0.21444362 \\
2 & 0.56332946 & 0.57483917 & 0.57483917 & 0.23765352 & 0.23765352 & 0.22916286 \\
3 & 0.56716139 & 0.57152377 & 0.57152377 & 0.23648111 & 0.23648111 & 0.23363643 \\
4 & 0.56857589 & 0.57025101 & 0.57025101 & 0.23597974 & 0.23597974 & 0.23498908 \\
5 & 0.56910703 & 0.56975715 & 0.56975715 & 0.23577100 & 0.23577100 & 0.23541318 \\
6 & 0.56930976 & 0.56956402 & 0.56956402 & 0.23568569 & 0.23568569 & 0.23555262 \\
7 & 0.56938813 & 0.56948808 & 0.56948808 & 0.23565122 & 0.23565122 & 0.23560066 \\
\end{tabularx}
\vspace{1mm}
\caption{Constants computed for the L-shape domain and mixed boundary conditions.}
\label{ta:Lshape_mixed}
\end{table}

\subsection{3D Computations}  
\label{section:3D}

We present two benchmarks with the unit cube and the Fichera corner domain, 
the first with known and the second with unknown values of the constants. 
Their coarse (level 1) meshes are displayed in Figure \ref{meshes3D}.
       
For the unit cube we have by Remark \ref{stateoftheartconstestrem3D}
exact values
$$c_{0,\ga}=\frac{1}{\sqrt{3}\pi}\approx0.18377629,\quad
c_{1,\ga}=c_{1,\emptyset}=\frac{1}{\sqrt{2}\pi}\approx0.22507907,\quad 
c_{0,\emptyset}=\frac{1}{\pi}\approx0.31830988,$$
and our approximative values converge to them, see Table \ref{ta:unitcube}. 
For one case of mixed boundary conditions (with missing boundary part $\ga_{b}$) we have
by Remark \ref{stateoftheartconstestrem3D} and \eqref{constval3Da}, \eqref{constval3Db} exact values
$$c_{0,\ga_{t,l,r,f,\bk}}=\frac{2}{3\pi}\approx0.21220659,\quad 
c_{1,\ga_{t,l,r,f,\bk}}=c_{1,\ga_{b}}=\frac{2}{\sqrt{5}\pi}\approx0.28470501,\quad
c_{0,\ga_{b}}=\frac{2}{\pi}\approx0.63661977,$$
and our approximative values converge again to them, 
see Table \ref{ta:unitcube_mixed}. 
Approximative values for the Fichera corner domain are provided 
in Tables \ref{ta:Fichera} and \ref{ta:Fichera_mixed}. 
We notice a slightly lower than quadratic convergence 
of all constants with respect to the mesh size $h$. 
Geometrical parameters of tetrahedral meshes used for the unit cube domain are given in Table \ref{ta:unitcube_meshes}.  

\begin{table}[h]
\begin{tabularx}{\textwidth}{X | X X X X X X}
mesh level 
& $c_{0,\emptyset,\pone}$
& $c_{2,\ga,\rt}$
& $c_{1,\ga,\ned}$
& $c_{1,\emptyset,\ned}$
& $c_{2,\emptyset,\rt}$
& $c_{0,\ga,\pone}$\\
\hline
\hline
1 & 0.31114284 & 0.32265677 & 0.22964649 & 0.22346361 & 0.18305860 & 0.16330104 \\
2 & 0.31500347 & 0.31939334 & 0.22727295 & 0.22461307 & 0.18369611 & 0.17685247 \\
3 & 0.31720303 & 0.31857551 & 0.22577016 & 0.22497862 & 0.18375776 & 0.18178558 \\
4 & 0.31799426 & 0.31837527 & 0.22526682 & 0.22505528 & 0.18377095 & 0.18324991 \\
\hline
$\infty$ & 0.31830988 & 0.31830988 & 0.22507907 & 0.22507907 & 0.18377629 & 0.18377629 \\
\end{tabularx}
\vspace{1mm}
\caption{Constants computed for the unit cube domain and full boundary conditions.}
\label{ta:unitcube}
\begin{tabularx}{\textwidth}{X | X X X X X X}
mesh level 
& $c_{0,\ga_{b},\pone}$ 
& $c_{2,\ga_{t,l,r,f,\bk},\rt}$  
& $c_{1,\ga_{t,l,r,f,\bk},\ned}$  
& $c_{1,\ga_{b},\ned}$ 
& $c_{2,\ga_{b},\rt}$ 
& $c_{0,\ga_{t,l,r,f,\bk},\pone}$ \\
\hline
\hline
1 & 0.63279353 & 0.63799454  & 0.28810408  & 0.28568645  & 0.21207495  & 0.19466267  \\
2 & 0.63563506 & 0.63694323  & 0.28621730  & 0.28506833  & 0.21221199  & 0.20590030  \\
3 & 0.63636820 & 0.63669754  & 0.28518535  & 0.28483451  & 0.21220585  & 0.21033840  \\
4 & 0.63655623 & 0.63663874  & 0.28483637  & 0.28474355  & 0.21220553  & 0.21170560  \\
\hline
$\infty$ & 0.63661977 & 0.63661977 & 0.28470501 & 0.28470501 & 0.21220659 & 0.21220659 \\
\end{tabularx}
\vspace{1mm}
\caption{Constants computed for the unit cube domain and mixed boundary conditions.}
\label{ta:unitcube_mixed}
\begin{tabularx}{\textwidth}{X | X X X X X X}
mesh level 
& $c_{0,\emptyset,\pone}$
& $c_{2,\ga,\rt}$
& $c_{1,\ga,\ned}$
& $c_{1,\emptyset,\ned}$
& $c_{2,\emptyset,\rt}$
& $c_{0,\ga,\pone}$\\
\hline
\hline
1 & 0.34328060 & 0.37118723  & 0.30375245  & 0.26905796  & 0.15938388  & 0.12490491  \\
2 & 0.35193318 & 0.36341148  & 0.28961049  & 0.27500043  & 0.15638922  & 0.14329827  \\
3 & 0.35628919 & 0.36072519  & 0.28329415  & 0.27728510  & 0.15500593  & 0.15042074  \\
4 & 0.35808207 & 0.35976508  & 0.28054899  & 0.27811443  & 0.15444291  & 0.15286199  \\
\end{tabularx}
\vspace{1mm}
\caption{Constants computed for the Fichera corner domain and full boundary conditions.}
\label{ta:Fichera}
\begin{tabularx}{\textwidth}{X | X X X X X X}
mesh level 
& $c_{0,\ga_{b},\pone}$ 
& $c_{2,\ga_{t,l,r,f,\bk},\rt}$  
& $c_{1,\ga_{t,l,r,f,\bk},\ned}$  
& $c_{1,\ga_{b},\ned}$ 
& $c_{2,\ga_{b},\rt}$ 
& $c_{0,\ga_{t,l,r,f,\bk},\pone}$ \\
\hline
\hline
1 & 0.59192242 & 0.60790729  & 0.32388929  & 0.30017867  & 0.19741476  & 0.17047720  \\
2 & 0.59806507 & 0.60397104  & 0.31363043  & 0.30335884  & 0.19584527  & 0.18566517  \\
3 & 0.60032617 & 0.60255178  & 0.30884202  & 0.30465706  & 0.19504508  & 0.19157456  \\
4 & 0.60117716 & 0.60202588  & 0.30682949  & 0.30515364  & 0.19471006  & 0.19355739  \\
\end{tabularx}
\vspace{1mm}
\caption{Constants computed for the Fichera corner domain and mixed boundary conditions.}
\label{ta:Fichera_mixed}
\end{table}

\subsection{Testing of the Monotonicity Properties}
\label{section:monotonicity}

We perform some monotonicity tests on the constants
depending on the respective boundary conditions, i.e., we display the mapping 
$$\gan\longmapsto(c_{0,\gat},c_{0,\gan},c_{1,\gat},c_{1,\gan},c_{2,\gat},c_{2,\gan})$$
for a monotone increasing sequence of $\gan$. 
Figures \ref{neumannsequence2D} and \ref{neumannsequence3D} depict examples of such sequences in 2D/3D. 
The boundary part $\gan$ is represented discretely as a set of Neumann edges in 2D
or a set of Neumann faces in 3D. 
Boundary faces or edges are checked for their connectivity and a breadth-first search (BFS) 
algorithm is applied to order them in a sequence.  
All constants are then evaluated for every element of the sequence 
and the results are displayed in Figures \ref{monoconstants2DusLs} and \ref{monoconstants3DucFc}.

\begin{figure}
\includegraphics[width=0.302\textwidth]{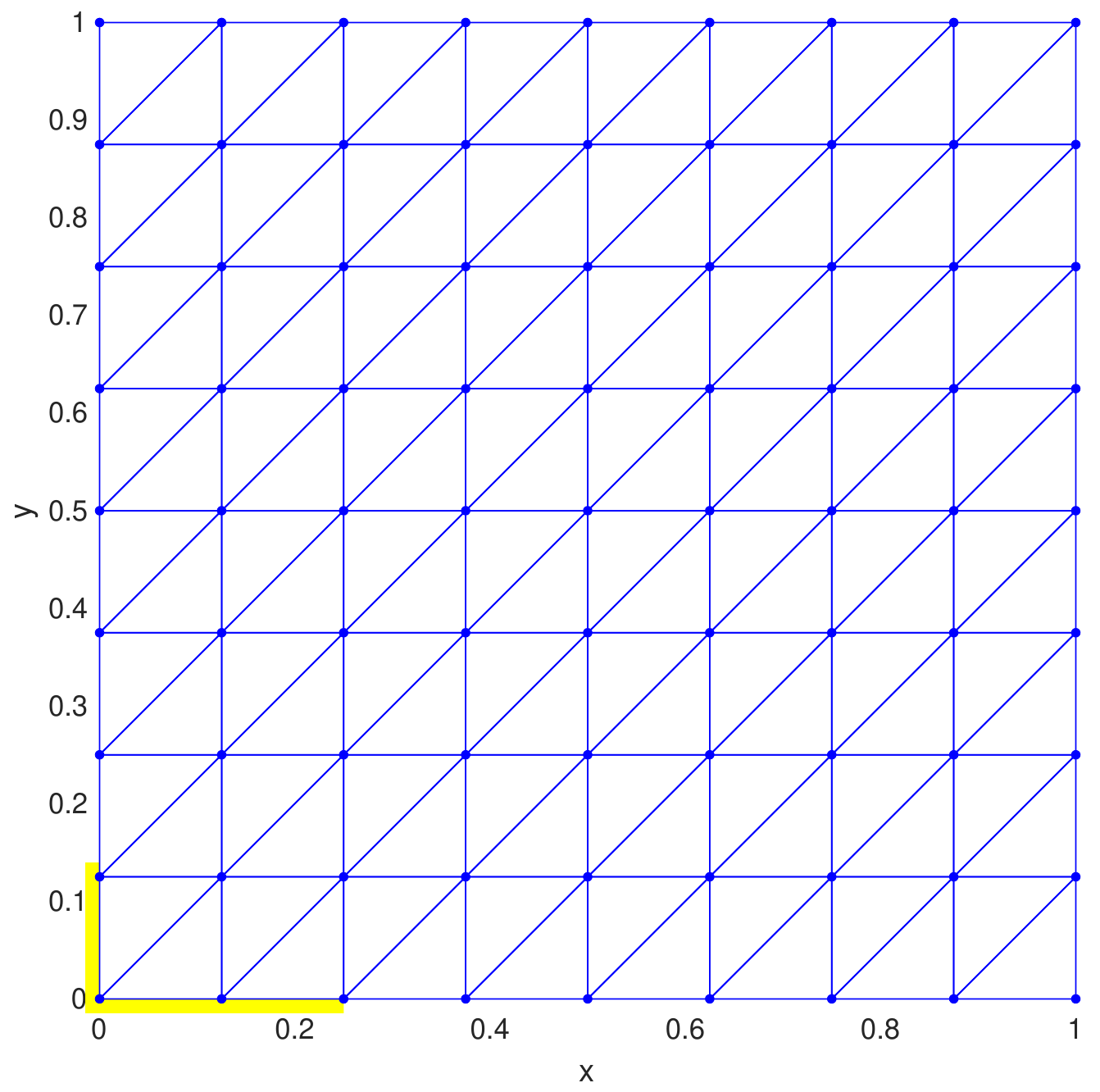}
\hspace{0.03\textwidth}
\includegraphics[width=0.3\textwidth]{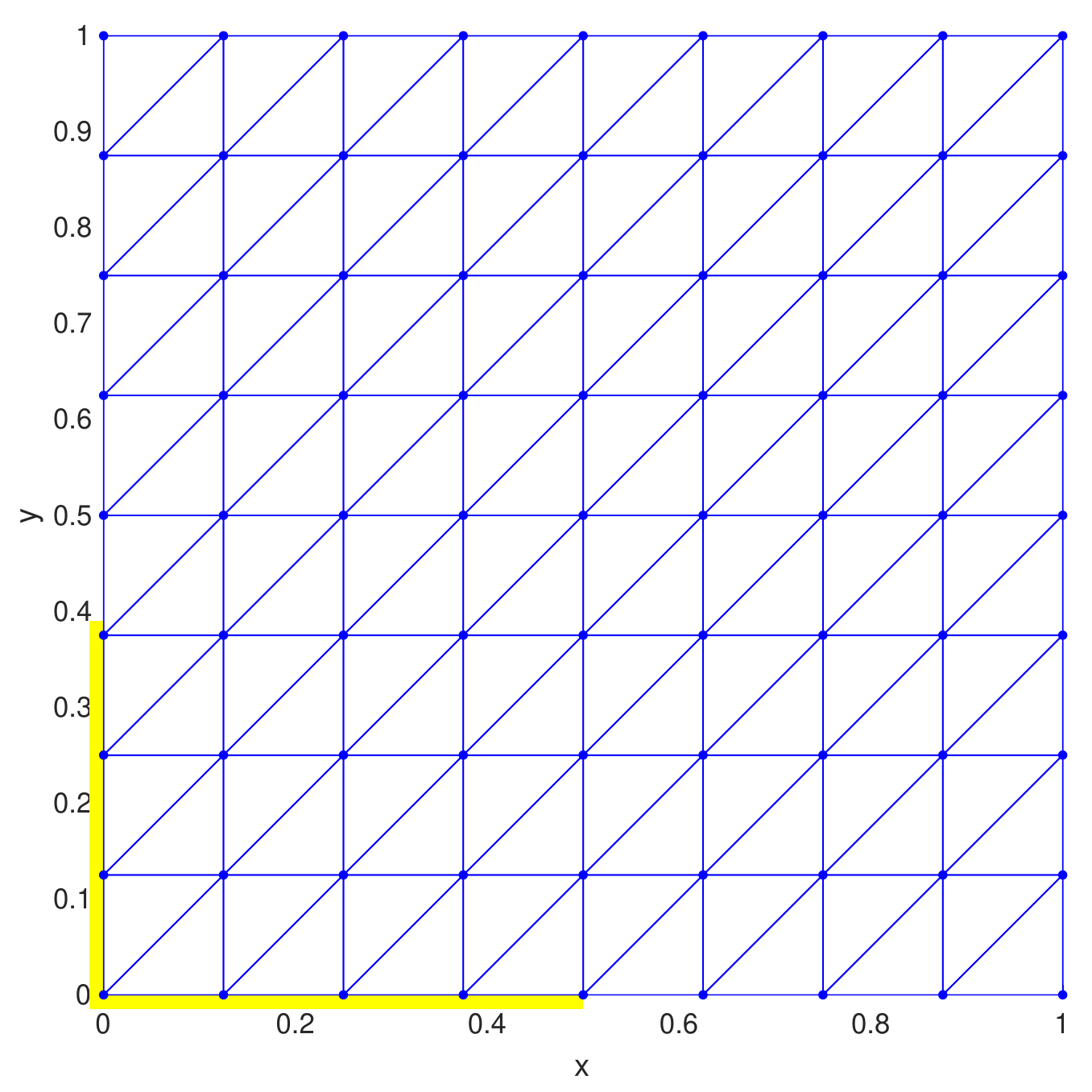}
\hspace{0.03\textwidth}
\includegraphics[width=0.3\textwidth]{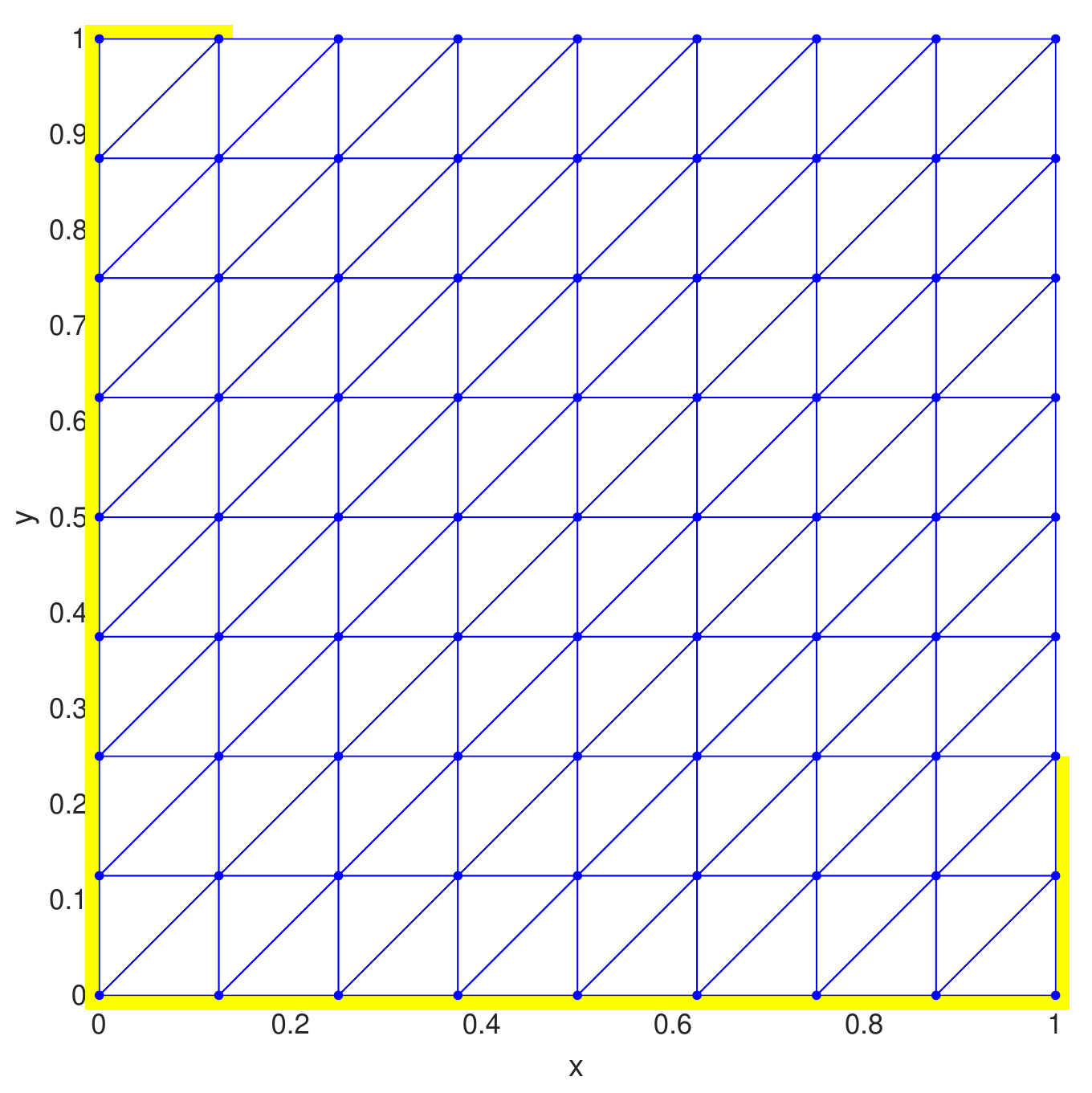}
\caption{Elements of a monotone increasing sequence of  Neumann edges $\gan$  with 3 (left), 7 (middle) and 19 (right) Neumann edges marked in yellow.
A full boundary $\Gamma$ of the considered level 2 mesh of the unit square domain consists of 32 edges.}
\label{neumannsequence2D}
\vspace{0.5cm}
\includegraphics[width=0.305\textwidth]{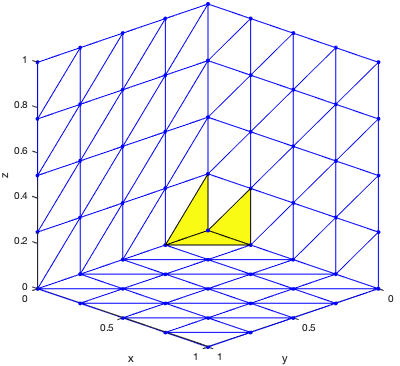}
\hspace{0.03\textwidth}
\includegraphics[width=0.3\textwidth]{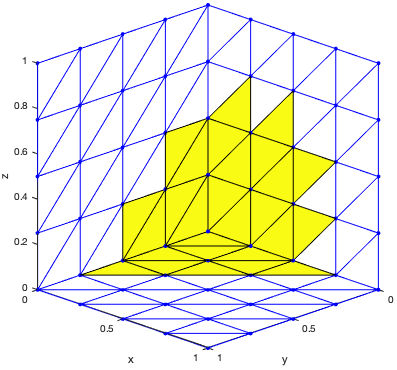}
\hspace{0.03\textwidth}
\includegraphics[width=0.305\textwidth]{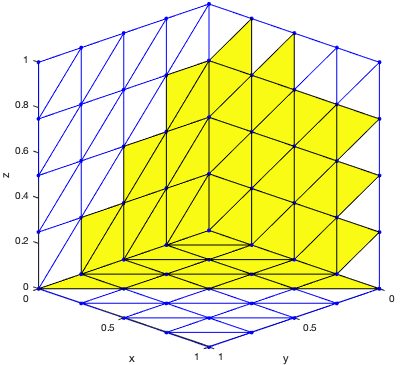}
\caption{Elements of a monotone increasing sequence of  Neumann faces $\gan$  with 3 (left), 27 (middle) and 51 (right) Neumann faces marked in yellow.
A full boundary $\Gamma$ of the considered level 1 mesh of the unit cube domain consists of 192 faces.}
\label{neumannsequence3D}
\end{figure}

\begin{figure}[h]
\includegraphics[width=0.49\textwidth]{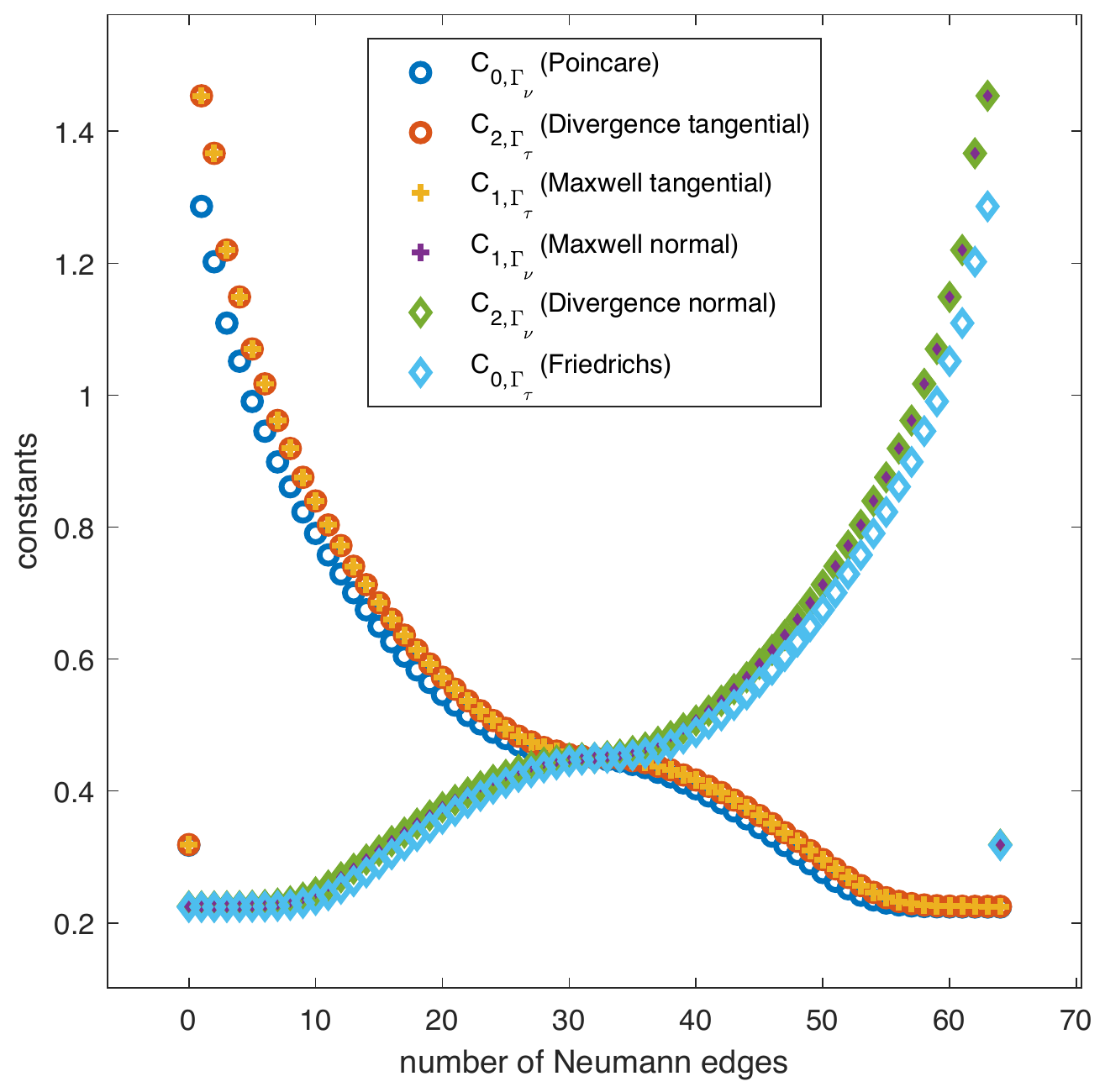}
\includegraphics[width=0.486\textwidth]{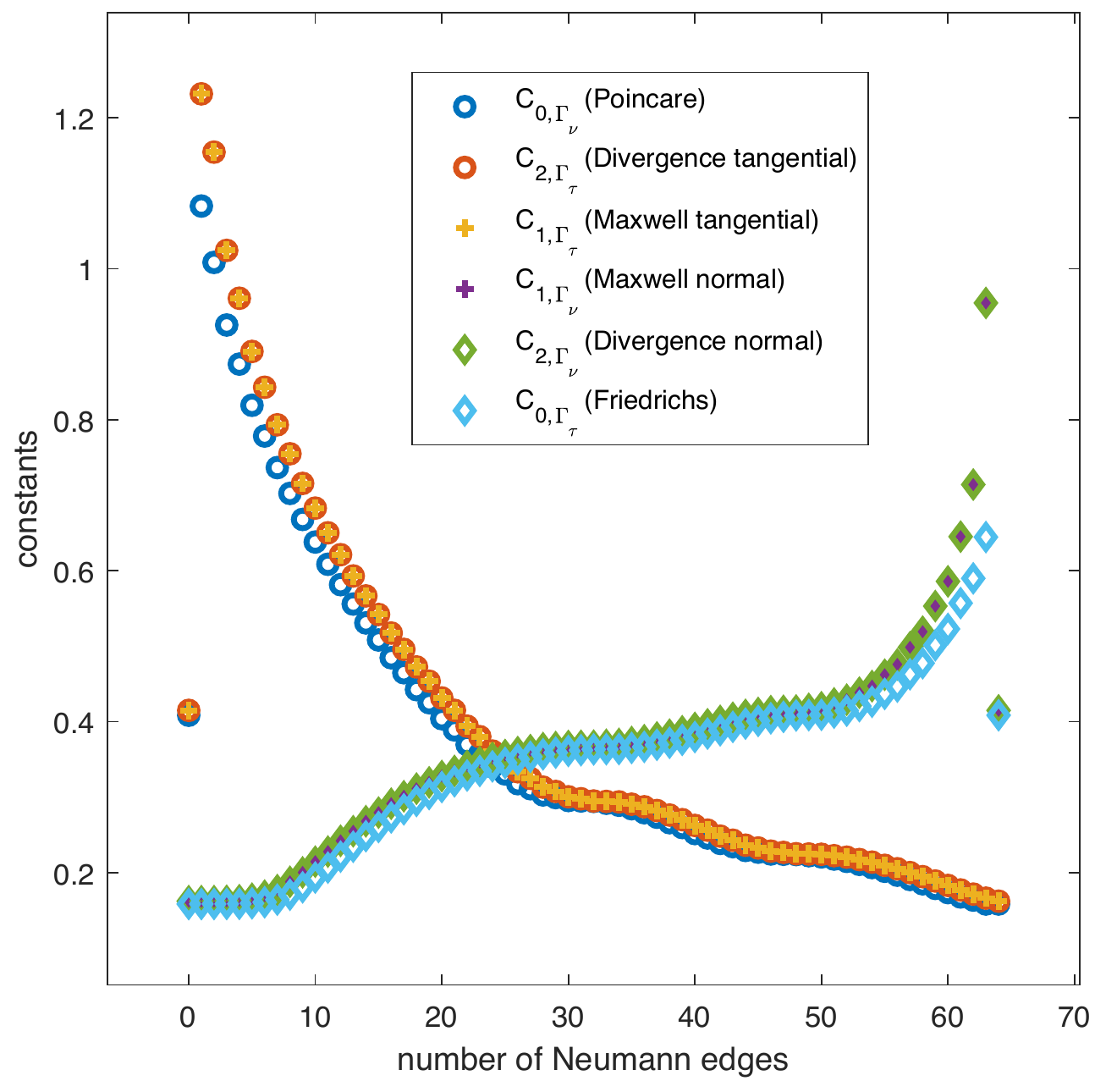}
\caption{Constants computed for the level 3 meshes of the unit square domain (left) and of the L-shape domain (right) 
- monotonicity test for a monotone increasing sequence of Neumann edges $\gan$.}
\label{monoconstants2DusLs}
\vspace*{5mm}
\includegraphics[width=0.49\textwidth]{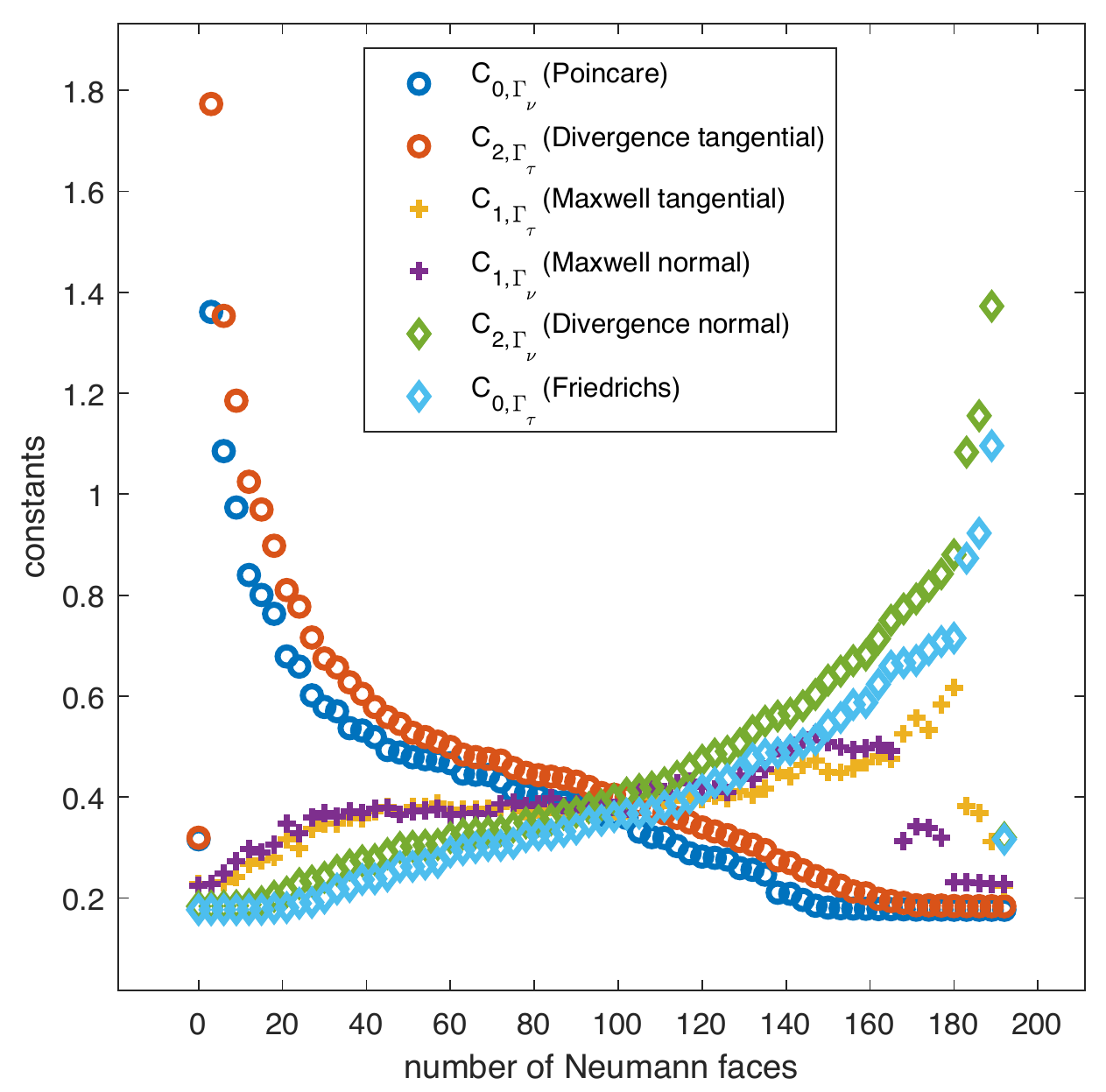}
\includegraphics[width=0.491\textwidth]{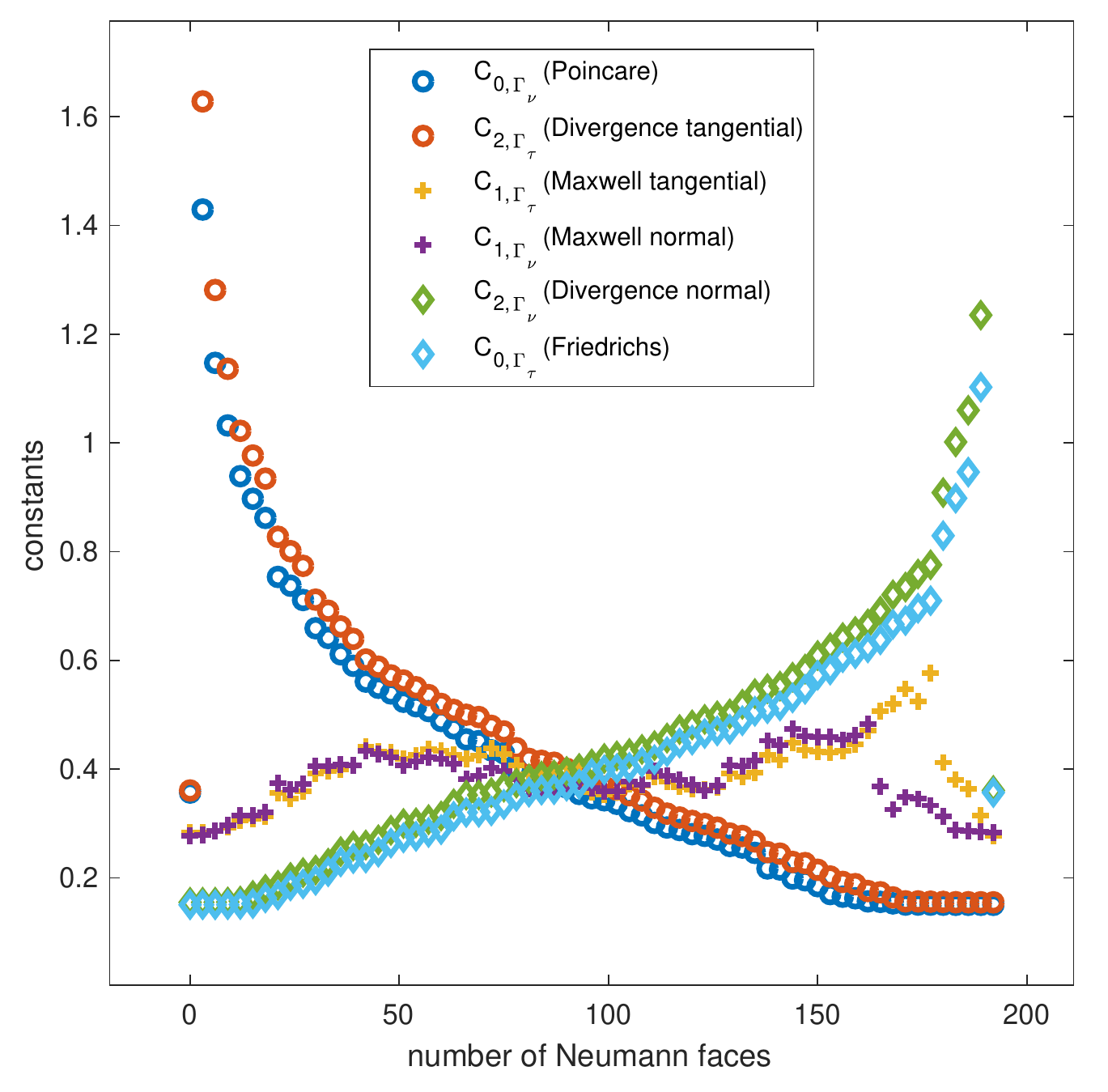}
\caption{Constants computed for the level 1 meshes of unit cube domain (left) and of the Fichera corner domain (right)
- monotonicity test for a monotone increasing sequence of Neumann faces $\gan$.}
\label{monoconstants3DucFc}
\end{figure}

\subsection{Computational Details and MATLAB Code}
It is more computationally demanding to evaluate divergence and Maxwell constants than Laplace constants, 
since the numbers of faces (in 3D) and edges are higher than the number of nodes, cf. eg. Table \ref{ta:unitsquare_meshes} and Table \ref{ta:unitcube_meshes}.

A generalized eigenvalue system
\begin{equation}
K v = \lambda^2 \, M v 
\label{generalized_EV}
\end{equation}
with a positive semidefinite and symmetric matrix $K\in\reals^{n \times n}$ 
and a positive definite and symmetric matrix $M\in\reals^{n \times n}$ 
is solved for a smallest positive eigenvalue $\lambda^2 > 0$. 
We apply two computational techniques.

\subsubsection{A nested iteration technique}
 An eigenvalue evaluated on a coarser mesh (eg. by the second technique explained below) 
is used as initial guess on a finer (uniformly refined) mesh, 
where an inbuilt MATLAB function \verb+eigs+ is applied for the search of the closest eigenvalue.
Without additional preconditioning (multigrid, domain decompositions) of eigenvalue solvers we can efficiently find smallest positive eigenvalues for all considered meshes.

However, it was noticed this technique did not converge for some cases of mixed boundary conditions in 3D because a sequence of corresponding Laplace constants did not form a monotone sequence in the monotonicity test.  Then, since the dimension of the kernel of a corresponding stiffness matrix is known ($0$ or $1$), we simply compute the smallest eigenvalue or two smallest eigenvalues with 0 being the smallest eigenvalue.  

\subsubsection{A projection to the range of $K$}
We apply the QR-decomposition of $K$ in the form 
\begin{equation}
K E=\tilde Q \tilde R,
\label{qr_decomposition}
\end{equation}
where $E \in\reals^{n \times n}$ is a permutation matrix, 
$\tilde Q \in\reals^{n \times n}$ is an orthogonal matrix 
and $\tilde R \in\reals^{n \times n}$ is an upper triangular matrix 
with diagonal entries ordered in decreasing order as 
\begin{equation*}
|\tilde R_{1,1}| \geq \dots \geq |\tilde R_{r,r}| \geq \dots \geq  |\tilde R_{n,n}|.
\label{sequence}
\end{equation*}
The number $r \leq n $ of nonzero entries of the sequence above then determines 
the range $K$ and all rows of $\tilde R$ with indices larger than $r$ are zero rows, cf. Figure \ref{qr_unit_cube_level_1} and  Figure \ref{qr_unit_cube_level_2}.
Therefore, we also have
\begin{equation}
K  E=Q  R
\label{qr_decomposition_restricted}
\end{equation}
where $Q \in\reals^{n \times r}$ is a restriction of $\tilde Q$ to its first $r$ columns 
and $R \in\reals^{r \times n}$ a restriction of $\tilde R$ to its first $r$ rows. 
Then, a mapping $v=Q z$ projects a (column) vector $z\in\reals^{r}$ to the range of $K$ 
and the generalized eigenvalue system \eqref{generalized_EV} 
to $$K Q z= \lambda^2 \, M Q z. $$
The multiplication of both sides by $Q^{\top} M^{-1} $ transforms the above relation  (since $Q^{\top}Q$ is an identity matrix of size $r \times r$) to a standard eigenvalue problem
\begin{equation}
Q^{\top} M^{-1} K Q z = \lambda^2 z.
\label{standard_EV}
\end{equation}
In view of \eqref{qr_decomposition_restricted}, the symmetry of $K$ and the orthogonality of the permutation matrix $E$, 
it holds $$K Q= (Q R E^{\top}) Q = (Q R E^{\top})^{\top} Q = E R^{\top} Q^{\top} Q = E R^{\top}$$ and this formula is applied in  our practical computations. A matrix $M^{-1}$ is full and expensive to compute, its memory storage is large and the multiplication with $Q^{\top} M^{-1}$ is costly. Therefore, this projection technique can only be applied for coarser meshes.
\begin{figure}[t]
\includegraphics[width=0.99\textwidth]{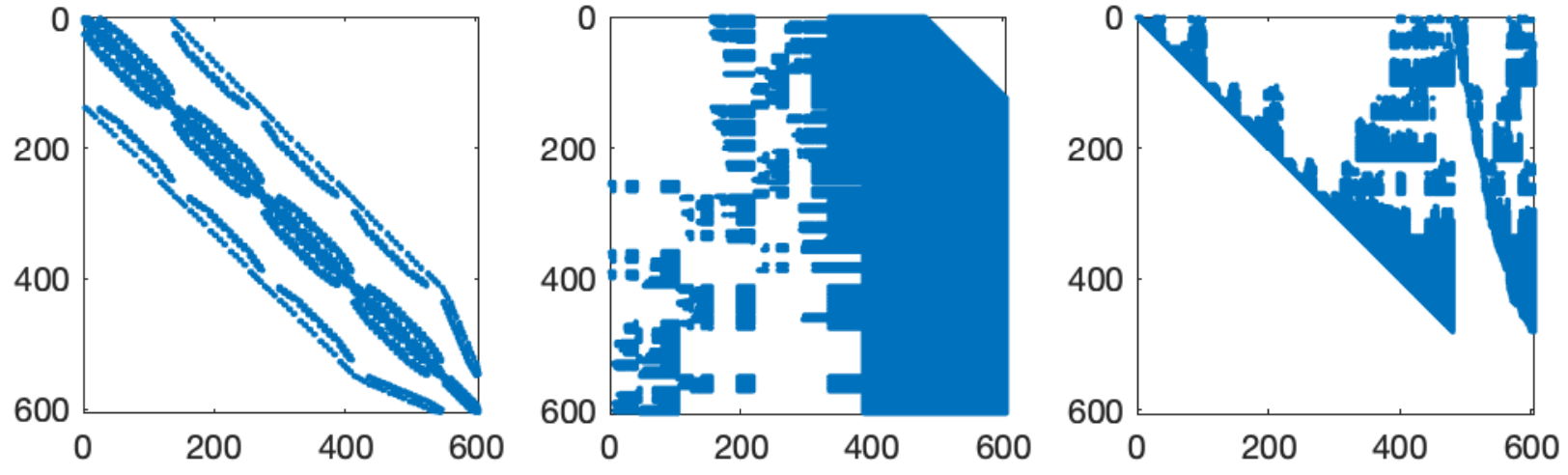}
\caption{An example of a positive semidefinite matrix $K\in\reals^{604 \times 604}$ (left) taken as a $\stiff^\ned$ matrix from level 1 mesh discretization of the unit cube domain and matrices $\tilde Q$ (middle) and $\tilde R$ (right) from the QR decomposition \eqref{qr_decomposition}.
The number of last zero rows of $\tilde R$ is 124 and it determines the dimension of the kernel of $K$. Therefore, the dimension of the range of $K$ is 480. 
}
\label{qr_unit_cube_level_1}
\end{figure}
\begin{figure}[t]
\includegraphics[width=0.99\textwidth]{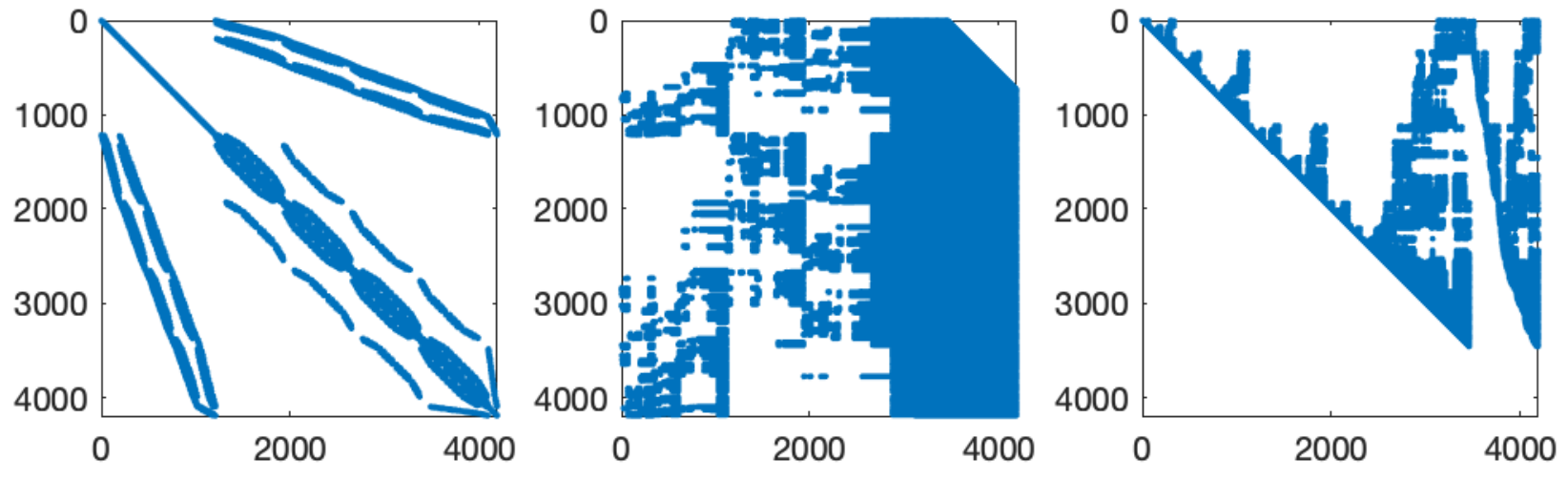}
\caption{An example of a positive semidefinite matrix $K\in\reals^{4184 \times 4184}$ (left) taken as a $\stiff^\ned$ matrix from level 2 mesh discretization of the unit cube domain and matrices $\tilde Q$ (middle) and $\tilde R$ (right) from the QR decomposition \eqref{qr_decomposition}.
The number of last zero rows of $\tilde R$ is 728 and it determines the dimension of the kernel of $K$. Therefore, the dimension of the range of $K$ is 3456.
}
\label{qr_unit_cube_level_2}
\end{figure}

\begin{table}[h]
\begin{tabulary}{\textwidth}{C | R R R  R}
mesh level 
& $\quad$ elements 
& $\qquad$nodes 
&$\qquad$  edges
& $\quad$ boundary edges
\\
\hline
\hline
1 & 32 & 25 & 56 & 16  $\qquad$\\
2 & 128 & 81 & 208 & 32  $\qquad$\\
3 & 512 & 289 & 800  & 64  $\qquad$\\
4 & 2.028 & 1.089 & 3.136 & 128  $\qquad$\\
5 & 8.192 & 4.225 & 12.416 & 256 $\qquad$ \\
6 & 32.768 & 16.641 & 49.408 & 512  $\qquad$\\
7 & 131.072 & 66.049 & 197.120 & 1.024  $\qquad$\\
\end{tabulary}
\vspace{1mm}
\caption{Discretization of the unit square domain by uniform triangular meshes.
}
\label{ta:unitsquare_meshes}
\begin{tabulary}{\textwidth}{C | R R R R  R}
mesh level 
& $\quad$ elements 
& $\qquad$nodes 
&$\qquad$  edges
& $\qquad$ faces 
& $\quad$ boundary faces 
\\
\hline
\hline
1 & 384 & 125 & 604 &864 & 192 $\qquad$\\
2 & 3.072 & 729 & 4.184 &6.528 & 768 $\qquad$\\
3 & 24.576 & 4.913 &31.024  &50.688 & 3.072 $\qquad$\\
4 & 196.608 & 35.937 & 238.688 &399.360 & 12.288 $\qquad$\\
\end{tabulary}
\vspace{1mm}
\caption{Discretization of the unit cube domain by uniform tetrahedral meshes.}
\label{ta:unitcube_meshes}
\end{table}

\subsubsection{A MATLAB code}
Numerical evaluations are based on finite element assemblies 
from \cite{anjamvaldmanfastmatlab, rahmanvaldmanfastmatlab} 
and also utilizes a 3D cube mesh and mesh visualizations from \cite{cermaksysalavaldman2019}. 
The code is freely available for download and testing at:
\begin{center}
\url{https://www.mathworks.com/matlabcentral/fileexchange/23991}
\end{center}
It can be easily modified to other domains and boundary conditions. The starting scripts for testing are  
\verb+start_2D+ and \verb+start_3D+.
To a given mesh, it automatically determines its boundary. In 2D, the code can also visualize eigenfunctions, 
see Figure  \ref{eigenfunctions} for the case of the L-shape domain. 

\section{Discussion of the Numerical Results and Conclusions}

Our numerical results, especially in 3D, 
did verify all the theoretical assertions of Theorem \ref{stateoftheartconstest},
see also Remark \ref{stateoftheartconstestrem3D} and Remark \ref{stateoftheartconstestrem2D}, in particular, 
\begin{itemize}
\item
the monotone dependence of the Poincar\'e-Friedrichs and divergence constants 
on the boundary conditions, i.e., the monotonicity of the mapping
$$\gan\mapsto c_{0,\gat}=c_{2,\gan},$$
\item
the `independence' of the Maxwell constants on the boundary conditions, i.e.,
$$\forall\,\gat\subset\ga\qquad
c_{1,\gat}=c_{1,\gan},$$
\item
as well as the boundedness of the full tangential and normal Maxwell constants 
by the Poincar\'e constant for convex $\om$, i.e.,
$$c_{1,\ga}=c_{1,\emptyset}\leq c_{0,\emptyset}=c_{2,\ga}.$$
\end{itemize}
While the first two assertions hold for general bounded Lipschitz domains
and Lipschitz interfaces,
the third assertion is analytically proved only for convex domains and the full boundary conditions.
In our numerical experiments, the unit cube served as a prototype for a convex domain, 
and we picked the Fichera corner domain as a typical example of a non-convex domain, 
see Figure \ref{meshes3D} for both initial meshes.

\subsection{Extended Inequalities}

To our surprise, even for mixed boundary conditions 
and for non-convex geometries, the \emph{extended inequalities}
\begin{align}
\label{PVconjPFM}
c_{0,\ga}
\leq\min\{c_{0,\gat},c_{0,\gan}\}
\leq c_{1,\gat}=c_{1,\gan}
\leq\max\{c_{0,\gat},c_{0,\gan}\}
\leq\sup_{\gat\neq\emptyset}c_{0,\gat}
=\sup_{\gan\neq\ga}c_{2,\gan}
\end{align}
seem to hold for our examples, see Figure \ref{monoconstants3DucFc}.
In these special cases the Maxwell constants are always in between the 
Poincar\'e-Friedrichs (Laplace) constants. 
We emphasise that our examples possess (piecewise) vanishing curvature.
It remains an open question if \eqref{PVconjPFM} is true
- at least partially - in general or, e.g., for polyhedra.
Moreover, if $\gat$ approaches $\emptyset$,
the Poincar\'e-Friedrichs constants $c_{0,\gat}$ seem to be bounded, i.e.,
the suprema in \eqref{PVconjPFM} appear to be bounded,
although a kernel of dimension $1$ (constants) is approximated.


\begin{figure}[h]
\includegraphics[scale=0.49]{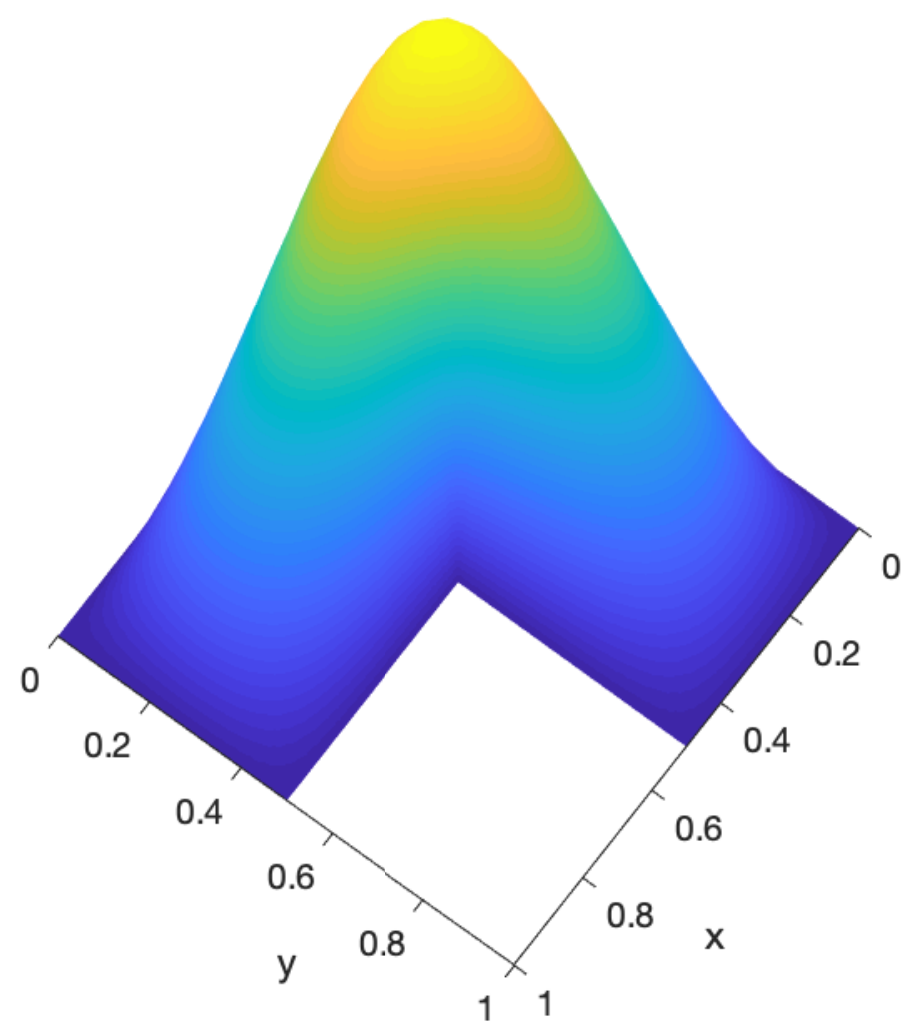}
\includegraphics[scale=0.692]{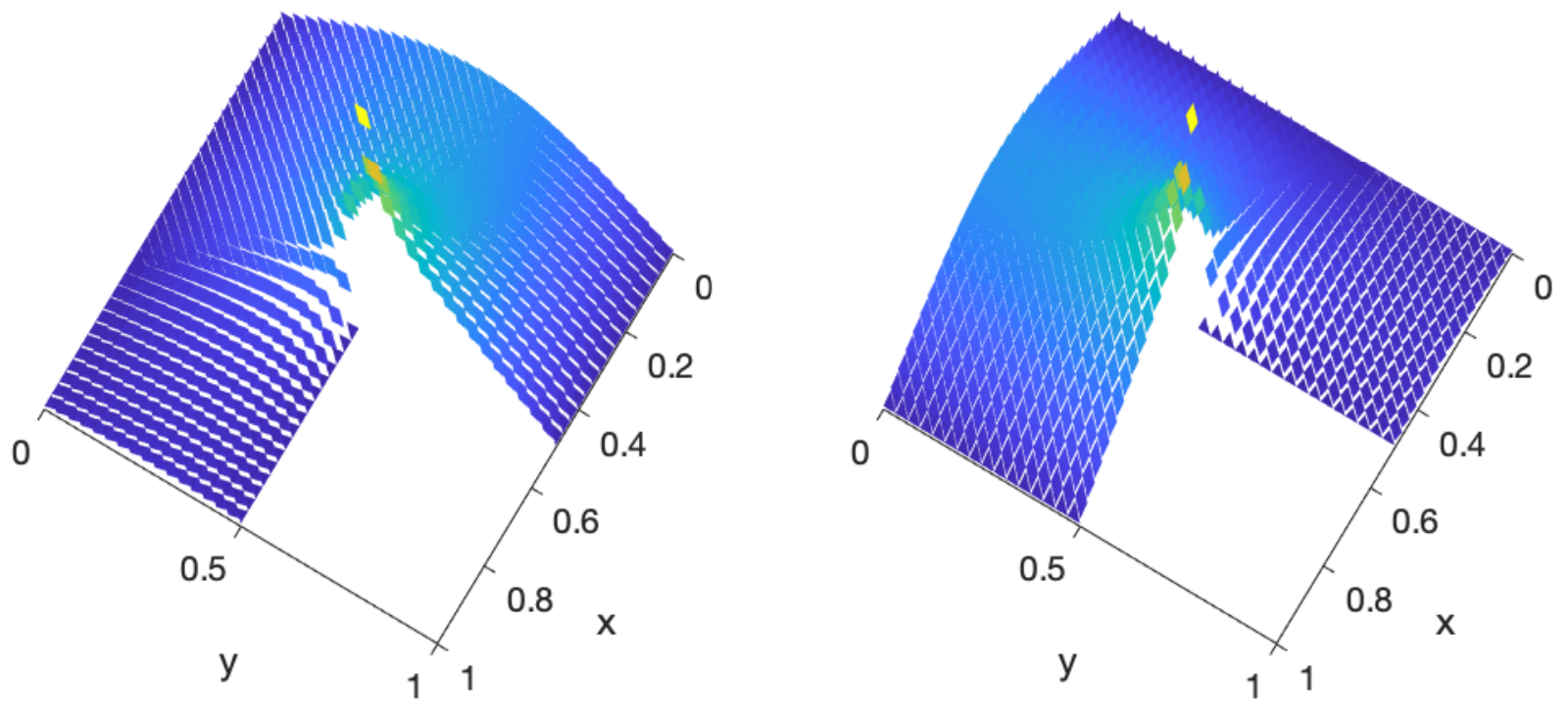}
\includegraphics[scale=0.647]{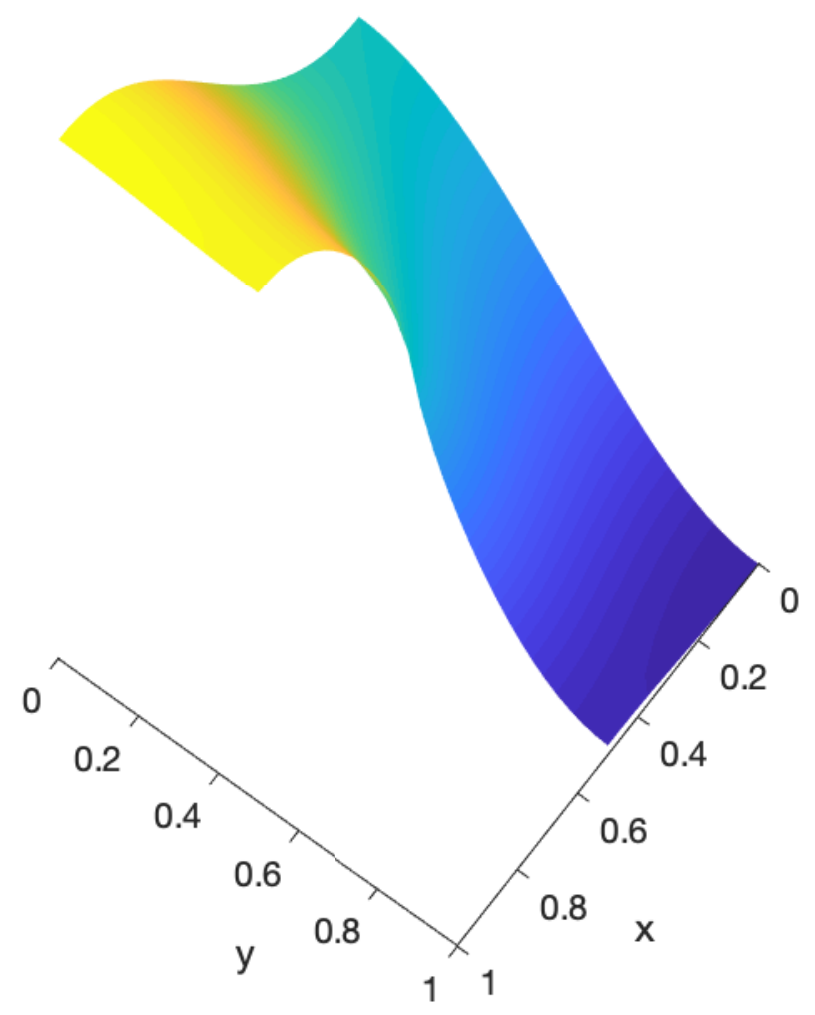}
\includegraphics[scale=0.637]{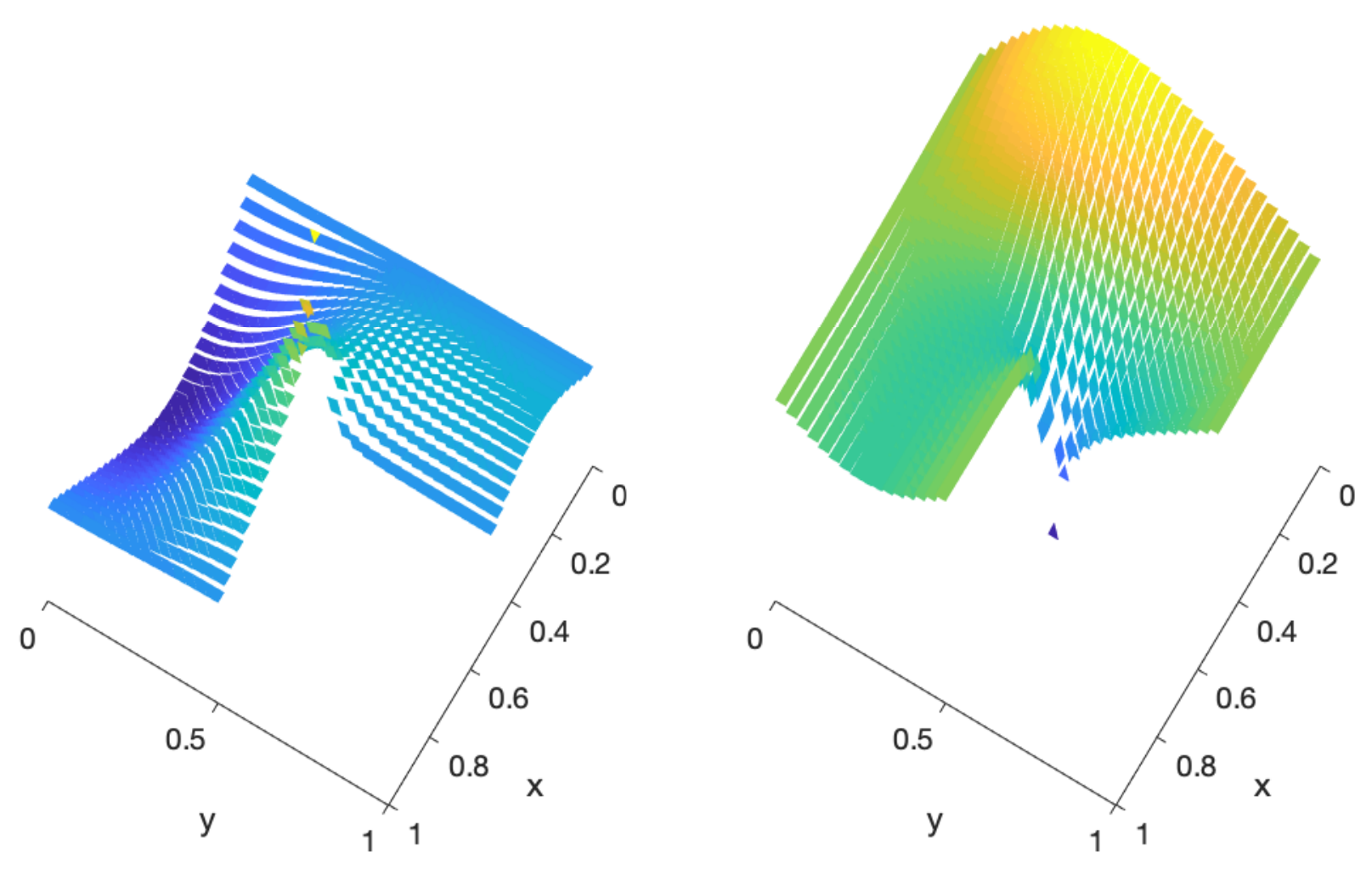}
\caption{Eigenfunctions - Friedrichs (top left), Poincar\'e (bottom left), tangential Maxwell (top middle and right), normal Maxwell (bottom middle and right) -
for the L-shape domain with full boundary conditions.} \label{eigenfunctions}
\end{figure}

\subsubsection{Hints for the Extended Inequalities}

We note the well known integration by parts formula
\begin{align}
\label{intpartgradrotdiv}
\norm{\grad{}E}_{\L{2}(\om)}^2
&=\norm{\rot{}E}_{\L{2}(\om)}^2
+\norm{\div{}E}_{\L{2}(\om)}^2,
\end{align}
being valid for all vector fields $E\in\H{}{\ga}(\grad{},\om)$, 
the closure of $\om$-compactly supported test fields, see \eqref{Hgatdef}.
Using a more sophisticated integration by parts formula from 
\cite[Corollary 6]{bauerpaulytangnormkorn},
which has been proved already in, e.g., \cite[Theorem 2.3]{costabeldaugemaxwelllameeigenvaluespolyhedra}
for the case of full boundary conditions,
we see that \eqref{intpartgradrotdiv} remains true 
for \emph{polyhedral} domains $\om$ and for vector fields 
\begin{align}
\label{addregE}
E\in\H{}{\gat,\gan}(\grad{},\om)
:=\ol{\C{\infty}{\gat,\gan}(\overline{\om})}^{\H{}{}(\grad{},\om)}
\subset\H{}{}(\grad{},\om)\cap\H{}{\gat}(\rot{},\om)\cap\H{}{\gan}(\div{},\om),
\end{align}
where
$$\C{\infty}{\gat,\gan}(\ol{\om})
:=\big\{\Phi|_{\om}:\Phi\in\C{\infty}{}(\reals^{3}),\,
\text{\rm supp}\,\Phi\text{\rm  compact in }\reals^{3},\,
n\times\Phi|_{\gat}=0,\,n\cdot\Phi|_{\gan}=0\big\}.$$
Note that these results at least go back to 
the book of Grisvard \cite[Theorem 3.1.1.2]{grisvardbook},
see also the book of Leis \cite[p. 156-157]{leisbook}.

A first hint for a possible explanation of \eqref{PVconjPFM} is then the following observation:
Let $E_{1,\gat}$  be the minimiser from Remark \ref{fpmconstrem}. Then 
$$E_{1,\gat}\in D(\rot{\gat})\cap R(\rot{\gan})
\subset\H{}{\gat}(\rot{},\om)\cap\H{}{\gan}(\div{},\om),\qquad
\div{}E_{1,\gat}=0.$$
Hence, if $\om$ is a \emph{polyhedron} and if $E_{1,\gat}$ 
is regular\footnote{The additional regularity
of the minimiser $E_{1,\gat}$ is not realistic.} 
enough, i.e.,
$E_{1,\gat}\in\H{}{\gat,\gan}(\grad{},\om)$, then by \eqref{intpartgradrotdiv}
and \eqref{addregE}
\begin{align*}
\lambda_{1,\gat}
&=\frac{\norm{\rot{}E_{1,\gat}}_{\L{2}(\om)}}{\norm{E_{1,\gat}}_{\L{2}(\om)}}
=\frac{\norm{\grad{}E_{1,\gat}}_{\L{2}(\om)}}{\norm{E_{1,\gat}}_{\L{2}(\om)}}.
\end{align*}
Moreover, if $E_{1,\gat}$ admits the additional regularity 
$E_{1,\gat}\in\H{}{\gat}(\grad{},\om)$, then 
\begin{align*}
\lambda_{1,\gat}
&\geq\inf_{0\neq E\in\H{}{\gat}(\grad{},\om)}
\frac{\norm{\grad{}E}_{\L{2}(\om)}}{\norm{E}_{\L{2}(\om)}}
=\lambda_{0,\gat}.
\end{align*}


\subsection*{Acknowledgment}
The research of J. Valdman was supported by the Czech Science Foundation (GA R), through the grant 19-29646L. His visit in Essen in 2019 was financed by the Erasmus+ programme of the European Union.

\bibliographystyle{plain} 
\bibliography{bib-pv}

\section{Appendix: Some Proofs}
\label{appproofs}

\begin{proof}[Proof of \eqref{saops}.]
To show that, e.g., $\A^{*}\A$ is self-adjoint,
we first observe that $\A^{*}\A$ is symmetric.
Hence, so is $\A^{*}\A+1$. By Riesz' representation theorem,
for any $f\in\H{}{0}$ there exists a unique $x\in D(\A)$ such that
$$\forall\,\varphi\in D(\A)\qquad
\scp{\A x}{\A\varphi}_{\H{}{1}}+\scp{x}{\varphi}_{\H{}{0}}=\scp{f}{\varphi}_{\H{}{0}}.$$
Thus, $\A x\in D(\A^{*})$ and $\A^{*}\A x=f-x$, i.e.,
$x\in D(\A^{*}\A)$ and $(\A^{*}\A+1)x=f$.
In other words, $\A^{*}\A+1$ is onto.
Therefore, $\A^{*}\A+1$ is self-adjoint and so is $\A^{*}\A$.
Note that we did not need the additional assumption that $R(\A)$ is closed
or that $\A$ resp. $\A^{*}$ is onto.

We also present an alternative proof of the self-adjointness of $\A^{*}\A$
in the case that $R(\A)$ is closed.
For this, let $y\in\H{}{0}$ such that
there exists $z\in\H{}{0}$ with
\begin{align}
\label{AsAsa}
\forall\,x\in D(\A^{*}\A)\qquad
\scp{\A^{*}\A x}{y}_{\H{}{0}}=\scp{x}{z}_{\H{}{0}}.
\end{align}
Picking $x\in N(\A)$ shows that $z\,\bot_{\H{}{0}}N(\A)$, i.e., 
we have $z\in R(\A^{*})$ by \eqref{helm}.
For $\varphi\in D(\A^{*})$ we note $\A^{*}\varphi\in R(\A^{*})=R(\cA^{*})$.
Thus there is 
$$\psi_{\varphi}:=(\cA^{*})^{-1}\A^{*}\varphi\in D(\cA^{*})\subset R(\A)=R(\cA)
\qtext{with}\A^{*}\psi_{\varphi}=\A^{*}\varphi.$$
Moreover, there exists $x_{\varphi}:=\cA^{-1}\psi_{\varphi}\in D(\cA)$ with $\A x_{\varphi}=\psi_{\varphi}$
and thus $x_{\varphi}\in D(\cA^{*}\cA)$. By \eqref{AsAsa} we see
\begin{align*}
\scp{\A^{*}\varphi}{y}_{\H{}{0}}
&=\scp{\A^{*}\A x_{\varphi}}{y}_{\H{}{0}}
=\scp{x_{\varphi}}{z}_{\H{}{0}}
=\bscp{x_{\varphi}}{\A^{*}(\cA^{*})^{-1}z}_{\H{}{0}}
=\bscp{\A x_{\varphi}}{(\cA^{*})^{-1}z}_{\H{}{1}}\\
&=\bscp{\varphi}{(\cA^{*})^{-1}z}_{\H{}{1}}
+\underbrace{\bscp{\psi_{\varphi}-\varphi}{(\cA^{*})^{-1}z}_{\H{}{1}}}_{=0},
\end{align*}
as $\psi_{\varphi}-\varphi\in N(\A^{*})\,\bot_{\H{}{1}}R(\A)\supset D(\cA^{*})\ni(\cA^{*})^{-1}z$.
Therefore, $y\in D(\A)$ and $\A y=(\cA^{*})^{-1}z\in D(\A^{*})$,
showing $y\in D(\A^{*}\A)$ and $\A^{*}\A y=z$. This proves $(\A^{*}\A)^{*}=\A^{*}\A$.
\end{proof}

\begin{proof}[Proof of Lemma \ref{lemconstev}.]
We show a few selected assertions of Lemma \ref{lemconstev}.

$\bullet$ 
For an eigenvalue $\lambda>0$ and an eigenvector $(x,y)$
of $\begin{bmatrix}0&\A^{*}\\\A&0\end{bmatrix}$ it holds
$\A^{*} y=\lambda x$ and $\A x=\lambda y$.
Note that $x=0$ implies $y=0$.
Thus $0\neq x\in D(\A^{*}\A)$ and $\A^{*}\A x=\lambda\A^{*} y=\lambda^2x$, i.e.,
$x$ is an eigenvector and $\lambda^2$ is an eigenvalue of $\A^{*}\A$.

$\bullet$ 
If $\lambda^2>0$ is an eigenvalue and $x$ is an eigenvector of $\A^{*}\A$,
then $y_{\pm}:=\pm\lambda^{-1}\A x\in D(\A^{*})$ 
and $\A^{*} y_{\pm}=\pm\lambda^{-1}\A^{*}\A x=\pm\lambda x$, i.e.,
$(x,y_{\pm})$ is an eigenvector and $\pm\lambda$ is an eigenvalue of $\begin{bmatrix}0&\A^{*}\\\A&0\end{bmatrix}$.
Note that $y_{\pm}\neq0$ as $y_{\pm}=0$ implies $x=0$.

$\bullet$ 
If $\lambda^2>0$ is an eigenvalue and $x$ is an eigenvector of $\A^{*}\A$,
then $y:=\A x\in D(\A^{*})$ and we have $\A^{*}y=\A^{*}\A x=\lambda^2x\in D(\A)$.
Hence $y\in D(\A\A^{*})$ and $\A\A^{*}y=\lambda^2\A x=\lambda^2y$, i.e.,
$\lambda^2$ is an eigenvalue and $y$ is an eigenvector of $\A\A^{*}$.
Note that $y\neq0$ as $y=0$ implies $x=0$.

$\bullet$ 
To show that indeed, e.g., $\lambda_{\A}^2$ is the smallest positive eigenvalue of $\A^{*}\A$,
let us consider a sequence $(\tilde x_{n})$ in $D(\cA)\setminus\{0\}$ with 
$$\frac{\norm{\A\tilde x_{n}}_{\H{}{1}}}{\norm{\tilde x_{n}}_{\H{}{0}}}
\to\inf_{0\neq x\in D(\cA)}\frac{\norm{\A x}_{\H{}{1}}}{\norm{x}_{\H{}{0}}}=\lambda_{\A}>0.$$
Then $(x_{n}):=(\tilde x_{n}/\norm{\tilde x_{n}}_{\H{}{0}})\subset D(\cA)$ with $\norm{x_{n}}_{\H{}{0}}=1$ and
$$\lambda_{\A}\leq\norm{\A x_{n}}_{\H{}{1}}\to\lambda_{\A}.$$
Hence $(x_{n})$ is bounded in $D(\cA)$, yielding a subsequence - again denoted by $(x_{n})$ - 
as well as $x_{\A}\in\H{}{0}$ and $y_{\A}\in\H{}{1}$
with $x_{n}\wto x_{\A}$ in $\H{}{0}$, $\A x_{n}\wto y_{\A}$ in $\H{}{1}$,
and $x_{n}\to x_{\A}$ in $\H{}{0}$.
Then $x_{\A}\in D(\A)$ and $\A x_{\A}=y_{\A}$ as for all $\psi\in D(\A^{*})$
$$\scp{y_{\A}}{\psi}_{\H{}{1}}
\ot\scp{\A x_{n}}{\psi}_{\H{}{1}}
=\scp{x_{n}}{\A^{*}\psi}_{\H{}{0}}
\to\scp{x_{\A}}{\A^{*}\psi}_{\H{}{0}}.$$
Note that $x_{\A}\in R(\A^{*})$
as $x_{n}\in R(\A^{*})=N(\A)^{\bot_{\H{}{0}}}$,
especially, $x_{\A}\in D(\cA)$. Moreover, $\norm{x_{\A}}_{\H{}{0}}=1$.

By elementary calculations\footnote{For all 
$\varphi,\phi\in D(\cA)$ and for all $\eps\in\reals$ it holds
$\lambda_{\A}\norm{\varphi+\eps\phi}_{\H{}{0}}\leq\norm{\A(\varphi+\eps\phi)}_{\H{}{1}}$,
i.e.,
$$0\leq\big(\underbrace{\norm{\A\varphi}_{\H{}{1}}^2-\lambda_{\A}^2\norm{\varphi}_{\H{}{0}}^2}_{=:\alpha\geq0}\big)
+2\eps\Re\big(\underbrace{\scp{\A\varphi}{\A\phi}_{\H{}{1}}-\lambda_{\A}^2\scp{\varphi}{\phi}_{\H{}{0}}}_{=:\delta}\big)
+\eps^2\big(\underbrace{\norm{\A\phi}_{\H{}{1}}^2-\lambda_{\A}^2\norm{\phi}_{\H{}{0}}^2}_{=:\gamma\geq0}\big).$$
Let $\beta:=\Re\delta$ and
$0\leq f(\eps):=\alpha+2\beta\eps+\gamma\eps^2$. If $\gamma=0$ then $\beta=0$.
For $\gamma>0$ the minimum of $f$ is attained at $\eps=-\beta/\gamma$
and thus $0\leq f(-\beta/\gamma)=\alpha-2\beta^2/\gamma+\beta^2/\gamma=\alpha-\beta^2/\gamma$
yielding $\beta^2\leq\alpha\gamma$. Replacing $\eps$ by $-i\eps$ 
shows the same inequality $\beta^2\leq\alpha\gamma$ for $\beta:=\Im\delta$.
Hence $|\delta|^2\leq2\alpha\gamma$.}
we obtain for all $\varphi,\phi\in D(\cA)$
$$\big|\scp{\A\varphi}{\A\phi}_{\H{}{1}}-\lambda_{\A}^2\scp{\varphi}{\phi}_{\H{}{0}}\big|^2
\leq2\big(\norm{\A\varphi}_{\H{}{1}}^2-\lambda_{\A}^2\norm{\varphi}_{\H{}{0}}^2\big)
\big(\norm{\A\phi}_{\H{}{1}}^2-\lambda_{\A}^2\norm{\phi}_{\H{}{0}}^2\big).$$
In particular, for $\varphi:=x_{n}$ we get for all $\phi\in D(\cA)$
\begin{align}
\label{crucialest}
\big|\scp{\A x_{n}}{\A\phi}_{\H{}{1}}-\lambda_{\A}^2\scp{x_{n}}{\phi}_{\H{}{0}}\big|^2
\leq2\big(\norm{\A x_{n}}_{\H{}{1}}^2-\lambda_{\A}^2\big)
\big(\norm{\A\phi}_{\H{}{1}}^2-\lambda_{\A}^2\norm{\phi}_{\H{}{0}}^2\big)\to0
\end{align}
and thus
\begin{align*}
&\qquad\big|\scp{\A x_{\A}}{\A\phi}_{\H{}{1}}-\lambda_{\A}^2\scp{x_{\A}}{\phi}_{\H{}{0}}\big|\\
&\leq\big|\scp{\A(x_{\A}-x_{n})}{\A\phi}_{\H{}{1}}\big|
+\lambda_{\A}^2\big|\scp{x_{\A}-x_{n}}{\phi}_{\H{}{0}}\big|
+\big|\scp{\A x_{n}}{\A\phi}_{\H{}{1}}-\lambda_{\A}^2\scp{x_{n}}{\phi}_{\H{}{0}}\big|\to0.
\end{align*}
Hence, for all $\phi\in D(\cA)$
\begin{align}
\label{eigeneqcA}
\scp{\A x_{\A}}{\A\phi}_{\H{}{1}}
=\lambda_{\A}^2\scp{x_{\A}}{\phi}_{\H{}{0}}.
\end{align}
For $\phi\in D(\A)=N(\A)\oplus_{\H{}{0}}D(\cA)$, see the Helmholtz type decomposition \eqref{helmdecoAcA},
we decompose
$$\phi=\phi_{N}+\phi_{\cA}\in N(\A)\oplus_{\H{}{0}}D(\cA)$$
and compute by using \eqref{eigeneqcA}, $\A\phi=\A\phi_{\cA}$, and $x_{\A}\in R(\A^{*})\bot_{\H{}{0}}N(\A)$
$$\scp{\A x_{\A}}{\A\phi}_{\H{}{1}}
=\scp{\A x_{\A}}{\A\phi_{\cA}}_{\H{}{1}}
=\lambda_{\A}^2\scp{x_{\A}}{\phi_{\cA}}_{\H{}{0}}
=\lambda_{\A}^2\scp{x_{\A}}{\phi}_{\H{}{0}}.$$
Therefore, \eqref{eigeneqcA} holds for all $\phi\in D(\A)$, i.e.,
\begin{align}
\label{eigeneqA}
\forall\,\phi\in D(\A)\qquad
\scp{\A x_{\A}}{\A\phi}_{\H{}{1}}
=\lambda_{\A}^2\scp{x_{\A}}{\phi}_{\H{}{0}}.
\end{align}
This implies $\A x_{\A}\in D(\A^{*})$, i.e.,
$x_{\A}\in D(\A^{*}\A)$ and $\A^{*}\A x_{\A}=\lambda_{\A}^2x_{\A}$.
We even have $x_{\A}\in D(\cA^{*}\cA)$.
Thus, $\lambda_{\A}^2$ is an eigenvalue and $x_{\A}$ is an eigenvector of $\A^{*}\A$.
Note that \eqref{eigeneqcA} or \eqref{eigeneqA} implies (for $\phi=x_{\A}$)
$\norm{\A x_{\A}}_{\H{}{1}}^2=\lambda_{\A}^2\norm{x_{\A}}_{\H{}{0}}^2$, i.e., 
$\norm{\A x_{\A}}_{\H{}{1}}=\lambda_{\A}$.

Finally, we show that $(x_{n})$ even converges strongly in $D(\cA)$,
i.e., $(\A x_{n})$ converges strongly in $\H{}{1}$ respectively in $R(\A)$.
For this, we get for all $\phi\in D(\cA)$ by \eqref{crucialest} and \eqref{eigeneqcA}
\begin{align*}
&\qquad\big|\scp{\A(x_{n}-x_{\A})}{\A\phi}_{\H{}{1}}-\lambda_{\A}^2\scp{x_{n}-x_{\A}}{\phi}_{\H{}{0}}\big|^2\\
&=\big|\scp{\A x_{n}}{\A\phi}_{\H{}{1}}-\lambda_{\A}^2\scp{x_{n}}{\phi}_{\H{}{0}}\big|^2
\leq2\big(\norm{\A x_{n}}_{\H{}{1}}^2-\lambda_{\A}^2\big)
\big(\norm{\A\phi}_{\H{}{1}}^2-\lambda_{\A}^2\norm{\phi}_{\H{}{0}}^2\big).
\end{align*}
In particular, for $\phi=x_{n}-x_{\A}$ we see
\begin{align*}
\big|\norm{\A(x_{n}-x_{\A})}_{\H{}{1}}^2-\lambda_{\A}^2\norm{x_{n}-x_{\A}}_{\H{}{0}}^2\big|^2
\leq c\big(\norm{\A x_{n}}_{\H{}{1}}^2-\lambda_{\A}^2\big)\to0,
\end{align*}
and hence
\begin{align*}
\norm{\A(x_{n}-x_{\A})}_{\H{}{1}}^2
\leq\lambda_{\A}^2\norm{x_{n}-x_{\A}}_{\H{}{0}}^2
+\big|\norm{\A(x_{n}-x_{\A})}_{\H{}{1}}^2-\lambda_{\A}^2\norm{x_{n}-x_{\A}}_{\H{}{0}}^2\big|
\to0.
\end{align*}

$\bullet$ 
For $0\neq\lambda\in\sigma(\A^{*}\A)$ we have $\A^{*}\A x=\lambda x$
for some $0\neq x\in D(\A^{*}\A)=D(\cA^{*}\A)$.
Hence $x\in R(\A^{*})$ and thus $x\in D(\cA)$, showing $0\neq x\in D(\cA^{*}\cA)$.
So $\lambda\in\sigma(\cA^{*}\cA)$.

$\bullet$ 
For $\displaystyle0\neq\lambda\in\sigma\Big(\begin{bmatrix}0&\A^{*}\\\A&0\end{bmatrix}\Big)\setminus\{0\}$
we have $\A^{*}y=\lambda x$ and $\A x=\lambda y$ 
for some $0\neq(x,y)\in D(\A)\times D(\A^{*})$.
Hence $(x,y)\in R(\A^{*})\times R(\A)$ and thus $(x,y)\in D(\cA)\times D(\cA^{*})$, 
showing $0\neq(x,y)\in D(\cA)\times D(\cA^{*})$.
So $\displaystyle\lambda\in\sigma\Big(\begin{bmatrix}0&\cA^{*}\\\cA&0\end{bmatrix}\Big)$.

$\bullet$ 
It holds
$$\bnorm{\cA^{-1}}_{R(\A),R(\A^{*})}
=\sup_{0\neq y\in D(\cA^{-1})}\frac{\norm{\cA^{-1}y}_{\H{}{0}}}{\norm{y}_{\H{}{1}}}
=\sup_{0\neq x\in D(\cA)}\frac{\norm{x}_{\H{}{0}}}{\norm{\A x}_{\H{}{1}}}
=\Big(\inf_{0\neq x\in D(\cA)}\frac{\norm{\A x}_{\H{}{1}}}{\norm{x}_{\H{}{0}}}\Big)^{-1}
=c_{\A}.$$

$\bullet$ 
Let $x_{\A^{*}\A}$ with $\norm{x_{\A^{*}\A}}_{\H{}{0}}=1$ 
be an eigenvector of $\A^{*}\A$ to the eigenvalue $\lambda_{\A}^2$. Then $x_{\A^{*}\A}\in R(\A^{*})$
and $\lambda_{\A}^2(\cA^{*}\cA)^{-1}x_{\A^{*}\A}=x_{\A^{*}\A}$. Thus
\begin{align*}
&\qquad\bnorm{(\cA^{*}\cA)^{-1}}_{R(\A^{*}),R(\A^{*})}
\leq\bnorm{\cA^{-1}}_{R(\A),R(\A^{*})}\bnorm{(\cA^{*})^{-1}}_{R(\A^{*}),R(\A)}=c_{\A}^2\\
&=\sup_{0\neq x\in D((\cA^{*}\cA)^{-1})}\frac{\bnorm{(\cA^{*}\cA)^{-1}x}_{\H{}{0}}}{\norm{x}_{\H{}{0}}}
\geq\bnorm{(\cA^{*}\cA)^{-1}x_{\A^{*}\A}}_{\H{}{0}}
=\frac{1}{\lambda_{\A}^2}
=c_{\A}^2.
\end{align*}

$\bullet$ 
For $x\in N(\A^{*}\A)$ we have $0=\scp{\A^{*}\A x}{x}_{\H{}{0}}=\norm{\A x}_{\H{}{1}}^2$, 
i.e., $x\in N(\A)$. Analogously, we see $N(\A\A^{*})=N(\A^{*})$.
For $x\in N(\A\A^{*}\A)$ we have $\A x\in N(\A\A^{*})=N(\A^{*})$, i.e.,
$x\in N(\A^{*}\A)=N(\A)$. The latter arguments can be repeated for any higher power.

$\bullet$ 
For $y\in R(\A)$ we see $x:=\cA^{-1}y\in D(\cA)\subset R(\A^{*})$
and $z:=(\cA^{*})^{-1}x\in D(\cA^{*})\subset R(\A)$.
Thus $z\in D(\cA\cA^{*})$ and $\A^{*}z=x$ and $\A x=y$
as well as $\A\A^{*}z=\A x=y\in R(\cA\cA^{*})=R(\A\A^{*})$.
The latter arguments can be repeated for any higher power,
completing the proof.
\end{proof}

\begin{proof}[Proof of Lemma \ref{lemconstevcomplex}.]
{\bf(i)}
By \eqref{saops} we just have to show that $\A_{0}\A_{0}^{*}+\A_{1}^{*}\A_{1}$ is self-adjoint.
For this, let $y\in\H{}{1}$ such that there exists $z\in\H{}{1}$ with
\begin{align}
\label{AsAzosa1}
\forall\,x\in D(\A_{0}\A_{0}^{*}+\A_{1}^{*}\A_{1})\qquad
\bscp{(\A_{0}\A_{0}^{*}+\A_{1}^{*}\A_{1})x}{y}_{\H{}{1}}=\scp{x}{z}_{\H{}{1}}.
\end{align}
Picking $x\in N_{0,1}$ shows that $z\,\bot_{\H{}{1}}N_{0,1}$
and hence, according to Theorem \ref{compembtheo}, $y$ and $z$ can be orthogonally decomposed into
\begin{align*}
y=y_{R(\A_{0})}+y_{R(\A_{1}^{*})}+y_{N_{0,1}}
&\in R(\A_{0})\oplus_{\H{}{1}}R(\A_{1}^{*})\oplus_{\H{}{1}}N_{0,1},\\
z=z_{R(\A_{0})}+z_{R(\A_{1}^{*})}
&\in R(\A_{0})\oplus_{\H{}{1}}R(\A_{1}^{*}).
\end{align*}
\eqref{AsAzosa1} implies for all $x\in D(\A_{0}\A_{0}^{*}+\A_{1}^{*}\A_{1})$
\begin{align}
\label{AsAzosa2}
\begin{split}
\scp{\A_{0}\A_{0}^{*} x}{y_{R(\A_{0})}}_{\H{}{1}}
+\scp{\A_{1}^{*}\A_{1} x}{y_{R(\A_{1}^{*})}}_{\H{}{1}}
&=\bscp{(\A_{0}\A_{0}^{*}+\A_{1}^{*}\A_{1})x}{y}_{\H{}{1}}\\
&=\scp{x}{z}_{\H{}{1}}
=\scp{x}{z_{R(\A_{0})}}_{\H{}{1}}
+\scp{x}{z_{R(\A_{1}^{*})}}_{\H{}{1}}.
\end{split}
\end{align}
For $x\in D(\cA_{1}^{*}\cA_{1})\subset R(\A_{1}^{*})\subset N(\A_{0}^{*})$ 
we see by \eqref{AsAzosa2} that
$\scp{\A_{1}^{*}\A_{1} x}{y_{R(\A_{1}^{*})}}_{\H{}{1}}=\scp{x}{z_{R(\A_{1}^{*})}}_{\H{}{1}}$ holds, 
yielding by \eqref{saops}, i.e., $\cA_{1}^{*}\cA_{1}$ is self-adjoint, that 
$y_{R(\A_{1}^{*})}\in D(\cA_{1}^{*}\cA_{1})\subset N(\A_{0}^{*})$ 
with $\A_{1}^{*}\A_{1}y_{R(\A_{1}^{*})}=z_{R(\A_{1}^{*})}$.
Analogously we see by using $x\in D(\cA_{0}\cA_{0}^{*})$ 
that $y_{R(\A_{0})}\in D(\cA_{0}\cA_{0}^{*})\subset N(\A_{1})$ 
with $\A_{0}\A_{0}^{*}y_{R(\A_{0})}=z_{R(\A_{0})}$.
Thus $y\in D(\A_{0}\A_{0}^{*}+\A_{1}^{*}\A_{1})$ with
$(\A_{0}\A_{0}^{*}+\A_{1}^{*}\A_{1})y=z_{R(\A_{0})}+z_{R(\A_{1}^{*})}=z$,
i.e., we have shown $(\A_{0}\A_{0}^{*}+\A_{1}^{*}\A_{1})^{*}=\A_{0}\A_{0}^{*}+\A_{1}^{*}\A_{1}$.

{\bf(v)}
Let $0\neq\lambda\in\sigma(\A_{0}\A_{0}^{*}+\A_{1}^{*}\A_{1})$
and let $0\neq x\in D(\A_{0}\A_{0}^{*}+\A_{1}^{*}\A_{1})$ be an eigenvector to the eigenvalue $\lambda$.
Then $y:=\A_{1}^{*}\A_{1}x=\lambda x-\A_{0}\A_{0}^{*}x\in D(\A_{1}^{*}\A_{1})$ and 
$$\A_{1}^{*}\A_{1}y=\lambda\A_{1}^{*}\A_{1}x=\lambda y.$$
Thus, as long as $y\neq0$, $\lambda$ is an eigenvalue of $\A_{1}^{*}\A_{1}$ with eigenvector $y$.
On the other hand, if $y=0$, then 
$z:=\A_{0}\A_{0}^{*}x=\lambda x\in D(\A_{0}\A_{0}^{*})\setminus\{0\}$ 
and $\A_{0}\A_{0}^{*}z=\lambda\A_{0}\A_{0}^{*}x=\lambda z$.
Hence $\lambda$ is an eigenvalue of $\A_{0}\A_{0}^{*}$ with eigenvector $z$.
This shows 
$$\sigma(\A_{0}\A_{0}^{*}+\A_{1}^{*}\A_{1})\setminus\{0\}
\subset\big(\sigma(\A_{0}\A_{0}^{*})\setminus\{0\}\big)
\cup\big(\sigma(\A_{1}^{*}\A_{1})\setminus\{0\}\big).$$
For the other inclusion, let, e.g., $0\neq\lambda\in\sigma(\A_{1}^{*}\A_{1})$
and let $0\neq x\in D(\A_{1}^{*}\A_{1})$ be an eigenvector to the eigenvalue $\lambda$.
Then $x\in R(\A_{1}^{*})\subset N(\A_{0}^{*})$ and thus 
$(\A_{0}\A_{0}^{*}+\A_{1}^{*}\A_{1})x=\A_{1}^{*}\A_{1}x=\lambda x$, i.e.,
$\lambda$ is an eigenvalue of $\A_{0}\A_{0}^{*}+\A_{1}^{*}\A_{1}$ with eigenvector $x$. Thus 
$$\sigma(\A_{1}^{*}\A_{1})\setminus\{0\}
\subset\sigma(\A_{0}\A_{0}^{*}+\A_{1}^{*}\A_{1})\setminus\{0\},$$
and analogously we show 
$\sigma(\A_{0}\A_{0}^{*})\setminus\{0\}
\subset\sigma(\A_{0}\A_{0}^{*}+\A_{1}^{*}\A_{1})\setminus\{0\}$.

{\bf(iii)}
Let $x\in N(\A_{0}\A_{0}^{*}+\A_{1}^{*}\A_{1})$. Then
$$0=\bscp{(\A_{0}\A_{0}^{*}+\A_{1}^{*}\A_{1})x}{x}_{\H{}{1}}
=\norm{\A_{0}^{*}x}_{\H{}{0}}^2+\norm{\A_{1}x}_{\H{}{2}}^2,$$
showing $x\in N_{0,1}$. 
As $D(\A_{1})\cap D(\A_{0}^{*})\emb\H{}{1}$ is compact, 
so is $D(\A_{0}\A_{0}^{*}+\A_{1}^{*}\A_{1})\emb\H{}{1}$,
showing that the range $R(\A_{0}\A_{0}^{*}+\A_{1}^{*}\A_{1})$
is closed by Remark \ref{fatbr} (ii). Thus 
$$R(\A_{0}\A_{0}^{*}+\A_{1}^{*}\A_{1})
=N(\A_{0}\A_{0}^{*}+\A_{1}^{*}\A_{1})^{\bot_{\H{}{1}}}
=N_{0,1}^{\bot_{\H{}{1}}},$$
finishing the proof.
\end{proof}

\section{Appendix: Analytical Calculations}
\label{appanaex}

We compute the exact eigenvalues and eigenfunctions 
of Section \ref{anaex} in detail. 

\subsection{1D}
\label{appanaex1D}

Recall the situation and notations from Section \ref{ocplxcst1DdR} and Section \ref{anaex1D}.
In particular, 
$$\frac{1}{c_{0,\gat}}
=\lambda_{0,\gat}
=\lambda_{0,\gan}
=\frac{1}{c_{0,\gan}}.$$
Let $u=u_{0,\gat}$ be the first eigenfunction 
for the eigenvalue $\lambda^2$
with $\lambda=\lambda_{0,\gat}>0$
of $-\Delta_{\gat}$. Hence, we have
$E_{0,\gan}=\grad{}u_{0,\gat}$ and
$$u\in D(\Delta_{\gat})\cap\L{2}_{\gan}(\om)
\subset\H{1}{\gat}(\om)\cap\L{2}_{\gan}(\om),\qquad
\grad{}u=u'\in D(\div{\gan})
=\H{1}{\gan}(\om),$$
as well as
$$(-\Delta-\lambda^2)u
=-u''-\lambda^2u=0.$$
Then 
$$u=\alpha\sin(\lambda x)+\beta\cos(\lambda x),\qquad
u'(x)=\alpha\lambda\cos(\lambda x)-\beta\lambda\sin(\lambda x).$$ 
For the different boundary conditions we get:

$\bullet$ $\gat=\emptyset$ and $\gan=\ga$, i.e.,
$u'(0)=u'(1)=0$: 
$\alpha=0$, $\lambda=n\pi$, $n\in\mathbb{N}_{0}$, i.e.,
$$\lambda_{0,\emptyset}=\pi,\qquad
u_{0,\emptyset}(x)=\beta\cos(\pi x).$$

Note that in this case the first eigenvalue is $\lambda=0$.

$\bullet$ $\gat=\{0\}$ and $\gan=\{1\}$, i.e.,
$u(0)=u'(1)=0$: 
$\beta=0$, $\lambda=(n-1/2)\pi$, $n\in\mathbb{N}$, i.e.,
$$\lambda_{0,\{0\}}=\frac{\pi}{2},\qquad
u_{0,\{0\}}(x)=\alpha\sin(\frac{\pi}{2}x).$$

$\bullet$ $\gat=\{1\}$ and $\gan=\{0\}$, i.e.,
$u'(0)=u(1)=0$: 
$\alpha=0$, $\lambda=(n-1/2)\pi$, $n\in\mathbb{N}$, i.e.,
$$\lambda_{0,\{1\}}=\frac{\pi}{2},\qquad
u_{0,\{1\}}(x)=\beta\cos(\frac{\pi}{2}x).$$

$\bullet$ $\gat=\ga$ and $\gan=\emptyset$, i.e.,
$u(0)=u(1)=0$: 
$\beta=0$, $\lambda=n\pi$, $n\in\mathbb{N}$, i.e.,
$$\lambda_{0,\ga}=\pi,\qquad
u_{0,\ga}(x)=\alpha\sin(\pi x).$$

Note that from $\lambda_{0,\gat}=\lambda_{0,\gan}$ we already know
$\lambda_{0,\ga}=\lambda_{0,\emptyset}$ and $\lambda_{0,\{0\}}=\lambda_{0,\{1\}}$, i.e.,
$$\lambda_{0,\ga}=\lambda_{0,\emptyset}=\pi,\qquad
\lambda_{0,\{0\}}=\lambda_{0,\{1\}}=\frac{\pi}{2}.$$



\subsection{2D}
\label{appanaex2D}

Recall the situation and notations from Section \ref{ocplxcst2DdR} and Section \ref{anaex2D}.
In particular, 
$$\frac{1}{c_{0,\gat}}=\lambda_{0,\gat}=\lambda_{1,\gan}=\frac{1}{c_{1,\gan}}.$$
Let $u=u_{0,\gat}$ be the first eigenfunction
for the eigenvalue $\lambda^2$
with $\lambda=\lambda_{0,\gat}>0$ of $-\Delta_{\gat}$. 
Hence, we have
$E_{0,\gan}=\grad{}u_{0,\gat}$ and\footnote{Note that 
\begin{align*}
E_{1,\gat}&\in D(\square_{\gat})\cap R(\rotv{\gan})
\subset\H{}{\gat}(\rot{},\om)\cap R(\rotv{\gan}),\\
H_{1,\gan}=\rot{}E_{1,\gat}&\in D(\rotv{\gan})\cap R(\rot{\gat})
=\H{}{\gan}(\rotv{},\om)\cap\L{2}_{\gat}(\om)
=\H{1}{\gan}(\om)\cap\L{2}_{\gat}(\om).
\end{align*}}
$$u\in D(\Delta_{\gat})\cap\L{2}_{\gan}(\om)
\subset\H{1}{\gat}(\om)\cap\L{2}_{\gan}(\om),\qquad
\grad{}u\in D(\div{\gan})=\H{}{\gan}(\div{},\om),$$
as well as 
$$(-\Delta-\lambda^2)u=0.$$
Separation of variables shows with $u(x)=u_{1}(x_{1})u_{2}(x_{2})$
and $\grad{}u(x)=\begin{bmatrix}u_{1}'(x_{1})u_{2}(x_{2})\\u_{1}(x_{1})u_{2}'(x_{2})\end{bmatrix}$
$$0=(-\Delta-\lambda^2)u(x)
=-u_{1}''(x_{1})u_{2}(x_{2})-u_{1}(x_{1})u_{2}''(x_{2})-\lambda^2u_{1}(x_{1})u_{2}(x_{2}).$$
For fixed $x_{2}$ with $u_{2}(x_{2})\neq0$ we get
$$-u_{1}''(x_{1})-\mu_{1}^2u_{1}(x_{1})=0,\qquad
\mu_{1}^2=\frac{u_{2}''(x_{2})}{u_{2}(x_{2})}+\lambda^2,$$
i.e.,
$$-u_{1}''(x_{1})-\mu_{1}^2u_{1}(x_{1})=0,\quad
-u_{2}''(x_{2})-\mu_{2}^2u_{2}(x_{2})=0,\quad
\lambda^2=\mu_{1}^2+\mu_{2}^2.$$
The Dirichlet boundary conditions, i.e.,
$$u=0\qtext{on}\gat,$$
reduce to Dirichlet boundary conditions for $u_{1}$ and $u_{2}$, respectively,
and the Neumann boundary conditions, i.e.,
$$n\cdot\grad{}u=0\qtext{on}\gan,$$
reduce to Dirichlet boundary conditions for $u_{1}'$ and $u_{2}'$, respectively.
More precisely, we have:

$\bullet$
$\ga_{l}$,
$n=-e^1$, $x_{1}=0$:
\begin{align*}
0=u|_{\ga_{l}}&=u_{1}u_{2}|_{\ga_{l}}
&
&\impl
&
u_{1}(0)&=0,\\
0=n\cdot\grad{}u|_{\ga_{l}}&=-u_{1}'u_{2}|_{\ga_{l}}
&
&\impl
&
u_{1}'(0)&=0.
\end{align*}

$\bullet$
$\ga_{r}$,
$n=e^1$, $x_{1}=1$:
\begin{align*}
0=u|_{\ga_{r}}&=u_{1}u_{2}|_{\ga_{r}}
&
&\impl
&
u_{1}(1)&=0,\\
0=n\cdot\grad{}u|_{\ga_{r}}&=u_{1}'u_{2}|_{\ga_{r}}
&
&\impl
&
u_{1}'(1)&=0.
\end{align*}

$\bullet$
$\ga_{b}$,
$n=-e^2$, $x_{2}=0$:
\begin{align*}
0=u|_{\ga_{b}}&=u_{1}u_{2}|_{\ga_{b}}
&
&\impl
&
u_{2}(0)&=0,\\
0=n\cdot\grad{}u|_{\ga_{b}}&=-u_{1}u_{2}'|_{\ga_{b}}
&
&\impl
&
u_{2}'(0)&=0.
\end{align*}

$\bullet$
$\ga_{t}$,
$n=e^2$, $x_{2}=1$:
\begin{align*}
0=u|_{\ga_{t}}&=u_{1}u_{2}|_{\ga_{t}}
&
&\impl
&
u_{2}(1)&=0,\\
0=n\cdot\grad{}u|_{\ga_{t}}&=u_{1}u_{2}'|_{\ga_{t}}
&
&\impl
&
u_{2}'(1)&=0.
\end{align*}

The 1D case shows for the different boundary conditions the following:

$\bullet$ 
$\gat=\emptyset$ and $\gan=\ga$, i.e.,
$u_{1}'(0)=u_{1}'(1)=u_{2}'(0)=u_{2}'(1)=0$:
$\mu_{1}=n\pi$, $\mu_{2}=m\pi$, i.e., 
$\lambda=\sqrt{n^2+m^2}\pi$, $n,m\in\mathbb{N}_{0}$, and 
$$\lambda_{0,\emptyset}=\pi,\qquad
u_{0,\emptyset}(x)=\alpha\cos(\pi x_{1})+\beta\cos(\pi x_{2}).$$

Note that in this case the first eigenvalue is $\lambda=0$.

$\bullet$ 
$\gat=\ga_{b}$ and $\gan=\ga_{t,l,r}$, i.e.,
$u_{1}'(0)=u_{1}'(1)=u_{2}(0)=u_{2}'(1)=0$: 
$\mu_{1}=n\pi$, $\mu_{2}=(m-1/2)\pi$, i.e., 
$\lambda=\sqrt{n^2+(m-1/2)^2}\pi$, $n\in\mathbb{N}_{0}$, $m\in\mathbb{N}$, and 
$$\lambda_{0,\ga_{b}}=\frac{1}{2}\pi,\qquad
u_{0,\ga_{b}}(x)=\alpha\sin(\frac{\pi}{2}x_{2}).$$

$\bullet$ 
$\gat=\ga_{b,t}$ and $\gan=\ga_{l,r}$, i.e.,
$u_{1}'(0)=u_{1}'(1)=u_{2}(0)=u_{2}(1)=0$: 
$\mu_{1}=n\pi$, $\mu_{2}=m\pi$, i.e., 
$\lambda=\sqrt{n^2+m^2}\pi$, $n\in\mathbb{N}_{0}$, $m\in\mathbb{N}$, and 
$$\lambda_{0,\ga_{b,t}}=\pi,\qquad
u_{0,\ga_{b,t}}(x)=\alpha\sin(\pi x_{2}).$$

$\bullet$ 
$\gat=\ga_{b,l}$ and $\gan=\ga_{t,r}$, i.e.,
$u_{1}(0)=u_{1}'(1)=u_{2}(0)=u_{2}'(1)=0$:
$\mu_{1}=(n-1/2)\pi$, $\mu_{2}=(m-1/2)\pi$, i.e.,
$\lambda=\sqrt{(n-1/2)^2+(m-1/2)^2}\pi$, $n,m\in\mathbb{N}$, and 
$$\lambda_{0,\ga_{b,l}}=\frac{\sqrt{2}}{2}\pi,\qquad
u_{0,\ga_{b,l}}(x)=\alpha\sin(\frac{\pi}{2}x_{1})\sin(\frac{\pi}{2}x_{2}).$$

$\bullet$ 
$\gat=\ga_{b,l,r}$ and $\gan=\ga_{t}$, i.e.,
$u_{1}(0)=u_{1}(1)=u_{2}(0)=u_{2}'(1)=0$:
$\mu_{1}=n\pi$, $\mu_{2}=(m-1/2)\pi$, i.e., 
$\lambda=\sqrt{n^2+(m-1/2)^2}\pi$, $n,m\in\mathbb{N}$, and 
$$\lambda_{0,\ga_{b,l,r}}=\frac{\sqrt{5}}{2}\pi,\qquad
u_{0,\ga_{b,l,r}}(x)=\alpha\sin(\pi x_{1})\sin(\frac{\pi}{2}x_{2}).$$

$\bullet$ 
$\gat=\ga$ and $\gan=\emptyset$, i.e.,
$u_{1}(0)=u_{1}(1)=u_{2}(0)=u_{2}(1)=0$:
$\mu_{1}=n\pi$, $\mu_{2}=m\pi$, i.e., 
$\lambda=\sqrt{n^2+m^2}\pi$, $n,m\in\mathbb{N}$, and 
$$\lambda_{0,\ga}=\sqrt{2}\pi,\qquad
u_{0,\ga}(x)=\alpha\sin(\pi x_{1})\sin(\pi x_{2}).$$

All other cases follow by symmetry, i.e.,
\begin{align*}
\lambda_{0,\emptyset}
&=\pi,
&
\lambda_{0,\ga_{b,l}}
=\lambda_{0,\ga_{b,r}}
=\lambda_{0,\ga_{t,l}}
=\lambda_{0,\ga_{t,r}}
&=\frac{\sqrt{2}}{2}\pi,\\
\lambda_{0,\ga_{b}}
=\lambda_{0,\ga_{t}}
=\lambda_{0,\ga_{l}}
=\lambda_{0,\ga_{r}}
&=\frac{1}{2}\pi,
&
\lambda_{0,\ga_{b,l,r}}
=\lambda_{0,\ga_{t,l,r}}
=\lambda_{0,\ga_{b,t,l}}
=\lambda_{0,\ga_{b,t,r}}
&=\frac{\sqrt{5}}{2}\pi,\\
\lambda_{0,\ga_{b,t}}
=\lambda_{0,\ga_{l,r}}
&=\pi,
&
\lambda_{0,\ga}
&=\sqrt{2}\pi.
\end{align*}



\subsection{3D}
\label{appanaex3D}

Recall the situation and notations from Section \ref{seclapmax3D}, Theorem \ref{stateoftheartconstest},
and Section \ref{anaex3D}.
In particular, 
$$\frac{1}{c_{0,\gat}}
=\lambda_{0,\gat}
=\lambda_{2,\gan}
=\frac{1}{c_{2,\gan}},\qquad
\frac{1}{c_{1,\gat}}
=\lambda_{1,\gat}
=\lambda_{1,\gan}
=\frac{1}{c_{1,\gan}}.$$
Let $u=u_{0,\gat}$ be the first eigenfunction
for the eigenvalue $\lambda^2$
with $\lambda=\lambda_{0,\gat}>0$ of $-\Delta_{\gat}$. 
Analogously, let $E=E_{1,\gat}$ be the first eigenfunction
for the eigenvalue $\widetilde\lambda^2$
with $\widetilde\lambda=\lambda_{1,\gat}>0$ of $\square_{\gat}$. Hence,
\begin{align*}
u&\in D(\Delta_{\gat})\cap\L{2}_{\gan}(\om)
\subset\H{1}{\gat}(\om)\cap\L{2}_{\gan}(\om),
&
E&\in D(\square_{\gat})\cap R(\rot{\gan})
\subset\H{}{\gat}(\rot{},\om)\cap R(\rot{\gan}),\\
\grad{}u&\in D(\div{\gan})=\H{}{\gan}(\div{},\om),
&
\rot{}E&\in D(\rot{\gan})\cap R(\rot{\gat})
=\H{}{\gan}(\rot{},\om)\cap R(\rot{\gat}),\\
\end{align*}
and we have by $-\Delta=\rot{}\rot{}-\grad{}\div{}=\square-\grad{}\div{}$
$$(-\Delta-\lambda^2)u=0,\qquad
(-\Delta-\widetilde\lambda^2)E=(\square-\widetilde\lambda^2)E=0,$$
as $\div{}E=0$.
Let us first discuss $u$.
Separation of variables shows with 
$$u(x)=\widehat u(\widehat x)u_{3}(x_{3})=u_{1}(x_{1})u_{2}(x_{2})u_{3}(x_{3}),\quad
\widehat u(\widehat x)=u_{1}(x_{1})u_{2}(x_{2}),\qquad
x=\begin{bmatrix}\widehat x\\x_{3}\end{bmatrix},\quad
\widehat x=\begin{bmatrix}x_{1}\\x_{2}\end{bmatrix}$$ 
and 
$$\grad{}u(x)
=\begin{bmatrix}u_{3}(x_{3})\grad{}\widehat u(\widehat x)\\
u_{3}'(x_{3})\widehat u(\widehat x)\end{bmatrix}
=\begin{bmatrix}u_{1}'(x_{1})u_{2}(x_{2})u_{3}(x_{3})\\
u_{1}(x_{1})u_{2}'(x_{2})u_{3}(x_{3})\\
u_{1}(x_{1})u_{2}(x_{2})u_{3}'(x_{3})\end{bmatrix},\quad
\grad{}\widehat u(\widehat x)=\begin{bmatrix}u_{1}'(x_{1})u_{2}(x_{2})\\
u_{1}(x_{1})u_{2}'(x_{2})\end{bmatrix}$$
that
$$0=(-\Delta-\lambda^2)u(x)
=-\Delta\widehat u(\widehat x)u_{3}(x_{3})
-\widehat u(\widehat x)u_{3}''(x_{3})
-\lambda^2\widehat u(\widehat x)u_{3}(x_{3}).$$
For fixed $x_{3}$ with $u_{3}(x_{3})\neq0$ we get
$$-\Delta\widehat u(\widehat x)-\widehat\mu^2\widehat u(\widehat x)=0,\qquad
\widehat\mu^2=\frac{u_{3}''(x_{3})}{u_{3}(x_{3})}+\lambda^2,$$
i.e.,
$$-\Delta\widehat u(\widehat x)-\widehat\mu^2\widehat u(\widehat x)=0,\quad
-u_{3}''(x_{3})-\mu_{3}^2u_{3}(x_{3})=0,\quad
\lambda^2=\widehat\mu^2+\mu_{3}^2.$$
From the 2D case we already know $\widehat\mu^2=\mu_{1}^2+\mu_{2}^2$
and the splitting of $\widehat u$, i.e.,
$$\lambda^2=\mu_{1}^2+\mu_{2}^2+\mu_{3}^2$$
and
$$-u_{1}''(x_{1})-\mu_{1}^2u_{1}(x_{1})=0,\quad
-u_{2}''(x_{2})-\mu_{2}^2u_{2}(x_{2})=0,\quad
-u_{3}''(x_{3})-\mu_{3}^2u_{3}(x_{3})=0.$$
The Dirichlet boundary conditions, i.e.,
$$u=0\qtext{on}\gat,$$
reduce to Dirichlet boundary conditions for $u_{1}$, $u_{2}$, and $u_{3}$, respectively,
and the Neumann boundary conditions, i.e.,
$$n\cdot\grad{}u=0\qtext{on}\gan,$$
reduce to Neumann boundary conditions for $\widehat u$ and $u_{3}$
and hence to Dirichlet boundary conditions for $u_{1}'$, $u_{2}'$, and $u_{3}'$, respectively.

$\bullet$
$\ga_{\bk}$,
$n=-e^1$, $x_{1}=0$:
\begin{align*}
0=u|_{\ga_{\bk}}&=u_{1}u_{2}u_{3}|_{\ga_{\bk}}
&
&\impl
&
u_{1}(0)&=0,\\
0=n\cdot\grad{}u|_{\ga_{\bk}}&=-u_{1}'u_{2}u_{3}|_{\ga_{\bk}}
&
&\impl
&
u_{1}'(0)&=0.
\end{align*}

$\bullet$
$\ga_{f}$,
$n=e^1$, $x_{1}=1$:
\begin{align*}
0=u|_{\ga_{f}}&=u_{1}u_{2}u_{3}|_{\ga_{f}}
&
&\impl
&
u_{1}(1)&=0,\\
0=n\cdot\grad{}u|_{\ga_{f}}&=u_{1}'u_{2}u_{3}|_{\ga_{f}}
&
&\impl
&
u_{1}'(1)&=0.
\end{align*}

$\bullet$
$\ga_{l}$,
$n=-e^2$, $x_{2}=0$:
\begin{align*}
0=u|_{\ga_{l}}&=u_{1}u_{2}u_{3}|_{\ga_{l}}
&
&\impl
&
u_{2}(0)&=0,\\
0=n\cdot\grad{}u|_{\ga_{l}}&=-u_{1}u_{2}'u_{3}|_{\ga_{l}}
&
&\impl
&
u_{2}'(0)&=0.
\end{align*}

$\bullet$
$\ga_{r}$,
$n=e^2$, $x_{2}=1$:
\begin{align*}
0=u|_{\ga_{r}}&=u_{1}u_{2}u_{3}|_{\ga_{r}}
&
&\impl
&
u_{2}(1)&=0,\\
0=n\cdot\grad{}u|_{\ga_{r}}&=u_{1}u_{2}'u_{3}|_{\ga_{r}}
&
&\impl
&
u_{2}'(1)&=0.
\end{align*}

$\bullet$
$\ga_{b}$,
$n=-e^3$, $x_{3}=0$:
\begin{align*}
0=u|_{\ga_{b}}&=u_{1}u_{2}u_{3}|_{\ga_{b}}
&
&\impl
&
u_{3}(0)&=0,\\
0=n\cdot\grad{}u|_{\ga_{b}}&=-u_{1}u_{2}u_{3}'|_{\ga_{b}}
&
&\impl
&
u_{3}'(0)&=0.
\end{align*}

$\bullet$
$\ga_{t}$,
$n=e^3$, $x_{3}=1$:
\begin{align*}
0=u|_{\ga_{t}}&=u_{1}u_{2}u_{3}|_{\ga_{t}}
&
&\impl
&
u_{3}(1)&=0,\\
0=n\cdot\grad{}u|_{\ga_{t}}&=u_{1}u_{2}u_{3}'|_{\ga_{t}}
&
&\impl
&
u_{3}'(1)&=0.
\end{align*}

The 1D case shows for the different boundary conditions the following:

$\bullet$ 
$\gat=\emptyset$ 
and $\gan=\ga$, i.e.,
$u_{1}'(0)=u_{1}'(1)=u_{2}'(0)=u_{2}'(1)=u_{3}'(0)=u_{3}'(1)=0$:
$\mu_{1}=n\pi$, $\mu_{2}=m\pi$, $\mu_{3}=k\pi$, i.e., 
$\lambda=\sqrt{n^2+m^2+k^2}\pi$, $n,m,k\in\mathbb{N}_{0}$, and 
$$\lambda_{0,\emptyset}=\pi,\qquad
u_{0,\emptyset}(x)=\alpha\cos(\pi x_{1})+\beta\cos(\pi x_{2})+\gamma\cos(\pi x_{3}).$$

Note that in this case the first eigenvalue is $\lambda=0$.

$\bullet$ 
$\gat=\ga_{b}$ 
and $\gan=\ga_{t,l,r,f,\bk}$, i.e.,
$u_{1}'(0)=u_{1}'(1)=u_{2}'(0)=u_{2}'(1)=u_{3}(0)=u_{3}'(1)=0$: 
$\mu_{1}=n\pi$, $\mu_{2}=m\pi$, $\mu_{3}=(k-1/2)\pi$, i.e., 
$\lambda=\sqrt{n^2+m^2+(k-1/2)^2}\pi$, $n,m\in\mathbb{N}_{0}$, $k\in\mathbb{N}$, and 
$$\lambda_{0,\ga_{b}}=\frac{1}{2}\pi,\qquad
u_{0,\ga_{b}}(x)=\alpha\sin(\frac{\pi}{2}x_{3}).$$

$\bullet$ 
$\gat=\ga_{b,t}$ 
and $\gan=\ga_{l,r,f,\bk}$, i.e.,
$u_{1}'(0)=u_{1}'(1)=u_{2}'(0)=u_{2}'(1)=u_{3}(0)=u_{3}(1)=0$: 
$\mu_{1}=n\pi$, $\mu_{2}=m\pi$, $\mu_{3}=k\pi$, i.e., 
$\lambda=\sqrt{n^2+m^2+k^2}\pi$, $n,m\in\mathbb{N}_{0}$, $k\in\mathbb{N}$, and 
$$\lambda_{0,\ga_{b,t}}=\pi,\qquad
u_{0,\ga_{b,t}}(x)=\alpha\sin(\pi x_{3}).$$

$\bullet$ 
$\gat=\ga_{b,l}$ 
and $\gan=\ga_{t,r,f,\bk}$, i.e.,
$u_{1}'(0)=u_{1}'(1)=u_{2}(0)=u_{2}'(1)=u_{3}(0)=u_{3}'(1)=0$:
$\mu_{1}=n\pi$, $\mu_{2}=(m-1/2)\pi$, $\mu_{3}=(k-1/2)\pi$, i.e.,
$\lambda=\sqrt{n^2+(m-1/2)^2+(k-1/2)^2}\pi$, $n\in\mathbb{N}_{0}$, $m,k\in\mathbb{N}$, and 
$$\lambda_{0,\ga_{b,l}}=\frac{\sqrt{2}}{2}\pi,\qquad
u_{0,\ga_{b,l}}(x)=\alpha\sin(\frac{\pi}{2}x_{2})\sin(\frac{\pi}{2}x_{3}).$$

$\bullet$ 
$\gat=\ga_{b,t,l}$ 
and $\gan=\ga_{r,f,\bk}$, i.e.,
$u_{1}'(0)=u_{1}'(1)=u_{2}(0)=u_{2}'(1)=u_{3}(0)=u_{3}(1)=0$:
$\mu_{1}=n\pi$, $\mu_{2}=(m-1/2)\pi$, $\mu_{3}=k\pi$, i.e.,
$\lambda=\sqrt{n^2+(m-1/2)^2+k^2}\pi$, $n\in\mathbb{N}_{0}$, $m,k\in\mathbb{N}$, and 
$$\lambda_{0,\ga_{b,t,l}}=\frac{\sqrt{5}}{2}\pi,\qquad
u_{0,\ga_{b,t,l}}(x)
=\alpha\sin(\frac{\pi}{2}x_{2})\sin(\pi x_{3}).$$

$\bullet$ 
$\gat=\ga_{b,l,\bk}$ 
and $\gan=\ga_{r,f,t}$, i.e.,
$u_{1}(0)=u_{1}'(1)=u_{2}(0)=u_{2}'(1)=u_{3}(0)=u_{3}'(1)=0$:
$\mu_{1}=(n-1/2)\pi$, $\mu_{2}=(m-1/2)\pi$, $\mu_{3}=(k-1/2)\pi$, i.e.,
$\lambda=\sqrt{(n-1/2)^2+(m-1/2)^2+(k-1/2)^2}\pi$, $n,m,k\in\mathbb{N}$, and 
$$\lambda_{0,\ga_{b,l,\bk}}=\frac{\sqrt{3}}{2}\pi,\qquad
u_{0,\ga_{b,l,\bk}}(x)
=\alpha\sin(\frac{\pi}{2}x_{1})\sin(\frac{\pi}{2}x_{2})\sin(\frac{\pi}{2}x_{3}).$$

$\bullet$ 
$\gat=\ga_{b,t,l,r}$ 
and $\gan=\ga_{f,\bk}$, i.e.,
$u_{1}'(0)=u_{1}'(1)=u_{2}(0)=u_{2}(1)=u_{3}(0)=u_{3}(1)=0$:
$\mu_{1}=n\pi$, $\mu_{2}=m\pi$, $\mu_{3}=k\pi$, i.e.,
$\lambda=\sqrt{n^2+m^2+k^2}\pi$, $n\in\mathbb{N}_{0}$, $m,k\in\mathbb{N}$, and 
$$\lambda_{0,\ga_{b,t,l,r}}=\sqrt{2}\pi,\qquad
u_{0,\ga_{b,t,l,r}}(x)
=\alpha\sin(\pi x_{2})\sin(\pi x_{3}).$$

$\bullet$ 
$\gat=\ga_{b,t,l,\bk}$ 
and $\gan=\ga_{f,r}$, i.e.,
$u_{1}(0)=u_{1}'(1)=u_{2}(0)=u_{2}'(1)=u_{3}(0)=u_{3}(1)=0$:
$\mu_{1}=(n-1/2)\pi$, $\mu_{2}=(m-1/2)\pi$, $\mu_{3}=k\pi$, i.e.,
$\lambda=\sqrt{(n-1/2)^2+(m-1/2)^2+k^2}\pi$, $n,m,k\in\mathbb{N}$, and 
$$\lambda_{0,\ga_{b,t,l,\bk}}=\frac{\sqrt{6}}{2}\pi,\qquad
u_{0,\ga_{b,t,l,\bk}}(x)
=\alpha\sin(\frac{\pi}{2}x_{1})\sin(\frac{\pi}{2}x_{2})\sin(\pi x_{3}).$$

$\bullet$ 
$\gat=\ga_{b,t,l,r,\bk}$ 
and $\gan=\ga_{f}$, i.e.,
$u_{1}(0)=u_{1}'(1)=u_{2}(0)=u_{2}(1)=u_{3}(0)=u_{3}(1)=0$:
$\mu_{1}=(n-1/2)\pi$, $\mu_{2}=m\pi$, $\mu_{3}=k\pi$, i.e.,
$\lambda=\sqrt{(n-1/2)^2+m^2+k^2}\pi$, $n,m,k\in\mathbb{N}$, and 
$$\lambda_{0,\ga_{b,t,l,r,\bk}}=\frac{3}{2}\pi,\qquad
u_{0,\ga_{b,t,l,r,\bk}}(x)
=\alpha\sin(\frac{\pi}{2}x_{1})\sin(\pi x_{2})\sin(\pi x_{3}).$$

$\bullet$ 
$\gat=\ga$ 
and $\gan=\emptyset$, i.e.,
$u_{1}(0)=u_{1}(1)=u_{2}(0)=u_{2}(1)=u_{3}(0)=u_{3}(1)=0$:
$\mu_{1}=n\pi$, $\mu_{2}=m\pi$, $\mu_{3}=k\pi$, i.e.,
$\lambda=\sqrt{n^2+m^2+k^2}\pi$, $n,m,k\in\mathbb{N}$, and 
$$\lambda_{0,\ga}=\sqrt{3}\pi,\qquad
u_{0,\ga}(x)
=\alpha\sin(\pi x_{1})\sin(\pi x_{2})\sin(\pi x_{3}).$$

All other cases follow by symmetry, i.e.,
\begin{align*}
\lambda_{0,\emptyset}
&=\pi,\\
\lambda_{0,\ga_{b}}
=\lambda_{0,\ga_{t}}
=\lambda_{0,\ga_{l}}
=\lambda_{0,\ga_{r}}
=\lambda_{0,\ga_{f}}
=\lambda_{0,\ga_{\bk}}
&=\frac{1}{2}\pi,\\
\lambda_{0,\ga_{b,t}}
=\lambda_{0,\ga_{l,r}}
=\lambda_{0,\ga_{f,\bk}}
&=\pi,\\
\lambda_{0,\ga_{b,l}}
=\lambda_{0,\ga_{b,r}}
=\lambda_{0,\ga_{b,f}}
=\lambda_{0,\ga_{b,\bk}}
\hspace*{59mm}&\\
=\lambda_{0,\ga_{t,l}}
=\lambda_{0,\ga_{t,r}}
=\lambda_{0,\ga_{t,f}}
=\lambda_{0,\ga_{t,\bk}}
=\lambda_{0,\ga_{f,l}}
=\lambda_{0,\ga_{f,r}}
=\lambda_{0,\ga_{\bk,l}}
=\lambda_{0,\ga_{\bk,r}}
&=\frac{\sqrt{2}}{2}\pi,\\
\lambda_{0,\ga_{b,t,l}}
=\lambda_{0,\ga_{b,t,r}}
=\lambda_{0,\ga_{b,t,f}}
=\lambda_{0,\ga_{b,t,\bk}}
=\lambda_{0,\ga_{l,r,b}}
\hspace*{38mm}&\\
=\lambda_{0,\ga_{l,r,t}}
=\lambda_{0,\ga_{l,r,f}}
=\lambda_{0,\ga_{l,r,\bk}}
=\lambda_{0,\ga_{f,\bk,l}}
=\lambda_{0,\ga_{f,\bk,r}}
=\lambda_{0,\ga_{f,\bk,b}}
=\lambda_{0,\ga_{f,\bk,t}}
&=\frac{\sqrt{5}}{2}\pi,\\
\lambda_{0,\ga_{b,\bk,l}}
=\lambda_{0,\ga_{b,l,f}}
=\lambda_{0,\ga_{b,f,r}}
=\lambda_{0,\ga_{b,r,\bk}}
=\lambda_{0,\ga_{t,\bk,l}}
=\lambda_{0,\ga_{t,l,f}}
=\lambda_{0,\ga_{t,f,r}}
=\lambda_{0,\ga_{t,r,\bk}}
&=\frac{\sqrt{3}}{2}\pi,\\
\lambda_{0,\ga_{b,t,l,r}}
=\lambda_{0,\ga_{b,t,f,\bk}}
=\lambda_{0,\ga_{l,r,f,\bk}}
&=\sqrt{2}\pi,\\
\lambda_{0,\ga_{b,t,l,\bk}}
=\lambda_{0,\ga_{b,t,f,l}}
=\lambda_{0,\ga_{b,t,r,f}}
=\lambda_{0,\ga_{b,t,r,\bk}}
=\lambda_{0,\ga_{l,r,f,t}}
=\lambda_{0,\ga_{l,r,f,b}}
\hspace*{6mm}&\\
=\lambda_{0,\ga_{l,r,t,\bk}}
=\lambda_{0,\ga_{l,r,b,\bk}}
=\lambda_{0,\ga_{f,\bk,b,l}}
=\lambda_{0,\ga_{f,\bk,t,l}}
=\lambda_{0,\ga_{f,\bk,b,r}}
=\lambda_{0,\ga_{f,\bk,r,r}}
&=\frac{\sqrt{6}}{2}\pi,\\
\lambda_{0,\ga_{b,t,l,r,\bk}}
=\lambda_{0,\ga_{b,t,l,r,f}}
=\lambda_{0,\ga_{b,t,l,f,\bk}}
=\lambda_{0,\ga_{b,t,r,f,\bk}}
=\lambda_{0,\ga_{b,l,r,f,\bk}}
=\lambda_{0,\ga_{t,l,r,f,\bk}}
&=\frac{3}{2}\pi,\\
\lambda_{0,\ga}
&=\sqrt{3}\pi.
\end{align*}

Now, we take care of $E$. As $\div{}E=0$ and 
$(-\Delta-\widetilde\lambda^2)E=0$, a simple ansatz is given by, e.g.,
$$E:=\rot{}U=\begin{bmatrix}\p_{2}u\\-\p_{1}u\\0\end{bmatrix},\quad
U(x):=u(x)\,e^3=u(x)\begin{bmatrix}0\\0\\1\end{bmatrix},$$ 
where $u$ is a solution of $(-\Delta-\widetilde\lambda^2)u=0$, i.e.,
$(-\Delta-\widetilde\lambda^2)U=0$.
Then $\div{}E=0$ and 
$$-\Delta E
=\rot{}\rot{}E
=\rot{}\rot{}\rot{}U
=-\rot{}\Delta U
=\widetilde\lambda^2\rot{}U
=\widetilde\lambda^2E.$$
As $u$ solves $(-\Delta-\widetilde\lambda^2)u=0$ we have again by separation of variables
$$u(x)=u_{1}(x_{1})u_{2}(x_{2})u_{3}(x_{3}),\quad
\grad{}u(x)
=\begin{bmatrix}u_{1}'(x_{1})u_{2}(x_{2})u_{3}(x_{3})\\
u_{1}(x_{1})u_{2}'(x_{2})u_{3}(x_{3})\\
u_{1}(x_{1})u_{2}(x_{2})u_{3}'(x_{3})\end{bmatrix}$$
as well as 
$$\widetilde\lambda^2=\mu_{1}^2+\mu_{2}^2+\mu_{3}^2$$
and
$$-u_{1}''(x_{1})-\mu_{1}^2u_{1}(x_{1})=0,\quad
-u_{2}''(x_{2})-\mu_{2}^2u_{2}(x_{2})=0,\quad
-u_{3}''(x_{3})-\mu_{3}^2u_{3}(x_{3})=0.$$
Moreover, by the complex property $R(\rot{\gan})\subset N(\div{\gan})$, $E$ must satisfy
\begin{align*}
E&\in D(\rot{\gat})\cap R(\rot{\gan})\subset D(\rot{\gat})\cap N(\div{\gan}),\\
\rot{}E&\in D(\rot{\gan})\cap R(\rot{\gat})\subset D(\rot{\gan})\cap N(\div{\gat}),
\end{align*}
i.e., in classical terms
$$n\times E|_{\gat}=0,\quad
n\cdot E|_{\gan}=0,\quad
n\times\rot{}E|_{\gan}=0,\quad
n\cdot\rot{}E|_{\gat}=0.$$
As the fourth boundary condition is implied by the first one
and the second boundary condition is implied by the third one
($n\times\rot{}E|_{\gan}=0$ $\impl$ 
$0=n\cdot\rot{}\rot{}E|_{\gan}=\widetilde\lambda^2n\cdot E|_{\gan}$),
the third and fourth ones are (almost) redundant, 
and we are (almost) left with the simple boundary conditions 
$$n\times E|_{\gat}=0,\quad
n\cdot E|_{\gan}=0,$$
except for some special cases, where also the third one
$$n\times\rot{}E|_{\gan}=0$$
is needed.
For the computations of the boundary conditions we note\footnote{Alternatively,
$\rot{}E
=\rot{}\rot{}U
=-\Delta U+\grad{}\div{}U
=-\Delta u\,e^3+\grad{}\p_{3}u
=\begin{bmatrix}\p_{1}\p_{3}u\\\p_{2}\p_{3}u\\-\p_{1}^2u-\p_{2}^2u\end{bmatrix}
=\begin{bmatrix}-\p_{3}E_{2}\\\p_{3}E_{1}\\\p_{1}E_{2}-\p_{2}E_{1}\end{bmatrix}$.} 
$$\rot{}E
=\begin{bmatrix}-\p_{3}E_{2}\\\p_{3}E_{1}\\\p_{1}E_{2}-\p_{2}E_{1}\end{bmatrix}
=\begin{bmatrix}\p_{1}\p_{3}u\\\p_{2}\p_{3}u\\-\p_{1}^2u-\p_{2}^2u\end{bmatrix},$$
and thus
\begin{align*}
E
&=\begin{bmatrix}E_{1}\\E_{2}\\0\end{bmatrix}
=\begin{bmatrix}\p_{2}u\\-\p_{1}u\\0\end{bmatrix},
&
e^1\times E
&=\begin{bmatrix}0\\0\\E_{2}\end{bmatrix},
&
e^2\times E
&=\begin{bmatrix}0\\0\\-E_{1}\end{bmatrix},
&
e^3\times E
&=\begin{bmatrix}-E_{2}\\E_{1}\\0\end{bmatrix},
\end{align*}
\begin{align*}
e^3\times\rot{}E
&=\begin{bmatrix}0\\0\\1\end{bmatrix}
\times\begin{bmatrix}-\p_{3}E_{2}\\\p_{3}E_{1}\\\p_{1}E_{2}-\p_{2}E_{1}\end{bmatrix}
=-\begin{bmatrix}\p_{3}E_{1}\\\p_{3}E_{2}\\0\end{bmatrix}
=\begin{bmatrix}-\p_{2}\p_{3}u\\\p_{1}\p_{3}u\\0\end{bmatrix}.
\end{align*}
As an alternative we can also set boundary conditions for $U$ directly. Since 
$$E=\rot{}U\in R(\rot{\gan}),$$
we get $n\times U|_{\gan}=u\,n\times e^3|_{\gan}=0$.

$\bullet$
$\ga_{\bk}$,
$n=-e^1$, $x_{1}=0$:
\begin{align*}
0=n\times E|_{\ga_{\bk}}&=-E_{2}e^3|_{\ga_{\bk}}
=\p_{1}u\,e^3|_{\ga_{\bk}}
=u_{1}'u_{2}u_{3}\,e^3|_{\ga_{\bk}}
&
&\impl
&
u_{1}'(0)&=0,\\
0=n\cdot E|_{\ga_{\bk}}&=-E_{1}|_{\ga_{\bk}}
=-\p_{2}u|_{\ga_{\bk}}
=-u_{1}u_{2}'u_{3}|_{\ga_{\bk}}
&
&\impl
&
u_{1}(0)&=0.
\intertext{Alternatively,}
0=u\,n\times e^3|_{\ga_{\bk}}&=u\,e^2|_{\ga_{\bk}}
=u_{1}u_{2}u_{3}\,e^2|_{\ga_{\bk}}
&
&\impl
&
u_{1}(0)&=0.
\end{align*}

$\bullet$
$\ga_{f}$,
$n=e^1$, $x_{1}=1$:
\begin{align*}
0=n\times E|_{\ga_{f}}&=E_{2}e^3|_{\ga_{f}}
=-\p_{1}u\,e^3|_{\ga_{f}}
=-u_{1}'u_{2}u_{3}\,e^3|_{\ga_{f}}
&
&\impl
&
u_{1}'(1)&=0,\\
0=n\cdot E|_{\ga_{f}}&=E_{1}|_{\ga_{f}}
=\p_{2}u|_{\ga_{f}}
=u_{1}u_{2}'u_{3}|_{\ga_{f}}
&
&\impl
&
u_{1}(1)&=0.
\intertext{Alternatively,}
0=u\,n\times e^3|_{\ga_{f}}&=-u\,e^2|_{\ga_{f}}
=-u_{1}u_{2}u_{3}\,e^2|_{\ga_{f}}
&
&\impl
&
u_{1}(1)&=0.
\end{align*}

$\bullet$
$\ga_{l}$,
$n=-e^2$, $x_{2}=0$:
\begin{align*}
0=n\times E|_{\ga_{l}}&=E_{1}e^3|_{\ga_{l}}
=\p_{2}u\,e^3|_{\ga_{l}}
=u_{1}u_{2}'u_{3}u\,e^3|_{\ga_{l}}
&
&\impl
&
u_{2}'(0)&=0,\\
0=n\cdot E|_{\ga_{l}}&=-E_{2}|_{\ga_{l}}
=\p_{1}u|_{\ga_{l}}
=u_{1}'u_{2}u_{3}|_{\ga_{l}}
&
&\impl
&
u_{2}(0)&=0.
\intertext{Alternatively,}
0=u\,n\times e^3|_{\ga_{l}}&=-u\,e^1|_{\ga_{l}}
=-u_{1}u_{2}u_{3}\,e^1|_{\ga_{l}}
&
&\impl
&
u_{2}(0)&=0.
\end{align*}

$\bullet$
$\ga_{r}$,
$n=e^2$, $x_{2}=1$:
\begin{align*}
0=n\times E|_{\ga_{r}}&=-E_{1}e^3|_{\ga_{r}}
=-\p_{2}u\,e^3|_{\ga_{r}}
=-u_{1}u_{2}'u_{3}\,e^3|_{\ga_{r}}
&
&\impl
&
u_{2}'(1)&=0,\\
0=n\cdot E|_{\ga_{r}}&=E_{2}|_{\ga_{r}}
=-\p_{1}u|_{\ga_{r}}
=-u_{1}'u_{2}u_{3}|_{\ga_{r}}
&
&\impl
&
u_{2}(1)&=0.
\intertext{Alternatively,}
0=u\,n\times e^3|_{\ga_{r}}&=u\,e^1|_{\ga_{r}}
=u_{1}u_{2}u_{3}\,e^1|_{\ga_{r}}
&
&\impl
&
u_{2}(1)&=0.
\end{align*}

$\bullet$
$\ga_{b}$,
$n=-e^3$, $x_{3}=0$:
\begin{align*}
0=n\times E|_{\ga_{b}}
&=\begin{bmatrix}E_{2}\\-E_{1}\\0\end{bmatrix}|_{\ga_{b}}
=-\begin{bmatrix}\p_{1}u\\\p_{2}u\\0\end{bmatrix}|_{\ga_{b}}
=-\begin{bmatrix}u_{1}'u_{2}u_{3}\\u_{1}u_{2}'u_{3}\\0\end{bmatrix}|_{\ga_{b}}
&
&\impl
&
u_{3}(0)&=0,\\
0=n\cdot E|_{\ga_{b}}&=0\quad(\text{no condition}),\\
0=n\times\rot{}E|_{\ga_{b}}
&=\begin{bmatrix}\p_{3}E_{1}\\\p_{3}E_{2}\\0\end{bmatrix}|_{\ga_{b}}
=\begin{bmatrix}\p_{2}\p_{3}u\\-\p_{1}\p_{3}u\\0\end{bmatrix}|_{\ga_{b}}
=\begin{bmatrix}u_{1}u_{2}'u_{3}'\\-u_{1}'u_{2}u_{3}'\\0\end{bmatrix}|_{\ga_{b}}
&
&\impl
&
u_{3}'(0)&=0.
\intertext{Alternatively,}
0=u\,n\times e^3|_{\ga_{b}}&=0\quad(\text{no condition}).
\end{align*}

$\bullet$
$\ga_{t}$,
$n=e^3$, $x_{3}=1$:
\begin{align*}
0=n\times E|_{\ga_{t}}
&=\begin{bmatrix}-E_{2}\\E_{1}\\0\end{bmatrix}|_{\ga_{t}}
=\begin{bmatrix}\p_{1}u\\\p_{2}u\\0\end{bmatrix}|_{\ga_{t}}
=\begin{bmatrix}u_{1}'u_{2}u_{3}\\u_{1}u_{2}'u_{3}\\0\end{bmatrix}|_{\ga_{t}}
&
&\impl
&
u_{3}(1)&=0,\\
0=n\cdot E|_{\ga_{t}}&=0\quad(\text{no condition}),\\
0=n\times\rot{}E|_{\ga_{t}}
&=-\begin{bmatrix}\p_{3}E_{1}\\\p_{3}E_{2}\\0\end{bmatrix}|_{\ga_{t}}
=\begin{bmatrix}-\p_{2}\p_{3}u\\\p_{1}\p_{3}u\\0\end{bmatrix}|_{\ga_{t}}
=\begin{bmatrix}-u_{1}u_{2}'u_{3}'\\u_{1}'u_{2}u_{3}'\\0\end{bmatrix}|_{\ga_{t}}
&
&\impl
&
u_{3}'(1)&=0.
\intertext{Alternatively,}
0=u\,n\times e^3|_{\ga_{t}}&=0\quad(\text{no condition}).
\end{align*}

By construction, i.e.,
$$E=\begin{bmatrix}\p_{2}u\\-\p_{1}u\\0\end{bmatrix},$$
$u$ can be constant in one variable $x_{n}$
and simultaneously in two variables $x_{1}$, $x_{3}$ and $x_{2}$, $x_{3}$,
respectively, but not simultaneously in the two variables $x_{1}$, $x_{2}$
since this implies $E=0$.
The 1D case shows for the different boundary conditions the following:

$\bullet$ 
$\gat=\emptyset$ 
and $\gan=\ga$, i.e.,
$u_{1}(0)=u_{1}(1)=u_{2}(0)=u_{2}(1)=u_{3}'(0)=u_{3}'(1)=0$:
$\mu_{1}=n\pi$, $\mu_{2}=m\pi$, $\mu_{3}=k\pi$, i.e., 
$\widetilde\lambda=\sqrt{n^2+m^2+k^2}\pi$, $n,m\in\mathbb{N}$, $k\in\mathbb{N}_{0}$, and 
$$\lambda_{1,\emptyset}=\sqrt{2}\pi,\qquad
u_{1,\emptyset}(x)
=\alpha\sin(\pi x_{1})\sin(\pi x_{2}),\qquad
E_{1,\emptyset}(x)
=\alpha\pi\begin{bmatrix}\sin(\pi x_{1})\cos(\pi x_{2})\\-\cos(\pi x_{1})\sin(\pi x_{2})\\0\end{bmatrix}.$$
Note that we already know from the theory that 
$$\lambda_{1,\ga}=\lambda_{1,\emptyset}=\sqrt{2}\pi.$$
In the particular computation we get 
$u_{1}'(0)=u_{1}'(1)=u_{2}'(0)=u_{2}'(1)=u_{3}(0)=u_{3}(1)=0$
for $\gat=\ga$ and $\gan=\emptyset$:
$\mu_{1}=n\pi$, $\mu_{2}=m\pi$, $\mu_{3}=k\pi$, i.e., 
$\widetilde\lambda=\sqrt{n^2+m^2+k^2}\pi$, $n,k\in\mathbb{N}$, $m\in\mathbb{N}_{0}$,
or $m,k\in\mathbb{N}$, $n\in\mathbb{N}_{0}$.
We emphasise that here the actual case $n=m=0$ is not allowed
as this would imply $E=0$, see our discussion above.
The eigenvectors are
\begin{align*}
u_{1,\ga}(x)
&=\alpha\cos(\pi x_{1})\sin(\pi x_{3})
+\beta\cos(\pi x_{2})\sin(\pi x_{3}),\\
E_{1,\ga}(x)
&=\alpha\pi\sin(\pi x_{1})\sin(\pi x_{3})\begin{bmatrix}0\\1\\0\end{bmatrix}
-\beta\pi\sin(\pi x_{2})\sin(\pi x_{3})\begin{bmatrix}1\\0\\0\end{bmatrix}.
\end{align*}

$\bullet$ 
$\gat=\ga_{b}$ and
$\gan=\ga_{t,l,r,f,\bk}$, i.e.,
$u_{1}(0)=u_{1}(1)=u_{2}(0)=u_{2}(1)=u_{3}(0)=u_{3}'(1)=0$:
$\mu_{1}=n\pi$, $\mu_{2}=m\pi$, $\mu_{3}=(k-1/2)\pi$, i.e., 
$\widetilde\lambda=\sqrt{n^2+m^2+(k-1/2)^2}\pi$, $n,m,k\in\mathbb{N}$, and 
the minimum and eigenvectors are
$$\widetilde\lambda=\frac{3}{2}\pi,\qquad
u(x)
=\alpha\sin(\pi x_{1})\sin(\pi x_{2})\sin(\frac{\pi}{2}x_{3}),\qquad
E(x)
=\alpha\pi\sin(\frac{\pi}{2}x_{3})\begin{bmatrix}\sin(\pi x_{1})\cos(\pi x_{2})\\
-\cos(\pi x_{1})\sin(\pi x_{2})\\0\end{bmatrix}.$$
If $\gat=\ga_{l}$ 
and $\gan=\ga_{t,b,r,f,\bk}$, i.e.,
$u_{1}(0)=u_{1}(1)=u_{2}'(0)=u_{2}(1)=u_{3}'(0)=u_{3}'(1)=0$:
$\mu_{1}=n\pi$, $\mu_{2}=(m-1/2)\pi$, $\mu_{3}=k\pi$, i.e., 
$\widetilde\lambda=\sqrt{n^2+(m-1/2)^2+k^2}\pi$, $n,m\in\mathbb{N}$, $k\in\mathbb{N}_{0}$, and 
the minimum and eigenvectors are
$$\lambda_{1,\ga_{l}}=\frac{\sqrt{5}}{2}\pi,\qquad
u_{1,\ga_{l}}(x)
=\alpha\sin(\pi x_{1})\cos(\frac{\pi}{2}x_{2}),\qquad
E_{1,\ga_{l}}(x)
=\alpha\frac{\pi}{2}\begin{bmatrix}-\sin(\pi x_{1})\sin(\frac{\pi}{2}x_{2})\\
2\cos(\pi x_{1})\cos(\frac{\pi}{2}x_{2})\\0\end{bmatrix}.$$
This shows that by replacing the ansatz for $E$ by, e.g.,
$$E:=\rot{}U=\begin{bmatrix}0\\\p_{3}u\\-\p_{2}u\end{bmatrix},\quad
U(x):=u(x)\,e^1=u(x)\begin{bmatrix}1\\0\\0\end{bmatrix},$$ 
we get also the smaller eigenvalue $\widetilde\lambda=(\sqrt{5}/2)\pi$ in the case $\gat=\ga_{b}$.
Hence, by symmetry
$$\lambda_{1,\ga_{b}}
=\lambda_{1,\ga_{t}}
=\lambda_{1,\ga_{l}}
=\lambda_{1,\ga_{r}}
=\lambda_{1,\ga_{f}}
=\lambda_{1,\ga_{\bk}}
=\frac{\sqrt{5}}{2}\pi.$$

$\bullet$ 
$\gat=\ga_{b,t}$ 
and $\gan=\ga_{l,r,f,\bk}$, i.e.,
$u_{1}(0)=u_{1}(1)=u_{2}(0)=u_{2}(1)=u_{3}(0)=u_{3}(1)=0$: 
$\mu_{1}=n\pi$, $\mu_{2}=m\pi$, $\mu_{3}=k\pi$, i.e., 
$\widetilde\lambda=\sqrt{n^2+m^2+k^2}\pi$, $n,m,k\in\mathbb{N}$, and 
the minimum and eigenvectors are
$$\widetilde\lambda=\sqrt{3}\pi,\qquad
u(x)
=\alpha\sin(\pi x_{1})\sin(\pi x_{2})\sin(\pi x_{3}),\qquad
E(x)
=\alpha\pi\sin(\pi x_{3})\begin{bmatrix}\sin(\pi x_{1})\cos(\pi x_{2})\\
-\cos(\pi x_{1})\sin(\pi x_{2})\\0\end{bmatrix}.$$
If $\gat=\ga_{l,r}$ 
and $\gan=\ga_{b,t,f,\bk}$, i.e.,
$u_{1}(0)=u_{1}(1)=u_{2}'(0)=u_{2}'(1)=u_{3}'(0)=u_{3}'(1)=0$: 
$\mu_{1}=n\pi$, $\mu_{2}=m\pi$, $\mu_{3}=k\pi$, i.e., 
$\widetilde\lambda=\sqrt{n^2+m^2+k^2}\pi$, $n\in\mathbb{N}$, $m,k\in\mathbb{N}_{0}$, and 
the minimum and eigenvectors are
$$\lambda_{1,\ga_{l,r}}=\pi,\qquad
u_{1,\ga_{l,r}}(x)
=\alpha\sin(\pi x_{1}),\qquad
E_{1,\ga_{l,r}}(x)
=-\alpha\pi\cos(\pi x_{1})\begin{bmatrix}0\\1\\0\end{bmatrix}.$$
Hence, again by changing the ansatz,
we get also the smaller eigenvalue $\widetilde\lambda=\pi$ 
in the case $\gat=\ga_{b,t}$.
Thus, by symmetry
$$\lambda_{1,\ga_{l,r}}=\lambda_{1,\ga_{b,t}}=\lambda_{1,\ga_{f,\bk}}=\pi.$$

$\bullet$ 
$\gat=\ga_{b,l}$ 
and $\gan=\ga_{t,r,f,\bk}$, i.e.,
$u_{1}(0)=u_{1}(1)=u_{2}'(0)=u_{2}(1)=u_{3}(0)=u_{3}'(1)=0$: 
$\mu_{1}=n\pi$, $\mu_{2}=(m-1/2)\pi$, $\mu_{3}=(k-1/2)\pi$, i.e., 
$\widetilde\lambda=\sqrt{n^2+(m-1/2)^2+(k-1/2)^2}\pi$ with $n,m,k\in\mathbb{N}$, and 
the minimum and eigenvectors are
$$\widetilde\lambda=\frac{\sqrt{6}}{2}\pi,\quad
u(x)
=\alpha\sin(\pi x_{1})\cos(\frac{\pi}{2}x_{2})\sin(\frac{\pi}{2}x_{3}),\quad
E(x)
=-\alpha\frac{\pi}{2}\sin(\frac{\pi}{2}x_{3})\begin{bmatrix}\sin(\pi x_{1})\sin(\frac{\pi}{2}x_{2})\\
2\cos(\pi x_{1})\cos(\frac{\pi}{2}x_{2})\\0\end{bmatrix}.$$
If $\gat=\ga_{f,l}$ 
and $\gan=\ga_{b,t,r,\bk}$, i.e.,
$u_{1}(0)=u_{1}'(1)=u_{2}'(0)=u_{2}(1)=u_{3}'(0)=u_{3}'(1)=0$: 
$\mu_{1}=(n-1/2)\pi$, $\mu_{2}=(m-1/2)\pi$, $\mu_{3}=k\pi$, i.e., 
$\widetilde\lambda=\sqrt{(n-1/2)^2+(m-1/2)^2+k^2}\pi$, $n,m\in\mathbb{N}$, $k\in\mathbb{N}_{0}$, and 
the minimum and eigenvectors are
\begin{align*}
\lambda_{1,\ga_{f,l}}
&=\frac{\sqrt{2}}{2}\pi,
&
u_{1,\ga_{f,l}}(x)
&=\alpha\sin(\frac{\pi}{2}x_{1})\cos(\frac{\pi}{2}x_{2}),
&
E_{1,\ga_{f,l}}(x)
&=-\alpha\frac{\pi}{2}\begin{bmatrix}\sin(\frac{\pi}{2}x_{1})\sin(\frac{\pi}{2}x_{2})\\
\cos(\frac{\pi}{2}x_{1})\cos(\frac{\pi}{2}x_{2})\\0\end{bmatrix}.
\end{align*}
Hence, again by changing the ansatz,
we obtain also the smaller eigenvalue $\widetilde\lambda=(\sqrt{2}/2)\pi$ 
in the case $\gat=\ga_{b,l}$.
Thus, by symmetry
\begin{align*}
&\hspace*{5mm}\lambda_{1,\ga_{b,l}}
=\lambda_{1,\ga_{b,r}}
=\lambda_{1,\ga_{b,f}}=\lambda_{1,\ga_{b,\bk}}\\
&=\lambda_{1,\ga_{t,l}}=\lambda_{1,\ga_{t,r}}
=\lambda_{1,\ga_{t,f}}=\lambda_{1,\ga_{t,\bk}}
=\lambda_{1,\ga_{f,l}}=\lambda_{1,\ga_{l,\bk}}
=\lambda_{1,\ga_{\bk,r}}=\lambda_{1,\ga_{f,r}}
=\frac{\sqrt{2}}{2}\pi.
\end{align*}

$\bullet$ 
$\gat=\ga_{b,l,t}$ 
and $\gan=\ga_{r,f,\bk}$, i.e.,
$u_{1}(0)=u_{1}(1)=u_{2}'(0)=u_{2}(1)=u_{3}(0)=u_{3}(1)=0$: 
$\mu_{1}=n\pi$, $\mu_{2}=(m-1/2)\pi$, $\mu_{3}=k\pi$, i.e., 
$\widetilde\lambda=\sqrt{n^2+(m-1/2)^2+k^2}\pi$, $n,m,k\in\mathbb{N}$, and 
the minimum and eigenvectors are
$$\widetilde\lambda=\frac{3}{2}\pi,\quad
u(x)
=\alpha\sin(\pi x_{1})\cos(\frac{\pi}{2}x_{2})\sin(\pi x_{3}),\quad
E(x)
=-\alpha\frac{\pi}{2}\sin(\pi x_{3})\begin{bmatrix}\sin(\pi x_{1})\sin(\frac{\pi}{2}x_{2})\\
2\cos(\pi x_{1})\cos(\frac{\pi}{2}x_{2})\\0\end{bmatrix}.$$
If $\gat=\ga_{f,l,\bk}$ 
and $\gan=\ga_{t,r,b}$, i.e.,
$u_{1}'(0)=u_{1}'(1)=u_{2}'(0)=u_{2}(1)=u_{3}'(0)=u_{3}'(1)=0$: 
$\mu_{1}=n\pi$, $\mu_{2}=(m-1/2)\pi$, $\mu_{3}=k\pi$, i.e., 
$\widetilde\lambda=\sqrt{n^2+(m-1/2)^2+k^2}\pi$, $n,k\in\mathbb{N}_{0}$, $m\in\mathbb{N}$, and 
the minimum and eigenvectors are
\begin{align*}
\lambda_{1,\ga_{f,l,\bk}}
&=\frac{1}{2}\pi,
&
u_{1,\ga_{f,l,\bk}}(x)
&=\alpha\cos(\frac{\pi}{2}x_{2}),
&
E_{1,\ga_{f,l,\bk}}(x)
&=-\alpha\frac{\pi}{2}\sin(\frac{\pi}{2}x_{2})\begin{bmatrix}1\\0\\0\end{bmatrix}.
\end{align*}
Hence, again by changing the ansatz,
we obtain also the smaller eigenvalue $\widetilde\lambda=(1/2)\pi$ in the case 
$\gat=\ga_{b,l,t}$.
Thus, by symmetry
\begin{align*}
&\hspace*{5mm}\lambda_{1,\ga_{b,l,t}}=\lambda_{1,\ga_{b,r,t}}
=\lambda_{1,\ga_{b,f,t}}=\lambda_{1,\ga_{b,\bk,t}}
=\lambda_{1,\ga_{r,l,t}}=\lambda_{1,\ga_{r,l,b}}\\
&=\lambda_{1,\ga_{r,l,f}}=\lambda_{1,\ga_{r,l,\bk}}
=\lambda_{1,\ga_{f,\bk,l}}=\lambda_{1,\ga_{f,\bk,r}}
=\lambda_{1,\ga_{f,\bk,t}}=\lambda_{1,\ga_{f,\bk,b}}
=\frac{1}{2}\pi.
\end{align*}

$\bullet$ 
$\gat=\ga_{b,l,\bk}$ 
and $\gan=\ga_{t,r,f}$, i.e.,
$u_{1}'(0)=u_{1}(1)=u_{2}'(0)=u_{2}(1)=u_{3}(0)=u_{3}'(1)=0$: 
$\mu_{1}=(n-1/2)\pi$, $\mu_{2}=(m-1/2)\pi$, $\mu_{3}=(k-1/2)\pi$, i.e., 
$\widetilde\lambda=\sqrt{(n-1/2)^2+(m-1/2)^2+(k-1/2)^2}\pi$, $n,m,k\in\mathbb{N}$, and 
the minimum and eigenvectors are
\begin{align*}
\lambda_{1,\ga_{b,l,\bk}}
&=\frac{\sqrt{3}}{2}\pi,\\
u_{1,\ga_{b,l,\bk}}(x)
&=\alpha\cos(\frac{\pi}{2}x_{1})\cos(\frac{\pi}{2}x_{2})\sin(\frac{\pi}{2}x_{3}),
&
E_{1,\ga_{b,l,\bk}}(x)
&=\alpha\frac{\pi}{2}\sin(\frac{\pi}{2}x_{3})\begin{bmatrix}-\cos(\frac{\pi}{2}x_{1})\sin(\frac{\pi}{2}x_{2})\\
\sin(\frac{\pi}{2}x_{1})\cos(\frac{\pi}{2}x_{2})\\0\end{bmatrix}.
\end{align*}
By symmetry
\begin{align*}
\lambda_{1,\ga_{b,l,\bk}}=\lambda_{1,\ga_{b,r,\bk}}
=\lambda_{1,\ga_{b,l,f}}=\lambda_{1,\ga_{b,r,f}}
=\lambda_{1,\ga_{t,l,\bk}}=\lambda_{1,\ga_{t,r,\bk}}
=\lambda_{1,\ga_{t,l,f}}=\lambda_{1,\ga_{t,r,f}}
=\frac{\sqrt{3}}{2}\pi.
\end{align*}

We summarise
\begin{align*}
\lambda_{1,\emptyset}
=\lambda_{1,\ga}
&=\sqrt{2}\pi,\\
\lambda_{1,\ga_{b}}
=\lambda_{1,\ga_{t}}
=\lambda_{1,\ga_{l}}
=\lambda_{1,\ga_{r}}
=\lambda_{1,\ga_{f}}
=\lambda_{1,\ga_{\bk}}
&=\frac{\sqrt{5}}{2}\pi,\\
\lambda_{1,\ga_{l,r}}=\lambda_{1,\ga_{b,t}}=\lambda_{1,\ga_{f,\bk}}
&=\pi,\\
\lambda_{1,\ga_{b,l}}
=\lambda_{1,\ga_{b,r}}
=\lambda_{1,\ga_{b,f}}=\lambda_{1,\ga_{b,\bk}}
\hspace*{59mm}&\\
=\lambda_{1,\ga_{t,l}}=\lambda_{1,\ga_{t,r}}
=\lambda_{1,\ga_{t,f}}=\lambda_{1,\ga_{t,\bk}}
=\lambda_{1,\ga_{f,l}}=\lambda_{1,\ga_{l,\bk}}
=\lambda_{1,\ga_{\bk,r}}=\lambda_{1,\ga_{f,r}}
&=\frac{\sqrt{2}}{2}\pi,\\
\lambda_{1,\ga_{b,l,t}}=\lambda_{1,\ga_{b,r,t}}
=\lambda_{1,\ga_{b,f,t}}=\lambda_{1,\ga_{b,\bk,t}}
=\lambda_{1,\ga_{r,l,t}}
\hspace*{38mm}&\\
=\lambda_{1,\ga_{r,l,b}}
=\lambda_{1,\ga_{r,l,f}}=\lambda_{1,\ga_{r,l,\bk}}
=\lambda_{1,\ga_{f,\bk,l}}=\lambda_{1,\ga_{f,\bk,r}}
=\lambda_{1,\ga_{f,\bk,t}}=\lambda_{1,\ga_{f,\bk,b}}
&=\frac{1}{2}\pi,\\
\lambda_{1,\ga_{b,l,\bk}}=\lambda_{1,\ga_{b,r,\bk}}
=\lambda_{1,\ga_{b,l,f}}=\lambda_{1,\ga_{b,r,f}}
=\lambda_{1,\ga_{t,l,\bk}}=\lambda_{1,\ga_{t,r,\bk}}
=\lambda_{1,\ga_{t,l,f}}=\lambda_{1,\ga_{t,r,f}}
&=\frac{\sqrt{3}}{2}\pi,
\end{align*}
and all other cases follow by $\lambda_{1,\gan}=\lambda_{1,\gat}$ as well as symmetry.


\vspace*{5mm}
\hrule
\vspace*{3mm}


\end{document}